\documentclass[reqno]{amsart}
\usepackage{xifthen,xparse,ifmtarg,suffix,xstring,iftex}
\usepackage{stackrel}
\usepackage{amsmath,amsthm,amssymb}
\usepackage[disallowspaces,fixamsmath]{mathtools}
\usepackage{todonotes}
\usepackage{parskip}
\usepackage{enumitem}
\usepackage{ifmtarg}
\usepackage{etoolbox}
\usepackage{xr}

    \usepackage{mathrsfs} 
    \DeclareMathAlphabet{\mathbrush}{T1}{pbsi}{xl}{n}

\usepackage[style=numeric,
            backend=bibtex,
            doi=true,
            sorting=nyt,
            clearlang=true,
            hyperref=true,
            arxiv=abs
            ]{biblatex}
\DeclareFieldFormat{url}{\href{#1}{\mkbibacro{URL}}} %Want URLs to be there but not be printed
\renewbibmacro*{doi+eprint+url}{% Only print url if we have neither a doi or an eprint
  \printfield{doi}%
  \newunit\newblock%
  \iftoggle{bbx:eprint}
    {\usebibmacro{eprint}}
    {}%
  \newunit\newblock%
  \iffieldundef{doi}
    {\usebibmacro{url+urldate}}%
    {}}
\AtEveryBibitem{
  \clearlist{language}
}
\usepackage[unicode,psdextra]{hyperref}
\usepackage{hyperxmp}
\hypersetup{pdfauthor={Peter M. Gerdes},
            pdfsubject={Computability Theory, Turing Reducibility and r.e. sets},
            pdfcontactemail={gerdes@invariant.org},
            pdfcontacturl={http://invariant.org},
            pdflang={en},
            pdfcopyright={Copyright (C) \today, Peter M. Gerdes},
            baseurl={http://invariant.org},
            unicode=true,
            pdfdisplaydoctitle=true,
            pdfencoding=unicode,
            psdextra=true,
            pdfcontactaddress={Department of Mathematics,
                                \xmpquote{Indiana University\xmpcomma Bloomington\xmpcomma IN USA}}
}

\usepackage[steps,coarse,lstrdelim=strleft,rstrdelim=strright]{rec-thy}

\theoremstyle{plain}% default
    \newtheorem{theorem}{Theorem}
    
    \newtheorem{lemma}{Lemma}[section]
    \newtheorem{proposition}[lemma]{Proposition}
\theoremstyle{definition}
    \newtheorem{definition}[lemma]{Definition}
    \newtheorem{question}{Question}
    
\theoremstyle{remark}
    \newtheorem{observation}[lemma]{Observation}
\theoremstyle{plain}% default

\makeatletter

\DeclareMathOperator{\id}{id}
\DeclareMathOperator{\sdom}{sdom}
\NewDocumentCommand{\hgtenc}{m}{\abs{#1}^{\dagger}}
\newcommand*{\modr}[1][]{\symbf{\mathfrak{m}}_{#1}}
\newcommand*{\modrinv}[1][]{\symbf{\mathfrak{m}}^{-1}_{#1}}
\NewDocumentCommand{\encstr}{O{}O{}}{\symbf{\mathfrak{e}}^{#2}_{#1}}
\newcommand*{\encinv}{\symbf{\mathfrak{e}}^{-1}}
\let\card=\hgt
\def\mydash{---~}
\NewDocumentCommand{\agree}{sO{}}{\mathrel{\cong_{#2}\IfBooleanT{#1}{^*}}} 
\NewDocumentCommand{\sagree}{sot0}{\mathrel{\cong^{\IfBooleanTF{#1}{*\IfValueT{#2}{, #2}\IfBooleanT{#3}{, 0} }{\IfValueT{#2}{#2}\IfBooleanT{#3}{0}}}_{\Box}}}
\NewDocumentCommand{\nsagree}{sot0}{\mathrel{\ncong^{\IfBooleanTF{#1}{*\IfValueT{#2}{, #2}\IfBooleanT{#3}{, 0} }{\IfValueT{#2}{#2}\IfBooleanT{#3}{0}}}_{\Box}}}
\NewDocumentCommand{\boxeq}{}{\mathrel{=_{\Box}}}
\NewDocumentCommand{\nboxeq}{}{\mathrel{\neq_{\Box}}}
\NewDocumentCommand{\ED}{O{f, S}D(){\e,\i}}{\Psi_{#2}\left(#1\right)}
\newcommand*{\ddist}{\symbf{\mathfrak{d}}}
\NewDocumentCommand{\Icond}{so}{P_\mathbb{I}\IfBooleanT{#1}{^{*}}\IfValueT{#2}{^{*}\!\left(#2\right)}}
\NewDocumentCommand{\Igeq}{o}{\mathrel{\leq_{\mathbb{I}}\IfValueT{#1}{^{#1}}}}
\NewDocumentCommand{\Ileq}{o}{\mathrel{\geq_{\mathbb{I}}\IfValueT{#1}{^{#1}}}}
\NewDocumentCommand{\nIgeq}{o}{\mathrel{\not\leq_{\mathbb{I}}\IfValueT{#1}{^{#1}}}}
\NewDocumentCommand{\nIleq}{o}{\mathrel{\not\geq_{\mathbb{I}}\IfValueT{#1}{^{#1}}}}
\NewDocumentCommand{\Pcond}{}{\mathbb{P}}
\NewDocumentCommand{\Pgeq}{o}{\mathrel{\leq_{\Pcond}\IfValueT{#1}{^{#1}}}}
\NewDocumentCommand{\nPgeq}{o}{\mathrel{\not\leq_{\Pcond}\IfValueT{#1}{^{#1}}}}
\NewDocumentCommand{\Pleq}{o}{\mathrel{\geq_{\Pcond}\IfValueT{#1}{^{#1}}}}
\NewDocumentCommand{\Qcond}{}{\mathbb{Q}}
\NewDocumentCommand{\nQgeq}{o}{\mathrel{\not\leq_{\Qcond}\IfValueT{#1}{^{#1}}}}
\NewDocumentCommand{\nQleq}{o}{\mathrel{\not\geq_{\Qcond}\IfValueT{#1}{^{#1}}}}
\NewDocumentCommand{\Qgeq}{o}{\mathrel{\leq_{\Qcond}\IfValueT{#1}{^{#1}}}}
\NewDocumentCommand{\Qleq}{o}{\mathrel{\geq_{\Qcond}\IfValueT{#1}{^{#1}}}}

\NewDocumentCommand{\deninv}{mo}{\lambda\@ifmtarg{#1}{}{\left(#1\IfValueT{#2}{; #2}\right)}}

\NewDocumentCommand{\smarthat}{ O{1mu} m }{%
    \mathrlap{\widehat{#2}}\phantom{\widehat{#2}}\mkern3mu% invisible width + tiny extra
}
\let\hat=\smarthat
\let\wwedge=\doublewedge

\makeatother
\begin{document}

\title{Comparing Notions of Dense Computability On \( \baire \) and \( \cantor \)}

\author{Peter M. Gerdes}
\address{Indiana University, Bloomington}
\email{gerdes@invariant.org}

% \keywords{}
% \thanks{}

\date{}
% \dedicatory{}
\begin{abstract}
A relatively new topic in computability theory is the study of notions of computation that are robust against mistakes on some kind of small set.  However, despite the recent popularity of this topic relatively foundational questions about the notions of reducibility involved still persist.  In this paper, we examine two notions of robust information coding, effective dense reducibility and coarse reducibility and answer the question posed in  \cite{Astor2019Dense}: whether the degrees of functions under these reductions are the same as the degrees of sets.  Despite the surface similarity of these two reducibilities we show that every uniform coarse degree contains a set but that this fails even for the non-uniform effective dense degrees.  We then further distinguish these two notions by showing that whether \( g \) is coarsely reducible to \( f \) is an arithmetic property of \( f \) and \( g \) while for non-uniform effective dense reducibility it is a \( \piin{1} \) complete property.  To prove these results we introduce notions of forcing that allow us to build generic effective dense and coarse descriptions which may be of use in further exploration of these topics \mydash including the open questions we pose in the final section.
\end{abstract}

\maketitle
% \todo{Redefine card to look different than lh }                                                         
                     
\section{Introduction}

% While the question of when a set can be computably approximated  has been studied for some time, there has been a resurgence of interest in the 

% computing a function correctly on `most' values 

% Recently, there has been significant interest in notions of computation that are 
% The notion of coarse computability introduced in \cite{Dzhafarov2017Notions} 

Notions of computability that are robust against `mistakes' on some kind of small set  have been studied in a number of recent papers.  Many of these notions make equal sense when applied to functions as when applied to sets.  While recent interest in the topic of approximating a set on most values (in the sense of density \( 1 \)) can be largely traced back to \cite{Jr2012Generic}; in \cite{Dzhafarov2017Notions} the focus was shifted to the broader context of trying to understand when being mostly right about one object allows you to be mostly right about another and the concept of (uniform and non-uniform) coarse degrees (implicit in \cite{Jr2012Generic} and also \cite{Terwijn1998Computability}) was explicitly defined.  This paradigm was further developed in  \cite{Astor2019Dense} who defined (uniform and non-uniform) effective dense degrees.  However, despite these notions drawing significant interest in recent years a number of relatively basic questions remain unanswered. 

For  many of these notions it is unknown whether considering functions as well as sets expands the class of degrees, i.e., whether or not every equivalence class contains a set.  This paper answers that question, posed as question 7.2 in \cite{Astor2019Dense}, for the uniform and non-uniform versions of both coarse and effective dense reducibilities.  We also prove results characterizing the complexity of coarse and effective dense reducibilities.  In both cases we show that, despite the surface similarity of the two notions, coarse and effective dense reductions have extremely different behavior.    

We briefly preview the theorems we will prove with \( \NCleq, \UCleq, \NEDleq, \UEDleq \) denoting  coarse and effective dense reducibility respectively (non-uniform versions listed first).

\newtheorem*{thm:coarse:repeat}{Theorem \ref{thm:coarse}}
\begin{thm:coarse:repeat}
There are computable functionals \( \Gamma\maps{\baire}{\cantor}  \) and \(  \hat{\Gamma}\maps{\cantor}{\baire} \) such that for all \( f \in \baire \), \( \Gamma \) is a uniform coarse reduction of \( \Gamma(f) \) to \( f \), \(  \hat{\Gamma} \circ \Gamma = \id  \)  and \( \hat{\Gamma} \) is a uniform coarse reduction of \( f \) to \( \Gamma(f) \). 
\end{thm:coarse:repeat}

Since every uniform coarse reduction is also a non-uniform coarse reduction this shows that coarse degrees contain sets in the most uniform manner one could hope for.  In contrast, we show that most (i.e., a co-meager set) of functions aren't even non-uniformly equivalent to a set \( X \).

\newtheorem*{thm:effective-dense:repeat}{Theorem \ref{thm:effective-dense}}
\begin{thm:effective-dense:repeat}
If  \( f \in \baire \) is \( 3 \)-generic then there is no set \( X \) such that \( f \NEDequiv X \).  
\end{thm:effective-dense:repeat}

This clearly entails the corresponding result for uniform effective dense reducibility.  We also further distinguish these two notions of reducibility with two results about their complexity.  Interestingly, the non-uniform versions of these two reductions differ in complexity by almost as much as one could imagine as the following two theorems illustrate.

\newtheorem*{thm:ned-piii-complete:repeat}{Theorem \ref{thm:ned-piii-complete}}
\begin{thm:ned-piii-complete:repeat}
\( f \NEDgeq Z \) is  \( \piin{1} \) and there is a computable functional taking \( \alpha \in \omega, X \subset \omega \) and producing \( \deltazn(X){2} \)  indexes for \( f \in \baire, Z  \in \cantor \) such that \( f \nUEDgeq Z \) and \( f \NEDgeq Z \) iff \( \alpha \in \kleeneO[X] \).  Indeed, there is a strictly monotonic computable function \( l_k \) such that there is no computable functional \( \Gamma \) satisfying \[
\forall(S \subset \omega)\left(\forall(k)\left[\udensity[l_k](S) \leq 2^{-k} \right] \implies \Gamma(f^{\Box S}) \in (\omega \union \set{\Box})^{\omega} \land   \Gamma(f^{\Box S}) \sagree0 Z  \right)
\]  
\end{thm:ned-piii-complete:repeat}

This theorem asserts that the relation \( f \NEDgeq Z \) is \( \piin{1} \) complete as a relation of reals \mydash not merely as a property of the indexes for \( f, Z \) \mydash since given a \( \piin{1} \) relation \( P(X) \) we can compute \( \alpha \) such that \( P(X) \iff \alpha \in \kleeneO[X]  \) and therefore there is a computable functional which takes \( X \) and an index for \( P \)  and returns \( \deltazn(X){2}  \) indexes for \( f, Z \) such that \( P(X) \iff f \NEDgeq Z  \).   It also demonstrates that there is a very strong sense in which a non-uniform effective dense reduction can be non-uniform \mydash this will contrast with our results about coarse reducibility.  We hypothesize, but do not prove, that the relation \( f \NEDequiv g \) (non-uniformly effective dense equivalent) is also  \( \piin{1} \) complete.  This result is in sharp contrast to our result for non-uniform coarse reducibility.  

\newtheorem*{thm:nonu-coarse-complexity:repeat}{Theorem \ref{thm:nonu-coarse-complexity}}
\begin{thm:nonu-coarse-complexity:repeat}
There are \( \sigmazn{4} \) formulas \( \psi_U, \psi_N \) such that \[
\forall(f, g \in \baire)\left(f \NCgeq g \iff \psi_N(f, g) \land f \UCgeq g \iff \psi_U(f, g) \right)
\] 
\end{thm:nonu-coarse-complexity:repeat}

This shows that the question of coarse equivalence is a \( \deltazn{5} \) property.  We will prove this result by showing (roughly) that \( g \) is non-uniformly coarsely reducible to \( f \) iff there is some condition which forces that every sufficiently generic (in a sense to be defined) coarse description of \( f \) computes a coarse description of \( g \).  In turn, this will show that non-uniform coarse reductions are only non-uniform as a result of uncertainty about when a coarse description gets close enough to the object being described.  This contrasts with the result above about non-uniform effective dense reducibility.  Taken together these results show radical differences between non-uniform coarse reducibility and non-uniform effective dense reducibility and somewhat more limited differences for the uniform reductions.  The extent to which these differences persist when we restrict attention to the sets \mydash perhaps effective dense reductions are much better behaved when restricted to \( \cantor \) \mydash is left as one of several questions we pose in \S\ref{sec:future}.

\section{Definitions}

\subsection{Standard Notation}

We largely adopt standard notation which we briefly summarize.  The set of non-negative integers is \( \omega \) and we identify \( n \in \omega \) with its set of predecessors.  We denote the set of functions from \( Y \) to \( X \) by\footnote{To avoid ambiguity we assume that the superscript is never itself a set of functions, i.e., \( 2^{2^r} \) denotes the set of functions from \( \set{0, 1, \ldots, 2^{r} - 1} \) to \( \set{0,1} \)} \( X^Y \) with the convention that \( \omega^n \) implies \( n > 0 \).  Set difference is denoted by \( X \setminus Y \), complement by \( \setcmp{X} \) and restriction by \( X\restr{l} \).  We abbreviate \( X\restr{\set{y}{y \geq l}} \) by \( X\restr{\geq l} \).  We define \( f \symdiff g = \set{x}{f(x)\conv \neq g(x)\conv} \) (\textbf{note} the requirement that both functions be defined).  The image of \( X \) under \( f \) is denoted \( f[X] \) though we will overload this notation in certain clear circumstances.  To avoid excess parentheses we adopt the convention that whenever \( q(z) \) is a function we denote \( q(z)(x) \) by \( q(z;x) \)

We write elements of \( \bstrs \) and \( \wstrs \) (referred to as strings) like \( \str{x_0, x_1, \ldots, x_{n-1}} \) with \( \estr \) denoting the empty string.  We call elements of \( \baire \) or \( \cantor \) (viewed as a subset of \( \baire \)) reals or infinite strings  and identify sets with their characteristic functions \mydash both for subsets of \( \omega \) and subsets of \( n \) \mydash and functions with their graphs.  Unless otherwise indicated, capital letters range over \( \cantor \).  For strings and reals we denote that \( \sigma \) is (non-strictly) extended by \( \tau \) by \( \sigma \subfun \tau \),     incompatibility by \( \incompat \), compatibility by \( \compat \) and use \( \lh{\sigma} \) for the length of  \( \sigma \).   Concatenation is written as \( \sigma\concat \tau \), the longest proper initial segment of \( \sigma \) is denoted \( \sigma^{-} \), the longest common initial segment of \( \sigma \) and \( \tau \) is denoted by \( \sigma \meet \tau \) and the string consisting of  \( k \) copies of \( m \) is written as \( \str{m}^k \).   We extend the notation \( \subfun \) and \( \compat, \incompat \) to elements in \( (\omega \union \set{\diverge})^{< \omega} \) (finite partial functions from \( \omega \) to \( \omega \)) in the obvious way, e.g., \( \rho \subfun \rho' \) iff \( \forall(x \in \dom \rho)\left(\rho'(x)\conv = \rho(x)\right) \).

The \( i \)-th computable functional is \( \recfnl{i}{}{}\conv   \), we denote convergence by \( \recfnl{i}{X}{y}\conv   \) convergence in \( s \) steps by \( \recfnl{i}{X}{y}\conv[s]   \)  and divergence by \( \recfnl{i}{X}{y}\diverge   \).  We denote totality and non-totality by \( \recfnl{i}{X}{}\conv   \) and \( \recfnl{i}{X}{}\diverge   \). We use subscripts to denote stagewise approximations and superscripts for relativization and  identify \( \recfnl[s]{i}{X}{} \) with the longest string \( \sigma  \) with \(  \recfnl{i}{X}{y}\conv[s] = \sigma(y)   \). We assume that if \(  \recfnl{i}{X}{y}\conv[s] \) then \( y < s \) and the use of that computation \mydash denoted  \( \use{\recfnl{i}{X}{y}} \) \mydash is less than \( s \) as well.  We write \( f \cequiv g \) to mean \( \dom f = \dom g \) and \( \forall(x \in \dom f)\left(f(x) = g(x)\right) \).  The \( e \)-th set r.e.  in \( X \) is \( \REset(X){e} \).  We write \( \recset(X){e} \) to suggestively abbreviate the \( e \)-th \( 0-1 \) valued functional with the implication it represents an \( X \)-computable set.  We denote that \( x \) enters (leaves) \( Y \) at stage \( s + 1 \), i.e.,   \( x \in Y_{s+1}\setminus Y_s \) (\( x \in Y_s\setminus Y_{s+1} \)), by \( x \entersat{s+1} Y \) (\( x \leavesat{s+1} Y \) ) .         

We use \( \murec{x}{\phi(x)} \) to denote the least \( x \) satisfying \( \phi \).   We say \( \phi \) holds for almost all \( x \), denoted  \( \forall*(x)\phi(x)  \), iff  \( \lnot \existsinf(x)\lnot\phi(x)  \) and generally use stars to denote a claim holds with finitely many exceptions, e.g., \( A \subset* B \). 

We assume the existence of a standard coding function for \( \wstrs \) specify a pairing function in \S\ref{ssec:defs:density} with the code for \( (x,y) \) written  \( \pair{x}{y} \).   We will generally elide the distinction between objects and their codes using corner quotes to indicate the code  meant when there is a danger of ambiguity.  We  write \( \setcol{X}{n} = \set{y}{\pair{n}{y} \in X} \) for the set coded by the \( n \)-th column of \( X \), \( \setcol{X}{= n} = \set{y}{ y \in X \land \exists(x)\left(y = \pair{x}{n} \right)} \) and similarly for \( \setcol{X}{\leq n} \).  
%  A \Tplus B = \set{x}{x =2y \land y \in A \lor x = 2y+1 \land y \in B} \qquad \TPlus_{i \in S} X_i = \set{\pair{i}{y}}{y \in X_i}
% \]
% And we 

 % \( \setcol{X}{n}  \) (\( n \)-th column of \( X \)) and \( \setcol{X}{< n}  \) (the restriction of \( X \) to it's first \( n \)-columns) standardly and let \( \Tplus_{n \in S} X_n \eqdef \set{\pair{n}{y}}{n \in S \land y \in X_n}  \)   and extend this notion to functions in the obvious way (\( \setcol{\sigma}{n}(x) = \sigma(\pair{n}{x}) \)). 

\subsection{Less Common Definitions}  

We also adopt some less common notation.  First, we extend our identification of sets with their characteristic functions to finite strings, e.g., \( \sigma \in 2^n \) will be identified with \( \set{x}{\sigma(x) = 1}  \) and we apply this identification more aggressively than is common, e.g., writing \( \sigma \subset X \).  As a result, using the same symbol for the cardinality of a set as the length of a string would create ambiguity.  Therefore, we write  \( \card{X} \) to denote the cardinality of \( X \).  

We will sometimes want to replace the initial segment of a string (infinite or finite) so we define \( \sigma \triangleright f \) to be the string defined by
\[
(\sigma \triangleright f)(x) = \begin{cases}
                                    \sigma(x) & \text{if } \sigma(x)\conv \\
                                    f(x) & \text{otherwise}
                                \end{cases}
\] 
It will also be helpful to define a notion of agreement modulo a given set.  To this end we define
\[
f \agree[S] g \iff \dom f = \dom g \land \forall(y \in \dom f)\left(y \in S \lor f(y) = g(y) \right)
\]
We adopt this particular characterization to ensure that for finite strings \( \sigma \agree[S] \tau \) implies \( \lh{\sigma} = \lh{\tau} \).

% \todo{cut this next para probably..and if so we can probably just }
% Finally, given \( \tau \in \omega^n \) and \( q \in k^n \) we will code the pair in a way that takes advantage of the fact that there are only finitely many options for \( q \).  To avoid confusion with coding a pair of arbitrary integers we denote this by  \( \code{\tau, q} \) and we assume that  \( \code{\tau, q} = a\cdot \godelnum{\tau} + \goldenum{q} \) where  \( \godelnum{q} \) is the result of some bijection between \( k^n \) and \( \set{0, \ldots, a -1} \).  

  % so  \( \code{\tau, q}  \) is a bijection of \( \omega^n \cross k^n \) and \( \omega \).       

\subsection{Density}\label{ssec:defs:density}

% We use  \( \tau \gg \sigma \) to indicate that \( \tau \) (non-strictly) majorizes \( \sigma \), i.e., \( \forall(x \in \dom \sigma)\left(\tau(x)\donv \land \tau(x) \geq \sigma(x) \right) \) and \( \tau \ggneq \sigma \) to indicate that \( \tau(x) > \sigma(x) \) for all \( x \in \dom \sigma \). 

 We define (for \( X \subset \omega \))
\begin{equation*}
\udensity[a][b](X) \eqdef \udensity[{[a,b]}] \qquad \udensity[L] \eqdef \sup_{l \in L} \frac{X\restr{l}}{l} \qquad \udensity[a](X) \eqdef \udensity[a][\infty](X) \qquad \udensity(X) \eqdef \lim_{n \to \infty} \udensity[n](X) %\udensity[{[a,b]}](X) \eqdef \udensity[a][b](X)
\end{equation*}
We call \( \udensity[l][l](X)  \) the density of \( X \) below \( l \) and \( \udensity(X) \) the upper density of \( X \).  We define the lower density \( \ldensity(X) \) similarly but with \( \inf \) in place of \( \sup \) and the density of  \( X \) to be \( \ldensity(X) = \udensity(X) \) when those two values agree.  We observe the following basic properties of density.
\begin{observation}\label{obs:density}
For all \( X, Y \subset \omega \), \( \udensity(X)\conv, \ldensity(Y)\conv \in [0,1] \) and   
\begin{align*}
&\udensity(X) = 0 \iff \density(X) = 0 \iff \ldensity(\setcmp{X}) = 1 \iff \density(\setcmp{X}) = 1\\
&\udensity(X \union Y) = 0 \iff \udensity(X) = 0 \land \udensity(Y) = 0 \\
&\density(Y) = 1 \implies \density(X \isect Y) = \density(X)
\end{align*}
\end{observation}
Thus the subsets of \( \omega \) with density \( 0 \) form an ideal which we denote by \( \mathbb{I} \) and those with density \( 1 \) form a filter.  

% We also make the following further observation\footnote{This is trivial if \( p = 0 \) and if \( p > 0 \) this follows by elementary manipulation from the fact that
% \[
% \frac{\card{c[X]\restr{l}}}{l} = \frac{\card{X\restr{n}}}{l} \text{ where } n = \card{C\restr{l}} 
% \].}.
% \begin{observation}\label{obs:density-multiply}
% If \( c\maps{\omega}{\omega} \) is strictly monotonic, \( \density(c[\omega]) = p \) and \( \density(X) = q \) (\( \udensity(X) = q \)) then \( \density(c[X]) = pq \) (\( \udensity(c[X]) = pq \)).
% %Suppose that \( \density(C) = p \), \( \density(X) = q \) (\( \udensity(X) = q \) )  and that \( c(n) \) is a strictly monotonic enumeration of \( C \) then .      
% \end{observation} 

It will be helpful for our pairing function to ensure that every column has a well-defined positive density.  Specifically, we define our pairing function inductively as follows. 
\begin{equation}\label{eq:pairing}
\pair{0}{n} = 2n \qquad \pair{i+1}{n} = 2\pair{i}{n} + 1 
\end{equation}
We make the following observation about our pairing function which follows directly from the definition and by manipulation of limits.

\begin{observation}\label{obs:pairing}
 The pairing function satisfies \( \pair{0}{0} = 0 \), \( \pair{x}{y} \) is monotonic in both \( x \) and \( y \) and if \( S_i = \set{\pair{i}{x}}{x \in S}  \) then, for any \( i \), \( \density(S_i) = \frac{1}{2^{i+1}}\density(S) \)  and  \( \udensity(S_i) = \frac{1}{2^{i+1}}\udensity(S) \).  
\end{observation}

\subsection{Coarse Reducibility}

We will define coarse reducibility in terms of coarse descriptions.

\begin{definition}\label{def:coarse-description}
Given  \( f \in \baire \) and a function \( f' \) with \( \dom f = \dom f' = \omega \) we say \( f' \) is a coarse description of \( f \) if  \( \udensity(f \symdiff f') = 0 \). 
\end{definition}

% While we don't formally restrict our notion of a coarse description to an element of \( \baire \) this won't matter since we can simply redefine \( f'(x) \) to be \( 0 \) if \( f'(x) \nin \omega \) without affecting the status of \( f' \) as a coarse description of \( f \).  
Using this notion we define coarse reducibility.

 % We diverge from other works that call coarse reducibility non-uniform coarse reducibility as we feel that terminology is confusing (the reductions involved can, but need not, be non-uniform) and is a poor match for the usual practice of only adding an adjective to distinguish the uniform version of a property.

\begin{definition}\label{def:coarse-reductions}
\( g \in \baire \) is non-uniformly coarsely reducible to \( f \in \baire \), written  \( f \NCgeq g \), iff every coarse description of \( f \) computes a coarse description of \( g \). \( g \in \baire \) is uniformly coarsely reducible to \( f \in \baire \), written \( g \UCleq f \), just if  a single computable functional  maps every coarse description of \( f \) to a coarse description of \( g \).  Equivalences and strict inequalities are defined in the usual fashion.   % A non-uniform coarse reduction of \( g \in \baire \) to \( f \in \baire \)  is a map \( \Lambda\maps{\baire}{\baire} \) taking each coarse description \( f' \in \baire \) of \( f \) to a coarse description \( \Lambda(f') \) of \( g \) such that \( \Lambda(f') \Tleq f' \).  A uniform coarse reduction is one where \( \Lambda \) is given by a computable functional.  We write \( g \NCleq f \) (\( g \UCleq f \)) and \( g \NCeq f \) (\( g \UCeq f \)) for (uniform) non-uniform coarse reducibility and (uniform) non-uniform coarse equivalence respectively.
\end{definition}

As we identify sets with their characteristic functions (and think of \( \cantor \) as a subspace of \( \baire \)) this definition also defines coarse and uniform coarse reductions for sets.  It isn't hard to see that this definition agrees with one which only considers coarse descriptions in \( \cantor \) on \( \cantor \).

\subsection{Effective Dense Reducibility}

In \cite{Astor2019Dense} an interesting new notion of robust information coding is defined \mydash effective dense descriptions.  Informally speaking, an effective dense description of a function \( f \) is a function that agrees with \( f \) on a dense subset but is allowed to say ``dunno'' on a set of upper density \( 0 \).  This is a relatively strong notion of approximation since we know that on any \( x \) where an effective dense description of \( f \)  asserts that \( f(x) = y \) that really is true.   Following \cite{Astor2019Dense} we use \( \Box \) as a special symbol meaning ``dunno'' which we will formally\footnote{Even more formally, we can define \( \hat{f} \in \baire \) to be an effective dense description of \( f \in \baire \) just if \( \forall(x)\left(\hat{f}(x) = 0 \lor \hat{f}(x) - 1 = f(x)  \right) \) and \( \hat{f}(x) = 0 \) on a set of density \( 0 \).  However, it will be less confusing to abuse notation and say that \( \hat{f} \in (\omega \union \set{\Box})^{\omega} \) with the understanding that the transformation \( \hat{f}(x) - 1 \) is applied seamlessly in the background.} identify with \( -1 \)  to indicate the unsure answer.  To avoid the potential for confusion with \( \Box \)  we change the end proof symbol to \( \blacksquare \).    When talking about effective dense reducibility functions may be either from \( \omega \) to \( \omega \) or to \( \omega \union \set{\Box}{} \). When necessary, we will explicitly indicate which we mean.  Following \cite{Astor2019Dense} we define the strong domain of a function as follows.

%One might think that since the Turing degree of every function includes a set that reductions that arise from strong notions of information coding ought to also have this property.  In this section, we show this isn't the case answer question negatively 7.2 for effective dense reducibility.  First, however, we give equivalent definitions as those in \cite{Astor2019Dense}  for effective dense reducibility.

\begin{definition}\label{def:strong-domain}
The strong domain of \( f \), denoted \( \sdom f \), is defined as follows
\[
\sdom f = \set{x}{f(x)\conv \land f(x) \in \omega}
\]
\end{definition}

We also introduce a notion of weak equality that accommodates \( \Box \).

\begin{definition}\label{def:box-equal}
\begin{align*}
x \boxeq y &\iffdef x = y \lor x = \Box \lor y = \Box \\
f \sagree g &\iffdef \forall(x \in \omega)\left(f(x)\conv \boxeq g(x)\conv \right) \\
f \sagree[\epsilon] g &\iffdef f \sagree g \land \udensity(\sdom f \union \sdom g) \leq \epsilon 
% \tau \subfun_{\Box} \tau' &\iffdef \forall(x \in \dom \tau')\left(\tau(x)\conv \boxeq \tau' \right)
\end{align*}    
\end{definition}

In the last definition we assume that \( \tau \in (\omega \union \set{\Box})^{< \omega} \).  Consistent with our conventions we understand \( f \sagree* g \) and \( f \sagree*[\epsilon] g \) to mean the unstarred versions hold on a co-finite set.  We can now express the property of agreeing on strong domains of density \( 1 \). 

\begin{definition}\label{def:effective-dense-description}
\( g\maps{\omega}{\omega \union \set{\Box}} \) is an effective dense description of \( f \in \baire \)  if \( f \sagree0 g \).  
\end{definition}

Effective dense reducibility is now defined just as coarse reducibility replacing coarse description with effective dense description.  

\begin{definition}\label{def:effective-dense-reductions}
\( g \in \baire \) is non-uniformly effective densely reducible to \( f \in \baire \), written  \( f \NEDgeq g \), iff every effective dense description of \( f \) computes an effective dense description of \( g \). \( g \in \baire \) is uniformly effective densely reducible to \( f \in \baire \), written \( g \UCleq f \), just if there is a single computable functional  taking every effective dense description of \( f \) to an effective dense description of \( g \).  Equivalences and strict inequalities are defined in the usual fashion. 
\end{definition}

We adopt the following notation for constructing effective dense descriptions of a function. 

% .  Most of our results will require we build effective dense descriptions of various functions.  Rather than build these descriptions directly it will be more convenient to instead build elements of \( \mathbb{I} \) and turn them into effective dense descriptions using the following definition.

\begin{definition}\label{def:desc-mod-S}
For \( S \subset \omega \) and \( f \in (\omega \union \set{\Box})^{\omega} \) define 
\[ f^{\Box S} = \begin{cases}
                    f(x) & \text{if } x \nin S\\
                    \Box & \text{otherwise}
                \end{cases} \]   
where we understand \( \tau^{\Box \sigma} \) to have domain \( \dom \tau \isect \dom \sigma \).    
\end{definition}

While this obviously lets us build effective dense descriptions when \( S \in \mathbb{I} \) we observe that all effective dense descriptions are of this form.

\begin{observation}\label{obs:f-box-all-edd}
If \( f \in \baire \) then  \( \hat{f} \sagree0 f \) iff there is some \( S \in \mathbb{I} \) such that \( \hat{f} = f^{\Box S} \).   
\end{observation}

To check the only if, let  \( S = \hat{f}^{-1}(\Box) = \omega \setminus \sdom \hat{f} \) and note that \( S \)  has density \( 0 \) and \( \hat{f} = f^{\Box S}  \).  %We continue our practice of identifying \( \sigma \in \bstrs \) with \( \sigma^{-1}(1) \) to understand \( f^{\Box \sigma} \). 

\section{Effective Dense Complexity}\label{sec:effective-dense-complexity}

In this section, we establish that \( f \NEDgeq X \) has the maximal possible complexity by proving that the relation  \( f \NEDgeq Z \) is \( \piin{1} \) complete.  We prove this result before tackling the question of whether a function is effective dense equivalent to a set in \S\ref{sec:edd-no-sets} because the fact that this relation isn't hyperarithmetic \mydash more particularly, the lack of an easy proof  (if it is even true) that if \( f \NEDequiv X \) then \( f \NEDequiv Y \) for some \( Y \) hyperarithmetic in \( f \) \mydash will force us to use a more complicated strategy in  \S\ref{sec:edd-no-sets}.

There is a certain ambiguity in saying that the relation \( \NEDgeq \)  is \( \piin{1} \) complete. Does it mean that, relative to some notion of definability, the set of indexes such that \( f_i \NEDgeq Z_i \) is a complete \( \piin{1} \) relation on numbers or that we can reduce any \( \piin{1} \) relation on reals to a question of this form?  We prove that \( f \NEDgeq X \) is \( \piin{1} \) complete in both senses. 

\begin{theorem}\label{thm:ned-piii-complete}
\( f \NEDgeq Z \) is  \( \piin{1} \) and there is a computable functional taking \( \alpha \in \omega, X \subset \omega \) and producing \( \deltazn(X){2} \)  indexes for \( f \in \baire, Z  \in \cantor \) such that \( f \nUEDgeq Z \) and \( f \NEDgeq Z \) iff \( \alpha \in \kleeneO[X] \).  Indeed, there is a strictly monotonic computable function \( l_k \) such that there is no computable functional \( \Gamma \) satisfying \[
\forall(S \subset \omega)\left(\forall(k)\left[\udensity[l_k](S) \leq 2^{-k} \right] \implies \Gamma(f^{\Box S}) \in (\omega \union \set{\Box})^{\omega} \land   \Gamma(f^{\Box S}) \sagree0 Z  \right)
\]  
\end{theorem}

This result contrasts with the complexity of coarse reducibility which we will see in section \ref{sec:coarse-complexity} is arithmetic.   We state the theorem in terms of indexes so that we can give a non-relativized construction of \( f, Z \) from \( \alpha \) and then rely on relativization to give the complete theorem.   While we will build \( f, Z \) uniformly in \( P, X \) the effective dense reduction from \( Z \) to \( f \) (when it exists) will be strongly non-uniform.  The indeed claim requires that not only is \( Z \) never uniformly effective densely reducible to \( f \) this non-uniformity remains even when we restrict our attention to effective dense descriptions of \( f \) whose strong domains can be seen to achieve density \( 1 \) in a particularly uniform fashion.  We will later contrast this with the kind of non-uniformity possible for coarse reductions.
%  \newtheorem*{prop_ned_cons_nonu}{Proposition \ref{prop:ned-construction-non-uniform}}
% \begin{prop_ned_cons_nonu}
% When the \( \piin{1} \) predicate \( P(X) \) holds we have \( f \NEDgeq Z \) but not \( f \nUEDgeq Z \).  Indeed, there is a strictly monotonic computable function \( l_k \) such that there is no computable functional \( \Gamma \) satisfying \[
% \forall(S \subset \omega)\left(\forall(k)\left[\udensity[l_k](S) \leq 2^{-k} \right] \implies \Gamma(f^{\Box S}) \in (\omega \union \set{\Box})^{\omega} \land   \Gamma(f^{\Box S}) \sagree0 Z  \right)
% \] 
% \end{prop_ned_cons_nonu}

The key idea behind theorem \ref{thm:ned-piii-complete} is that we can use knowledge of the fact that \( S \supset U \) for some set \( U \) to help compute \( Z \) from \( f^{\Box S} \).  Thus, we can have one computation that computes an effective dense description of \( Z \) from the values of \( f^{\Box S} \) on some infinite subset of \( U \) when \( S \nsupset* U \)  and another computation that uses the information that  \( S \supset U \) to compute an effective dense description of \( Z \) from \( f^{\Box S} \) when \( S \supset* U \).  This ability to code information into the set of locations that the true value of \( f \) is hidden is a fundamental difference between effective dense reducibility and coarse reducibility.  

To actually make the existence of a reduction \( \piin{1} \) complete we go beyond considering \( 2 \) possible reductions by having a countable collection of sets \( U_i \) with reductions associated with finite binary strings corresponding to patterns of containment or non-containment (mod finite) of the sets \( U_i \).  Since \( \piin{1} \) predicates can be transformed into the claim that some r.e. set of binary strings lacks any infinite compatible sequence (this is the \( \bstrs \) version of asserting that a computable tree in \( \wstrs \) is well-founded) we can use this correspondence to construct \( f, Z \) for any given \( \piin{1} \) predicate.   Specifically, for each \( U_i \) we will encode information into \( f \) to help it compute \( Z \) which can be recovered from the values of \( f \) on any infinite subset of \( U_i \).  On the other hand, we will build the sets \( U_i \) so that \( S \supset* U_i \) communicates information about where \( Z \) might be non-zero.

% \begin{theorem}\label{thm:ned-piii-complete}
% \( f \UEDgeq g \) is    
% \end{theorem}
% \begin{proof}
% Consider the following formula

% \begin{align*}
% \psi(f,g) &\iffdef \forall(\tau_0 \subfun_{\Box} f)\exists(\tau_1 )
% \end{align*}

% \end{proof}

\subsection{Preliminaries}
It is easy to see that non-uniform (and also uniform) effective dense reducibility is \( \piin{1} \) as the sentence \[
\psi(S, f, g) \iffdef \udensity(S) = 0 \implies \exists(i)\left(\recfnl{i}{f^{\Box S}}{}\conv \sagree0 g   \right)
\] 
is clearly arithmetic making \( \forall(S \in \cantor)\psi(S, f, g) \) \mydash which is equivalent to \( f \NEDgeq g \) \mydash  \( \piin{1} \).  Thus, to prove theorem \ref{thm:ned-piii-complete} it will suffice to show that given \( \alpha \) we can build \( f, Z \) such that \( f \NEDgeq Z \) iff \( \alpha \in \kleeneO \) in a uniform (i.e., supporting relativization) fashion.  To this end we note the following result \mydash which is essentially the translation of the equivalence between \( \piin{1} \) predicates and computable trees in \( \wstrs \) translated into \( \bstrs \).

\begin{lemma}\label{lem:piii-normal-form}
Given \( \alpha \in \omega \) there is a computable function \( \pi\maps{\omega}{\bstrs} \) whose index is computable from \( \alpha \) such that all of the following hold
\begin{enumerate}
    \item\label{lem:piii-normal-form:injective} \( \pi \) is a total injective function.
    \item\label{lem:piii-normal-form:non-decreasing}  \( n \geq m \implies \lh{\pi(n)} \geq \lh{\pi(m)} \)
    \item\label{lem:piii-normal-form:bound} \( \lh{\pi(n)} \leq n \). 
    \item\label{lem:piii-normal-form:unbounded-height} For all \( k \) there is a \( \subfun \) increasing sequence of height \( k \).     
    \item\label{lem:piii-normal-form:complete} \( \alpha \nin \kleeneO \) iff there is an infinite \( \subfun \) increasing sequence.
\end{enumerate}
Moreover, this uniformly relativizes to produce an \( X \)-computable sequence such that \( \alpha \nin \kleeneO[X] \) iff there is an infinite \( \subfun \) increasing sequence.
\end{lemma}  
In the above lemma, an increasing sequence of height \( k \) is an injective function  \( i \mapsto n_i \)  from \( k \) to \( \omega \) such that \( \pi(n_i) \subfun \pi(n_{i+1}) \).  
\begin{proof}
It is well known (e.g., see \cite{Sacks1990Higher}) that we can uniformly find an index \( e \) from \( \alpha \) such that \( \alpha \nin \kleeneO \) iff \( \recfnl{e}{X}{} \) is total for some \( X \).  We can assume that if \( \sigma = \str{0}^{n}\concat[1] \) and  \( x \leq n \) then \( \recfnl{e}{\sigma}{0}\conv[n] \).  We first define a function \( p \) with \( p(0) = \estr \) and \( p(n), n > 0 \)  to be the first \( \sigma \in \bstrs \) found such that 
\[
\lh{\sigma} \leq n \land \forall(m < n)\left(\lh{\recfnl{e}{\sigma}{}} > \lh{\recfnl{e}{\sigma}{p(m)}}  \right)
\]  
It is easy to see that \( p(n) \) satisfies \ref{lem:piii-normal-form:complete}, \ref{lem:piii-normal-form:bound} and \ref{lem:piii-normal-form:injective} since our assumption above guarantees that for every \( n \) there  is some \( \sigma \) satisfying this equation and that same assumption also ensures that part \ref{lem:piii-normal-form:unbounded-height} is satisfied.   We now replace each \( p(n) \) anytime we violate \ref{lem:piii-normal-form:non-decreasing} with a sequence of all extensions of \( p(n) \) of some (fixed) sufficiently large length. 

Formally, let \( \pi(0) = \estr, n_0 = 0 \) and assume that at stage \( s - 1 \) we've defined \( n_s \) and  \( \pi\restr{n_s} \) from \( p\restr{s} \).  At stage \( s > 0 \) let \( t < s \) be maximal such that \( p(t) \subfun p(s) \) and let \( l = \max(\lh{\pi(n_s)} - \lh{p(s)}, 1 + \lh{\pi(n_t)} - \lh{p(s)}, 0)  \).  Set  \( n_{s+1} = n_s + 2^l \),  \( \pi(n_s + i) = p(s)\concat \sigma_i \) for \( i < 2^l \) where \( \sigma_i \) is the (length \( l \)) binary representation of \( i \).

The only condition it is unclear if \( \pi \) satisfies is \( \ref{lem:piii-normal-form:bound} \).  However, this is verified by a quick induction using the fact that \( \lh{p(s)} \leq s \leq n_s \).      

\end{proof}

Since we want to ensure \( f \nNEDgeq Z \) only when we have some infinite sequence \( \pi(n_i) \) with \( \pi(n_{i+1}) \supfun \pi(n_i) \) we will work on the \( i \)-th computable functional when that sequence has height \( i \).  To that end, we adopt the following piece of notation with the second equality following from part \ref{lem:piii-normal-form:non-decreasing} of lemma \ref{lem:piii-normal-form}.   

\begin{equation}\label{def:i-of-n}
i^\pi_n \eqdef \card{\set{m}{m < n \land \pi(m) \subfun \pi(n)}} =\card{\set{m}{ \pi(m) \subfun \pi(n)}}
\end{equation}

% \todo{See if we can cut this definition}
% \begin{definition}\label{def:P-and-uninjured}
% \[ P_s = \set{\pi(0), \pi(1), \ldots, \pi(s-1)}{} \qquad P = \Union_{s \in \omega} P_s \]
% We write \( \sigma \subfun P_s \) (\( \sigma \subfun P \)) to mean \( \sigma \subfun \pi(n) \) for some \( n < s \) (\( n \in \omega \)).  

% \end{definition}

In the rest of this section, we will  build \( f, Z \) so that \( f \NEDgeq Z \) iff every \( X \in \cantor \)  extends \( \pi(n) \) for only finitely many \( n \) using \( \deltazn{2} \) approximations \( f_s, Z_s \).  

\subsection{Basic Approach}

The basic (to be complicated) idea is to translate between an element \( X \in \cantor \) and a set \( U(X) \in \mathbb{I} \) so that \( X \) extends \( \pi(n) \) for infinitely many \( n \) iff  \( f^{\Box U(X)} \) doesn't compute any effective dense description of \( Z \).  Since finite changes to \( X \) need to potentially affect the existence of a computation from  \( f^{\Box U(X)} \) we ensure that each bit of \( X \) controls an infinite set. Specifically,  we  define a sequence of infinite sets \( U_i \in \mathbb{I}  \) with \( U_i \) contained in the \( i \)-th column of \( \omega \) and make the following definition.

\begin{definition}\label{def:U-of-set}
\[
U(X) \eqdef \Union_{i \in X} U_i \qquad \qquad U \eqdef U(\omega)
\]
We extend this to elements of \( \bstrs \) via our usual identification of sets and characteristic functions and define \( U_s(X)\) by replacing \( U_i \) with \( U_{i,s} \) in the above equation.     
\end{definition}   

Of course, not every \( S \in \mathbb{I} \) will be of the form \( U(X) \) for some \( X \).  So we will associate every \( S \in \mathbb{I} \) with a unique \( X \in \cantor \) based on whether \( S \supset* U_i \).  To that end we define.        

\begin{definition}\label{def:U-compatible}
For \( S \subset \omega \) and  \( \sigma  \) (in \( \bstrs \) or \( \cantor \)) we define  
\[ 
S \geq_U \sigma \iff \forall(i \in \dom \sigma)\left(S \supset* U_i \iff \sigma(i) = 1 \right)
\]  
\end{definition}

Note that, for all sets \( S \subset \omega \) there is a unique set \( X \in \cantor \) with \( S \geq_U X \).  Our construction will build computable functionals \( \theta_\sigma \) such that if \( \sigma \) isn't extended by any \( \pi(n) \)  and \(  S \geq_U \sigma  \) then \( \theta_\sigma(f^{\Box S}) \sagree*0 Z \).  When \( \sigma(i) = 1 \), so \( S \supset* U_i \), we will insist that \( \theta_\sigma(f^{\Box S};x) = \Box  \) whenever \( \pair{i}{x} \in S  \)  so (weak) agreement can be achieved just by enumerating elements into \( U_i \).  When \( \sigma(i) = 0 \), so \( S \nsupset* U_i \), then \( f^{\Box S} \) can see \( f \) on an infinite subset of \( U_i \).  By building \( U_i \) to be r.e., we ensure that in such cases \( f^{\Box S} \) can compute such an infinite subset of \( U_i \) and we will define sets \( Y_i \) that can be computed by all such \( f^{\Box S} \).

\begin{definition}\label{def:Y-sets}
Define \( Y_i \) (and \( Y_{i,s} \in \bstrs  \) similarly with respect to \( f_s \)) by \[
Y_i(x) = j \iff  \exists(z \in U_i)\left(\godelnum{\sigma} = f(z) \land \sigma(x) = j  \right)
\]
We write \( \lh{f(x)} \) to refer to the length of \( \sigma \in \bstrs \) such that \( f(x) = \godelnum{\sigma} \).     
\end{definition} 

We will ensure \( Y_i \in \cantor \) and that it is well-defined.  In the above definition, our coding function is a bijection of \( \bstrs \) with \( \omega \).  While any `reasonable' coding function would probably work we define the following coding
\begin{equation}\label{eq:coding-bstrs}
\godelnum{\sigma} =2^{\lh{\sigma}} -1 + \sum^{\lh{\sigma}-1}_{n = 0} 2^n\sigma(n)
\end{equation} 
To see that this coding is indeed a bijection of \( \bstrs \) and \( \omega \) note that our encoding maps \( \sigma \) to \( x -1 \) where \( x \) has the base-2 representation \( \str{1}\concat \sigma  \) and that every positive integer has a base-2 representation of that form.

As we commit to ensuring that \( Y_i \) is well-defined, i.e., any sequence \( z_j \) gives the same set, it will suffice to ensure that \( \lim_{j \to \infty} \lh{\sigma_j} = \infty  \) to guarantee that every \( f^{\Box S} \) with \( S \nsupset* U_i \) can compute \( Y_i \).  The following definition is meant to capture the property necessary to be able to enumerate \( x \) into \( U_i \)  

\begin{definition}\label{def:Yi-compat}
We say  \( x \) is compatible with \( Y_{i} \) at stage \( s \) if \( x = \pair{i}{y} \) for some \( y \), \( f_s(x) = \godelnum{\sigma} \) and \( \sigma \subfun Y_{i,s} \).
\end{definition}  

Note that the definition of compatibility is stronger than merely not disagreeing with \( Y_{i,s} \) and actually ensures that adding \( x = \pair{i}{y} \) to \( U_i \) leaves \( Y_{i_s} \) unchanged.  With this background in place we can now specify the reductions \( \theta_\sigma \).

\begin{definition}\label{def:theta-sigma}
Given \( S \in \mathbb{I}, f \in \baire \) for \( i < \lh{\sigma} \) we define
\begin{align*}
\chi^{\sigma}_i(S,f;x) &=  \begin{cases}
                        Y_i(\pair{\sigma\restr{i}}{x}) & \text{if } \sigma(i)\conv = 0 \\
                        0 & \text{if } \sigma(i)\conv = 1 \land \pair{i}{x} \nin S \\
                        \Box & \text{if }  \sigma(i)\conv = 1 \land \pair{i}{x} \in S
                      \end{cases}\\
\theta_\sigma(f^{\Box S};x) &= \sum_{i = 0}^{\lh{\sigma}-1} \chi^{\sigma}_i(S,f;x) \bmod 2
\end{align*}
Where we stipulate that \( \Box + j = \Box \).  We also define the stagewise approximations \( \chi^{\sigma}_{i,s}(S,f;x) \) and \( \theta_{\sigma,s} \)  by using \( Y_{i,s} \) in the above equation. 
\end{definition}
We also specify a computation that \( f^{\Box S} \) will use to compute \( \theta_\sigma(f^{\Box S};x) \) when  \( S \geq_U \sigma \).  The only case of interest is when \( \sigma(i) = 0 \) in which case  \( f^{\Box S} \)  computes \( Y_i(\pair{\sigma\restr{i}}{x}) \)   by searching for \( z \in U_i \setminus S \) with \( \lh{f(z)} > \pair{\sigma\restr{i}}{x} \)  and returning \( \tau(\pair{\sigma\restr{i}}{x}) \) where \( \tau = \godelnum{f(z)}  \).  We understand \( \pair{\sigma\restr{i}}{x} \) to use the encoding from \eqref{eq:coding-bstrs} for \( \sigma\restr{i} \). 

The reason for the use of the string \( \sigma \) in  \( Y_i(\pair{\sigma\restr{i}}{x}) \) is to ensure if \( i \) is the first location at which \( \sigma(i) = 0 \) disagrees with  \( \pi(n) \) then flipping the value of  \( Y_i(\pair{\sigma\restr{i}}{x}) \) by changing the values of \( f \) on \( U_i \)   will allow us to preserve the computation \( \theta_\sigma(f^{\Box S};x) \) for \( S \geq_U \sigma \) while flipping \( Z(x) \) to diagonalize against a computation from  \( U(\pi(n)) \supset U_i \) which can't see \( f\restr{U_i} \).

While we've described how we will ensure that \( f \NEDgeq Z \) when there is no \( X \) extending infinitely many \( \pi(n) \) it will take slightly more to ensure \( f \nNEDgeq Z \) when there is such an \( X \).  Our first thought was to simply ensure that \( f^{\Box U(X)} \) doesn't compute any effective dense description of \( Z \).  However, if we merely tried to produce disagreement between  \( \recfnl{i^\pi_n}{f^{\Box U(\sigma)}}{} \) and \( Z \) when \( \sigma  \) is extended by some \( \pi(n) \)  then even if \( \sigma \subfun X \) we might never observe such disagreement (convergence could wait to see parts of \( U(\tau) \) with \( \sigma \subfunneq \tau \subfun X \)).  We could search for \( \tau, m \) with  \( \sigma \subfun \tau \subfun \pi(m) \) and try to preserve disagreement between \( \recfnl{i^\pi_n}{f^{\Box U(\tau)}}{} \)  and \( Z \) but since we don't know which \( \tau \) extend to a full path this risks committing us to accept infinite restraint.    

The natural solution to this problem is to search for a stage \( s \), a string  \( \tau \supfun \sigma \) and \( x \)  where we can produce disagreement between \( \recfnl[s]{i^\pi_n}{f^{\Box U(\tau)}}{x} \) and \( Z_{s+1}(x) \)  and somehow dump elements in \( U_s(\tau) \setminus U_s(\sigma) \)  into \( U_{s+1}(\sigma) \) to preserve that disagreement.  This is essentially what we do, except instead of actually modifying \( U(\sigma) \) we instead define a sequence of adjustment sets \( \hat{U}_\sigma \) and place those elements into \( \hat{U}_\sigma \).  We make the following definition to help us talk about the resulting computations.        

\begin{definition}\label{def:hat-U}
\begin{alignat*}{2}
\hat{U}(X) &\eqdef \Union_{\sigma \subfun X} \hat{U}_\sigma \qquad & \mathbb{U}(X) &\eqdef \hat{U}(X) \union U(X) \\
\hat{U} &\eqdef \Union_{\sigma \subfun X} \hat{U}_\sigma  \qquad & \mathbb{U} &\eqdef U \Union \hat{U}
\end{alignat*}
\end{definition}

Stagewise versions are defined in the obvious way and denoted by an appropriate subscript.

\subsection{Requirements}

In stating our requirements we will use \( l_s \) to denote a value chosen large at the start of stage \( s \), i.e., \( l_s \) will be bigger than whatever computable function of values mentioned at or before stage \( s -1 \) is needed for the verification to succeed.  Using this notation, we can now state the requirements.  Note that \( \sigma \) in \req{P}{\sigma} ranges over \( \bstrs \).  

\begin{requirements}
\require{I}{k} \exists(l)\left(\udensity[l](\mathbb{U}) \leq \frac{1}{2^k}\right) \\
\require{P}{\sigma} S \geq_U \sigma \land S \in \mathbb{I} \land \forall*(m)\left(\pi(m) \nsupfun \sigma \right) \implies \theta_{\sigma}(f^{\Box S}) \sagree*0 Z \\
% \exists(l)\left(S \geq_U \pi(n) \land \forall(m > n)\left(\pi(m) \nsupfun \pi(n)\right) \implies \theta_{\pi(n)}(f^{\Box S})\restr{l} \sagree Z\restr{l}\right) \\
\require{Y}{n}  \forall*(x \in U_i)\left(\lh{f(x)} \geq n\right)   \\
\require{N}{n} \existsinf(x,s)\exists(\tau \supfun \pi(n))\left(\recfnl{i^\pi_n}{f^{\Box \mathbb{U}_s(\tau)}}{x}\conv[s] \implies  \recfnl{i^\pi_n}{f^{\Box \mathbb{U}(\pi(n)}}{x}\conv \neq Z(x) \right) 
\end{requirements} 
 
 As it may not be clear, we note that the purpose of \req{Y}{n} is to ensure that the sets \( Y_i \) are infinite and computed from \( f \) on any infinite subset of \( U_i \).  We won't specifically list it as a requirement, but we also will ensure that each \( U_i \) is infinite.  All requirements will be subject to finite injury from our attempts to satisfy \( \req{N}{n} \). 

As discussed above, to meet \( \req{N}{n} \) when \( \sigma = \pi(n) \)  we will need to move  finitely many elements from the sets \( U_j \) with \( \tau(j) = 1 \) and \( j \geq \lh{\tau} \) to the set \( \hat{U}_\sigma \).  This will only occur finitely many times for each \( j \) but will mean that the sequence \( U_j \) will not be uniformly r.e..

% that if we are working to satisfy \( \req{N}{n} \) and \( \udensity(\sdom \recfnl{i^\pi_n}{f^{\Box U_{\pi(n)}}}{}) = 1 \)  then we eventually can find some \( x \in \sdom \recfnl{i^\pi_n}{f^{\Box U_{\pi(n)}}}{} \) which we can add to each \( U_j \) with  \( j \in \sigma^{-1}(0) \) compatible with each \( Y_j \).  This will allow us to preserve the correctness of \( \theta_\tau  \) for short \( \tau \) incompatible with \( \sigma \).  This idea is vindicated by the next lemma.

\subsection{Construction}

\subsubsection{Machinery}

We introduce the following machinery for use in the construction.  Formally, these objects are all defined relative to a stage which will be indicated with a subscript when necessary but will otherwise be suppressed.  We assume that these objects are modified only in the ways we specify at each stage, e.g., \( R_{s+1} = R_s \) modulo any changes we make at stage \( s \) and that any maximum taken over an empty set in the definitions below yields \( 0 \). 

\begin{itemize}
    \item A finite set \( R \) (with \( R_0 = \eset \)) of triples \( (n, i, \sigma) \) with \( n, i \in \omega \) and \( \sigma \in \bstrs \)   which represent instances of \req{N}{n} which we haven't achieved a disagreement victory for yet.

    \item a finite set \( R^{S} \) (with \( R^{S}_0 = \eset \))  of quadruples \( (n, i, \sigma, x) \) where \( x \) witnesses the disagreement victory we have achieved (the S is for satisfied).

    %then  \( \frac{2n}{l_s} < \frac{1}{2^{s+1+n}} \) and for any \( l' < n \) and sequence \( \sigma_i \in 2^{l' +1}, i < n \) there are distinct \( x_0, x_1 \in S^{\sigma_0, \ldots, \sigma_{n-1}} \) (defined in \eqref{eq:S-simul-equal}) such that \( \pair{i}{x_0}, \pair{i}{x_1} \in [l_{s-1}, l_s) \) for all \( i < n \)         

    % \item A strictly monotonic function \( k(n) \) (with \( k_0 = \eset \)) defined on some initial segment of \( \omega \) with requirements with weaker priority than \( n \) obliged to respect \( \req{I}{k(n)} \).   We will also use \( k(n) \) to identify the minimum value of \( \lh{f(z)} \) such that a requirement with weaker priority than \( n \) can add \( z \) to \( U_i \).

    \item A restraint function \( r(n) \) (with \( r_0(n) = n  \) for all \( n \))  defined on some finite subset of \( \omega \) indicating the initial segment of \( f, Z \) and \( U_i \) which \( \req{N}{n} \) wishes to preserve. We define \( \overbar{r}(n) =  \max_{m < n} r(n) \).
    %and \( \overbar{k}(n) = \max_{m \leq n} k(n) \).  We also let and \( k = \max_{n \in \omega} k(n) \).        

\end{itemize}

Keeping track of the sets \( R, R^{S} \)  will make it more convenient to identify cases where a higher priority requirement \( \req{N}{n} \) renders a lower priority requirement \( \req{N}{n'} \) no longer satisfied. Finally, we specify the conditions under which a tuple in \( R \) may prompt us to act on its behalf.

\begin{definition}\label{def:viable-tuple}
The triple \( (n, i, \sigma) \in R_s \) is viable at stage \( s \) (with witnesses \( x, \tau \)) if all of the following obtain for some \( x, \godelnum{\tau} < s \)  
\begin{enumerate}
    \item \( n < s \land \sigma \subfun \tau \)
    \item\label{def:viable-tuple:pair-ls-bound} \( \pair{\sigma}{x} < l_{s-1} \)  
    \item \label{def:viable-tuple:convergence}\( \displaystyle \recfnl{i}{f^{\Box \mathbb{U}_s(\tau)}}{x}\conv[s] \in \omega \) 
    \item Either \( Z_s(x) \neq \recfnl{i}{f^{\Box \mathbb{U}_s(\tau)}}{x}  \) \mydash in which case we say  \( x, \tau \) are a trivial witness for the viability of  \(  (n, i, \sigma)  \)  \mydash or all of the following obtain
    \begin{enumerate}
        % \item \( Z_s(x) = Z_0(x) \) 
        \item \( x > \overbar{r}(n) \)
        \item \( x \geq l_n \)  
        \item\label{def:viable-tuple:compatible} For each \( j \) with \( \sigma(j) = 0 \), \( n \leq \lh{f(\pair{j}{x})} \) and  \( x \) is compatible with \( Y_j \) (at stage \( s \)).
    \end{enumerate} 
\end{enumerate}  
\end{definition} 

The idea behind this definition is that the triple \( (n, i, \sigma) \) is an instruction to try and ensure \( \recfnl{i}{f^{\Box \mathbb{U}(\sigma)}}{} \nsagree Z \) which we try and satisfy with priority \( n \).  While we could have simply used \( \lh{\sigma} \) as the priority we then would have had to break ties (e.g., via lexicographic ordering) and this would have resulted in a  proliferation of powers of \( 2 \) in the verification.  %As we will always ensure that at stage \( s \), if \( i < s \) then  \( \lh{Y_{i,s}} \geq l_{s-1} \) the      

\subsubsection{Initial Conditions}

We start by setting \( Z_0 = \eset \) and \( U_\omega = \eset \).  However, for \( f \) we will need to take a bit more care in setting the initial conditions.  This is because to satisfy \req{N}{n} we will sometimes need to add leave \( f_s \) unchanged while enumerating some \( x \) into \( Z \).  To do this without violating \req{P}{\sigma} we will also want to enumerate some \( \pair{j}{x} \) into \( U_j \) ensuring that \( \theta_\sigma(x) = \Box \) for all \( \sigma \) with \( \sigma(j) = 1 \).  We choose \( f_0 \) so that we can do this while keeping \( Y_j \) well-defined.       

\begin{equation}\label{eq:fzero}
f_0(\pair{i}{\pair{l}{k}}) =  (\floor{\frac{k}{2^{i(l+1)}}} \bmod 2^{l+1} ) + 2^{l+1} - 1
\end{equation}            
Here modulus really denotes the remainder, i.e., the least non-negative representative of the equivalence class.  The reason for the specific form is that it will ensure that the set of \( y \) such that \( f_0(\pair{i}{y}) \) codes a desired string for \( i < n \) has positive measure.  More specifically, it will let us prove lemma \ref{lem:any-bstrs-positive-density} which will allow us to show that we always have the chance to diagonalize against computations producing effective dense descriptions of \( Z \) without making \( Y_i \) either ill-defined or violating \req{Y}{n}  

\subsubsection{Stagewise Operation}

At each stage in our construction we execute the following steps.  We understand that whenever unspecified, we are dealing with the stage \( s \) version of all objects (e.g. \( R = R_s \)) but that changes appear in the stage \( s+1 \) version.  

\begin{steps}
    \step\label{step:satN} Search for the least \( n \) such that \( (n, i, \sigma) \in R_s \) such that \( (n, i, \sigma)  \) is viable at stage \( s \).  If such a triple is found, let \( x, \tau \) be the least witness for viability \mydash ordered lexicographically as \( (x, \godelnum{\tau}) \) \mydash  and execute the following substeps.  Otherwise, continue to the next step.
    \begin{steps}
        \step\label{step:satN:fix-U-hat} \[
          \text{Set }  \hat{U}_\sigma = \mathbb{U}(\tau) \setminus U(\sigma)
        \]

        % Set (understanding \( \mathbb{U}(\sigma^{-}) = \eset \) if \( \sigma = \estr \)) \[ 
        % \hat{U}_\sigma = \mathbb{U}(\tau) \setminus \left(U(\sigma) \Union \mathbb{U}(\sigma^{-})  \right)
        % \].
          %Note that since \( \sigma \subfun \tau \) we always have \( \hat{U}_\sigma \subset  \mathbb{U}(\tau)  \) and, as this is the only step we add elements to \( \hat{U}_\sigma \), this doesn't remove any elements from \( \hat{U}_\sigma \). 

        % \step\label{step:satN:resethatU}  

        \step\label{step:satN:resetU} For all \( j \geq \lh{\sigma} \) set \( U_j = \eset \).        

        \step Set \( r(n)  \) to be \( l_s \), remove \( (n, i, \sigma) \) from \( R \) and place \( (n, i, \sigma, x) \) into \( R^{S} \).  

         \step\label{step:satN:add-injured-to-R} Remove every  \( (n', i', \sigma', x') \in R^{S} \) with \( n' > n \) and place \( (n', i', \sigma') \) into \( R \).  Also, set \( \hat{U}_{\sigma'} = \eset \).

        \step\label{step:satN:change-Z}   If \( (x, \tau) \) was a  a trivial witness continue on to the next top-level step.  Otherwise, change the value of \( Z(x) \), i.e., \( Z_{s+1}(x) = 1 - Z_s(x) \) and continue on to the next substep.

        \step\label{step:satN:add-error-to-U}  For each \( j \) with \( \sigma(j) = 0 \) enumerate \( \pair{j}{x} \) into \( U_j \).

        \step\label{step:satN:flip-elements-of-U} For each \( j \) with \( \sigma(j) = 1 \) and each \( z \in U_j \) with \( f_s(z) = \godelnum{\zeta} \)  and  \( \lh{\zeta} > \pair{\sigma\restr{j}}{x}  \) set \( f(z) = \godelnum{\zeta'} \) where \( \zeta' \) is the result of flipping the \(  \pair{\sigma\restr{j}}{x}  \)-th bit of \( \zeta \), i.e., \( \zeta'( \pair{\sigma\restr{j}}{x}) = 1 - \zeta(\pair{\sigma\restr{j}}{x}) \).  %Note that by part \ref{def:viable-tuple:pair-ls-bound} of the \hyperref[def:viable-tuple]{the definition of viability} 

    \end{steps}

    \step\label{step:satN:ext-U} Let \( \hat{Y}_i = Y_{i,s} \) unless \( U_i \) was reset in step \ref{step:satN:resetU} in which case \( \hat{Y}_i = \estr \).  Choose \( \hat{l}_s \) to be \( \pair{\str{1}^s}{l_s} \)  and for \( i \leq s \) define    \( \sigma_i \in 2^{\hat{l}_s} \) to be
            \[
                \sigma_i(x) = \begin{cases}
                                \hat{Y}_{i,s}(x) & \text{if } \hat{Y}_i(x)\conv \\
                                0 & \text{if } x \leq \hat{l}_s \land \hat{Y}_i(x)\diverge \\
                                \diverge & \text{if } x \geq l_s
                            \end{cases}
            \]
           Choose \( x_i = \pair{i}{y} > l_s \) (and hence the uses of computations in previous steps) for each \( i \leq s \) such that \( f_s(x_i) = \godelnum{\sigma_i} \) and add \( x_i \) to \( U_i \). 

    \step Add \( (s, i^\pi_s, \pi(s)) \) to \( R \).       

\end{steps}

This completes the construction.

\subsection{Verification}

We first show that the limits \( R = \lim_{s\to\infty} R_s \) and \( R^{S} = \lim_{s\to\infty} R^{S}_s  \) are well-defined.

\begin{lemma}\label{lem:finite-injury}
A triple \( (n, i, \sigma) \) can only enter (or leave) \( R \) at most \( 2^n \) times.  Moreover, \( (n, i^{\pi}_n, \pi(n)) \entersat{s+1} R \) iff \( n = s \) or \( n < s \) and some   \( (n', i', \sigma' ) \leavesat{s+1} R  \) with \( n' < n \).    
\end{lemma}
Note that, an immediate consequence of this is that \( \overbar{r}_s(n) \), which is obviously increasing in \( s \), approaches a limit for all \( n \).  
\begin{proof}
It is enough to verify the moreover claim.  To see this, note that the only way \( (n, i, \sigma) \) re-enters \( R \) after stage \( n+1 \) is via step \ref{step:satN:add-injured-to-R} which requires some \( (n', i', \sigma') \) to leave \( R \) with \( n' < n \).             
\end{proof}

The next three lemmas show that our various sets and functions are well-behaved.

\begin{lemma}\label{lem:Ui-is-re}
The sets  \( U_i \) and \( \hat{U}_\sigma \) are only reset finitely many times making \( U_i \) a well-defined r.e. set and  \( \hat{U}_\sigma \) a well-defined finite set.  
\end{lemma}
\begin{proof}
 By finite injury (lemma \ref{lem:finite-injury}) let \( s  \) be large enough that at no stage \( t \geq s \) does any \( (n, i, \tau) \) leave \(  R \) with \( \lh{\tau} \leq i \) or \( \lh{\tau} \leq \lh{\sigma} \) .  It is evident from the construction that after stage \( s \) we never change \( \hat{U}_\sigma \) and don't  reset  \( U_i \).  As the construction proceeds computably, we never remove elements from \( U_i \) after stage \( s \) and  \( \hat{U}_{\sigma, s} \) is always finite this is enough to establish the main claim. 
\end{proof}

\begin{lemma}\label{lem:fZ-defined}
Both \( f, Z \) are well-defined.  Moreover, if no triple \( (m, i, \sigma) \) with \( m \leq n  \) leaves \( R \) after stage \( s \) then \( r\restr{n+1} = r_s\restr{n+1} \), \( f_s\restr{\overbar{r}(n+1)} = f\restr{\overbar{r}(n+1)} \) and \( Z_s\restr{\overbar{r}(n+1)} = Z\restr{\overbar{r}(n+1)} \).
\end{lemma}  
\begin{proof}
By lemma \ref{lem:finite-injury} it is enough to show the moreover claim.  The part relating to \( r \) is immediate since we only change \( r(n) \) when some triple \( (n, i', \sigma') \) leaves \( R \) or re-enters \( R \) which both require some higher priority triple leave \( R \).  The only way to change \( Z \) or \( f \) is via substeps of \ref{step:satN} \mydash  substep \ref{step:satN:change-Z} for \( Z \)  substep \ref{step:satN:flip-elements-of-U} for \( f \).

If \( x \) witnesses the viability of some \( (m, i, \sigma) \) \mydash as is necessary to execute these substeps \mydash then we must have  \( x > \overbar{r}(m)  \).  This immediately establishes the claim for \( Z \) as only substep \ref{step:satN:change-Z} changes \( Z(x) \).  For \( f \) it is enough to note that substep \ref{step:satN:flip-elements-of-U} only changes \( f(z) \) when \( \lh{f(z)} > x \) (as the code of a pair exceeds the values of its elements) which by \eqref{eq:coding-bstrs} implies \( z > x \).                 
\end{proof}

\begin{lemma}\label{lem:Yi-well-defined}
The sets \( Y_{i,s} \)  are well-defined as are the sets \( Y_i = \lim_{s\to\infty} Y_{i,s} \).  %Indeed, if no tuple  \( (m, j, \sigma) \) with \( m < n \) leaves \( R \) after stage \( s \) then \( Y_{i,s}\restr{\overbar{r}(n)} = Y_{i}\restr{\overbar{r}(n)} \).     
Moreover, if \( \tau \in \bstrs, \lh{\tau} \leq s \) and  \( x \leq l_{s} \) and \( t > s \) then \( Y_{i,t}(\pair{\tau}{x})\conv \)        
\end{lemma}
\begin{proof}
For the moreover claim we note that during stage \( t -1 \) step \ref{step:satN:ext-U} ensures there is some \( z \in U_{i, t} \) with \( \lh{f(z)} > \pair{\str{1}^s}{l_s}  \).  By the properties of our coding function (see eq.  \eqref{eq:coding-bstrs}) and the monotonicity of our pairing function it follows that \( \lh{f(z)} > \pair{\tau}{x} \).  This leaves only the main claim to show. 

We first show that \( Y_{i,s} \) is well-defined.   Both steps  \ref{step:satN:ext-U} and \ref{step:satN:flip-elements-of-U} explicitly preserve compatibility.  This leaves only step \ref{step:satN:add-error-to-U} to consider.  But, if \( x \) is added to \( U_i \) in  step \ref{step:satN:add-error-to-U} then \( x \) must witness the viability of the triple under consideration and by definition  \ref{def:viable-tuple} we must have \( x \) compatible with the current value of \( Y_i \). This ensures that \( Y_{i,s} \) is well-defined.  

By the moreover claim, to verify that the limit is well-defined it is enough to show that \( Y_{i,t}(x) \) only changes value finitely many times.   But this can only happen if either \( U_i \) is emptied, which lemma \ref{lem:Ui-is-re} tells us happens only finitely many times or via the operation of step \ref{step:satN:flip-elements-of-U} which only flips values which exceed \( \overbar{r}_s{n} \geq n - 1 \) when removing a triple starting with \( n \) from \( R \).  As lemma \ref{lem:finite-injury} ensures only finite injury the conclusion is established.
\end{proof}

Now that we've verified all the objects we construct are well-defined, we start by verifying that our maps \( \mathbb{U}(X) \) and \( U(X) \) really do produce sets of \( 0 \) density.  We first note that it will be sufficient to bound the density of elements which enter \( U \) at some point.

\begin{lemma}\label{lem:hatU-takes-from-U}
If  \( x \entersat{s} \hat{U} \) then  \( x \leavesat{s} U \).   Moreover, if \( x \in \hat{U}_{\sigma, s} \) then there was some stage \( t \leq s \) at which \( x \leavesat{t} U_j \) with \( j \geq \lh{\sigma} \).    
\end{lemma}

\begin{proof}
For the main claim assume, inductively, that if \( x \entersat{s} \hat{U} \) then \( x \leavesat{s} U \). The base case is trivial since \( U, \hat{U} \) both start as empty at stage \( 0 \).  Assuming the claim holds at stage \( s \) then the only way it could fail at stage \( s + 1 \)  is via the addition of elements to  \( \hat{U} \) at step \ref{step:satN:fix-U-hat} where we set \( \hat{U}_\sigma = \mathbb{U}(\tau) \setminus U(\sigma)  \).  Note that,
\[
 \mathbb{U}(\tau) \setminus U(\sigma) \subset \hat{U}(\tau) \union \Union_{j \geq \lh{\sigma}} U_j
\]
If \( x \in \hat{U}(\tau)  \) then \( x \) doesn't  \( \hat{U} \) at stage \( s + 1 \)  and if \( x \in U_j \) with \( j \geq \lh{\sigma} \) then step \ref{step:satN:resetU} executed at stage \( s  \) ensures that \( x \leavesat{s+1} U \).  

To verify the moreover claim, it is enough to supplement the above inductive argument with the fact that if \( x \entersat{s+1} \hat{U}_\sigma \) then either \( x \in  \hat{U}_{\tau, s} \) with \( \tau \supfun \sigma \) or \( x \in U_{j, s} \) with \( j \geq \lh{\sigma} \).            

\end{proof}

While we've proved that the individual sets \( U_i, \hat{U}_\sigma \) are well-defined this doesn't necessarily entail that the operator  \( \mathbb{U}(X) \) is well-defined.  We now verify that it is.

\begin{lemma}\label{lem:bbU-well-defined}
For all \( X \subset \omega \) the operators \( U(X) \) and   \( \mathbb{U}(X) \) are well-defined and equal to their stagewise limits.  
\end{lemma}
\begin{proof}
\[
\forall(X)\left[\mathbb{U}(X) = \lim_{s\to\infty} \mathbb{U}_s(X) = \hat{U}_s(X) \union U_s(X)  \right]
\]
By lemma \ref{lem:Ui-is-re} we've established that each \( U_i \) is a well-defined r.e. set and each \( \hat{U}_{\sigma} \) is a well-defined finite set both equal to their stagewise limits.  This suffices to show that \( U(X) \) is well-defined, since \( U_i \subset \setcol{\omega}{i} \),  but to show that \( \hat{U}(X) \) is well-defined we need to verify there is no \( x \) and  sequence \( \sigma_i, s_i, s'_i \) with \( \sigma_i \subfun \sigma_{i+1} \) and  \( s_i < s'_i < s_{i+1} \) with \( x \entersat{s_i} \hat{U}_{\sigma_i} \) and \( x \nin \mathbb{U}_{s'_i}(\sigma_i) \).  As each \( U_i \subset  \setcol{\omega}{j}  \) this contradicts the moreover from lemma \ref{lem:hatU-takes-from-U}.          
\end{proof}

We now show show that very few elements are ever placed into a \( U \) or \( \hat{U} \) set.  

\begin{lemma}\label{lem:U-is-in-I}
% All of the following sets have \( 0 \) density.
% \begin{itemize}
%     \item \( \displaystyle \Union_{s \in \omega} \mathbb{U}_s  \) 
%     \item \( \displaystyle \set{x}{Z(x) \neq 0}  = \set{x}{\exists(s)\left(Z_s(x) \neq Z_0(x) \right)} \) 
%     \item jjjj
% \end{itemize} 
\[ \Union_{s \in \omega} \mathbb{U}_s  \in \mathbb{I} \] 
\end{lemma}
This obvious entails that \( \mathbb{U} \in \mathbb{I} \) as is  every set of the form \( \mathbb{U}(X) \).
\begin{proof} 
% We argue that \[
% \card{\Union_{s\in \omega} U_s(\omega)\restr{l_n}} \leq n^3
% \]
The only way elements are added to \( U \) is via steps \ref{step:satN:ext-U} and \ref{step:satN:add-error-to-U}.  At stage \( s \),  step \ref{step:satN:ext-U} adds \( s  \) elements to   \( [l_{s}, l_{s+1}) \).  Hence, for each \( s \leq n - 1 \)  at most\( s \) elements are added to \( U_\omega \) below \( l_n \) by step \ref{step:satN:ext-U}.

Now suppose that step \ref{step:satN:add-error-to-U} adds elements to \( U\restr{l_n} \) while acting on the triple \( (m, i, \sigma) \).  In this case, we add \( \card{\set{j}{ \sigma(j) = 0}} \leq \lh{\sigma} \) elements so by part \ref{lem:piii-normal-form:bound} of lemma \ref{lem:piii-normal-form} at most \( m \) elements.  Since by definition \ref{def:viable-tuple}, we only add elements above \( l_m \) we must have \( m < n \).  By lemma \ref{lem:finite-injury} we can only remove the triple \( (m, i, \sigma) \) from \( R \) at most \( 2^m \) times meaning for each \( m < n -1  \) at most \( 2^m(m+1) \) elements are added by  step \ref{step:satN:add-error-to-U} to \( U\restr{l_n} \).  As both of these values are computable functions of \( n \) independent of our choice of \( l_n \) for a sufficiently fast growing choice of \( l_n \) the claim is clear. 

% Putting these two parts together with well-known summation formula yields 
% \[
% \card{U_\omega\restr{l_n}} \leq \sum_{m =0}^{n-1} m(m+1) + \sum_{s = 0}^{n-1} s = \sum_{m =1}^{n-1} m^2 + \sum_{m = 1}^{n-1} 2m = \frac{n(n-1)(2n-2)}{6} + n(n-1)
% \]  
% Expanding this out and using the fact that \( n^3 \geq n^2 \) when \( n \geq 1 \) (and \( U(\omega)\restr{l_0} = \eset \)) we get the desired bound.  As this bound is independent of \( l_n \) it is easy to see that by choosing \( l_n \) to grow fast enough we can ensure that \( \frac{n^3}{l_n} \) goes to \( 0 \).  Thus, by lemma \ref{lem:hatU-takes-from-U} we have that \[
% \mathbb{U} \subset \hat{U} \Union U \subset \Union_{s\in \omega} U_s(\omega) \in \mathbb{I}
% \]     
% The fact that \( \set{x}{Z(x) \neq 0} \in \mathbb{I} \) follows from the same argument since we also only add at most   
\end{proof}

We prove our reductions don't return \( \Box \) too frequently. 

\begin{lemma}\label{lem:theta_sigma-sdom-one}
\[
\forall(\sigma \in \bstrs)\forall(S \in \mathbb{I})\left[S \geq_U \sigma \implies  \udensity( \set{x}{\theta_\sigma(f^{\Box S};x) = \Box}  ) = 0\right]
\] 
\end{lemma}
\begin{proof}
Assume \( S \in \mathbb{I} \) and \( S \geq_U \sigma \).   By definition \ref{def:theta-sigma} we see that \( \theta_\sigma(f^{\Box S};x) = \Box \) only when there is some \( i \) with \( \sigma(i) = 1 \) and \( \pair{i}{x} \in S \).  Thus, if \( S_i = \set{x}{\pair{i}{x} \in S} \), then by observation \ref{obs:pairing} we have \( \udensity(S_i) = 0 \).  Hence, \[
\set{x}{\theta_\sigma(f^{\Box S};x) = \Box} \subset \Union_{i < \lh{\sigma}} S_i \in \mathbb{I}
\]  
\end{proof}

Before we can say anything about how our reductions behave on their strong domain we first need to be sure they actually converge.

\begin{lemma}\label{lem:compute-Yi} 
If \( S \nsupset* U_i \) then \( f^{\Box S} \) can compute \( Y_i \) via the procedure used in \( \theta_\sigma(f^{\Box S}) \).  Moreover, if \( S, \hat{S} \geq_U \sigma \) then  \( \theta_\sigma(f^{\Box S}) \sagree \theta_\sigma(f^{\Box \hat{S}}) \) where both sides are total and \( \theta_\sigma \) is a computable functional.    
\end{lemma} 
\begin{proof}
Since lemma \ref{lem:Ui-is-re} tells us \( U_i \) is r.e. and  since \( S \nsupset* U_i \), \( f^{\Box S} \) can enumerate infinitely many elements in \( U_i \setminus S \).   By lemma \ref{lem:Yi-well-defined} for the main claim it suffices to show that \req{Y}{n} is satisfied, i.e., that \( \lh{f(z)} \) goes to infinity as \( z \in U_i \) goes to infinity.  Step \ref{step:satN:ext-U} only adds \( z \) to \( U_i \) with \( \lh{f(z)} \geq l_s \).  Step \ref{step:satN:flip-elements-of-U} leaves \( \lh{f(z)} \) for \( z \in U_i \) unchanged.  This leaves only step \ref{step:satN:add-error-to-U} to consider. However, \hyperref[def:viable-tuple]{the definition of viability} ensures that when acting on a triple \( (m, i, \sigma) \) if \( y = \pair{i}{x} \) is added to \( U_i \) then \( \lh{f(y)} \geq m \) and \hyperref[lem:finite-injury]{finite injury} ensures that there is some stage after which only triples with \( m \geq n \) leave \( R \).  

As \( \theta_\sigma \) relies on computing \( Y_i \) for the finitely many \( i \) with \( \sigma(i) = 0 \) it is easily seen to be a computable functional.  By lemma \ref{lem:Yi-well-defined} we know that the sets \( Y_i \) are well-defined ensuring that both sides are total and an examination of definition \ref{def:theta-sigma}, plus the well-definition of \( Y_i \), shows that these two computations are either equal or one of them is \( \Box \).  
\end{proof}

The next lemma is the heart of the positive argument.  It shows how we preserve agreement for \( \theta_\sigma \) even while satisfying negative requirements for incompatible strings \( \tau \). 

\begin{lemma}\label{lem:theta-sigma-preserve-agreement}
Suppose that \( \delta \in \bstrs, \lh{\delta}\leq s, x < l_{s-1} \) and no triple \( (n, j, \sigma) \) with \( \sigma \supfun \delta \) or \( \lh{\sigma} < \lh{\delta} \) leaves \( R \) at stage \( s+1 \) then \[ 
\theta_{\delta,s}(f_s^{\Box U_s(\delta)}) \boxeq Z_s(x) \implies \theta_{\delta,s+1}(f_{s+1}^{\Box U_{s+1}(\delta)}) \boxeq Z_{s+1}(x)
\] 
\end{lemma} 
\begin{proof}
Recall that \[
\theta_{\delta,s}(f_s^{\Box U_s(\delta)}) = \sum_{i = 0}^{\lh{\delta}-1} \chi^{\delta}_{i,s}(U_s(\delta),f_s;x) \bmod 2 
\]
% Expanding definition \ref{def:theta-sigma} yields the following statement to be proven (recall \( \Box + z = \Box \))
% \begin{multline*}
% \sum_{i = 0}^{\lh{\delta}-1} \chi^{\delta}_{i,s}(U_s(\delta),f_s;x) \bmod 2 \boxeq  Z_s(x) \implies \\
%  \sum_{i = 0}^{\lh{\delta}-1} \chi^{\delta}_{i,s+1}(U_{s+1}(\delta),f_{s+1};x) \bmod 2 \boxeq  Z_{s+1}(x)
% \end{multline*}
Where  \[
\chi^{\delta}_{i,s}(U_s(\delta),f_s;x) =  \begin{cases}
                        Y_{i,s}(\pair{\delta\restr{i}}{x}) & \text{if } \delta(i)\conv = 0 \\
                        0 & \text{if } \delta(i)\conv = 1 \land \pair{i}{x} \nin U_s(\delta) \\
                        \Box & \text{if }  \delta(i)\conv = 1 \land \pair{i}{x} \in U_s(\delta)
                      \end{cases}
\]  
By lemma \ref{lem:Yi-well-defined} we know that \( Y_{i,s}(\pair{\delta\restr{i}}{x}) \) and \( Y_{i,s+1}(\pair{\delta\restr{i}}{x})\) are both defined for any \( i < \lh{\delta} \).  Thus, we may assume that   \( \theta_{\delta,s}(f_s^{\Box U_s(\delta)})\conv \boxeq Z_s(x) \) and prove the consequent.
 This clearly holds if no triple leaves \( R \) at stage \( s+1 \) so we may assume that \( (n, j, \sigma) \leavesat{s+1} R \), \( \lh{\sigma} \geq \lh{\delta} \), \( \sigma \nsupfun \delta \).  For the claim to fail, we must have  either \( Y_{i,s}(\pair{\delta\restr{i}}{x}) \neq Y_{i,s+1}(\pair{\delta\restr{i}}{x}) \) for some \( i < \lh{\delta} \) or  \( Z_s(x) \neq Z_{s+1}(x) \).  Therefore, we may also assume that the witness to the viability of the tuple  \( (n, j, \sigma) \) is of the form \( x, \tau \) with \( \tau \supfun \sigma \) and therefore that \( Z_s(x) \neq Z_{s+1}(x) \).   As \( \lh{\sigma} > \lh{\delta} \) we know that for  \( i < \lh{\delta} \), \( U_i \) isn't reset at stage \( s \) so \( U_{i,s} \subset U_{i,s+1} \) and therefore if \( \theta_{\delta,s}(f_s^{\Box U_s(\delta)})  = \Box \) then \( \theta_{\delta,s+1}(f_{s+1}^{\Box U_{s+1}(\delta)}) = \Box \).   Therefore, we may also assume that \( \theta_{\delta,s}(f_s^{\Box U_s(\delta)}) \neq \Box \). 

%  Thus, it is enough to show that (assuming that adding or subtracting \( \Box \) gives \( \Box \)) \[
% \theta_{\delta,s+1}(f_{s+1}^{\Box U_{s+1}(\delta)}) - \theta_{\delta,s}(f_s^{\Box U_s(\delta)}) \boxeq Z_{s+1}(x) - Z_s(x)
%  \] 

We claim that if  \( i < \lh{\delta} \) then
 \begin{equation}\label{eq:theta-sigma-preserve-agreement:agree-agree}
\delta(i) = \sigma(i) \lor \delta\restr{i} \neq \sigma\restr{i}  \implies   \chi^{\delta}_{i,s}(U_{s}(\delta),f_s;x)  \boxeq \chi^{\delta}_{i,s+1}(U_{s+1}(\delta),f_{s+1};x)
 \end{equation}
 If \( \delta(i) = 1 \) then the claim follows from the fact that \( U_i \) isn't reset at stage \( s \) so \( U_{i, s+1} \supset U_{i,s} \).  If \( \delta(i) = 0 \) and \( \sigma(i) = 0 \) then we don't do anything in step \ref{step:satN:flip-elements-of-U} to change \( \setcol{f_s}{i} \).  Hence,  \( Y_{i,s}(\pair{\delta\restr{i}}{x}) = Y_{i,s+1}(\pair{\delta\restr{i}}{x})  \) and therefore \( \chi^{\delta}_{i,s}(U_{s}(\delta),f_s;x)  = \chi^{\delta}_{i,s+1}(U_{s+1}(\delta),f_{s+1};x) \).  This leaves only the case where \( \delta(i) = 0, \sigma(i) = 1 \) and \( \delta\restr{i} \neq \sigma\restr{i}   \).  However, step  \ref{step:satN:flip-elements-of-U} only changes  \(  Y_{i}(\pair{\sigma\restr{i}}{x})  \) which, by assumption, isn't equal to \( \delta\restr{i} \).     

 Let \( k < \lh{\delta} \) be the first location at which \( \delta(k) \neq \sigma(k) \).  If \( \delta(k) = 1 \)  and \( \sigma(k) = 0 \) then step \ref{step:satN:add-error-to-U} adds \( \pair{k}{x} \) to \( U_k \) yielding \( \chi^{\delta}_{k,s+1}(U_{s+1}(\delta),f_{s+1};x) = \Box  \).  Therefore \( \theta_{\delta,s+1}(f_{s+1}^{\Box U_{s+1}(\delta)}) = \Box \boxeq Z_{s+1}(x) \).  Hence, it is enough to prove the conclusion assuming that \( \delta(k) =0 \) and  \( \sigma(k) = 1 \).  Assuming that neither \( \theta_{\delta,s+1}(f_{s+1}^{\Box U_{s+1}(\delta)})  \) or \( \theta_{\delta,s}(f_s^{\Box U_s(\delta)}) \) are \( \Box \) \mydash either of which would suffice for victory \mydash  we can apply   \eqref{eq:theta-sigma-preserve-agreement:agree-agree} to the definition of \( \theta_{\delta} \) to yield
  \begin{multline*}
\theta_{\delta,s+1}(f_{s+1}^{\Box U_{s+1}(\delta)}) - \theta_{\delta,s}(f_s^{\Box U_s(\delta)}) \equiv \\ 
\chi^{\delta}_{i,s+1}(U_{s+1}(\delta),f_{s+1};x) -  \chi^{\delta}_{k,s}(U_{s}(\delta),f_s;x) \pmod{2}
 \end{multline*}
 As  \( \sigma(k) = 1 \) step \ref{step:satN:flip-elements-of-U} flips \( Y_k(\pair{\sigma\restr{k}}{x})  \) and since \( \sigma\restr{k} = \delta\restr{k} \) and \( \delta(k) = 0 \)  we have  \( \chi^{\delta}_{i,s+1}(U_{s+1}(\delta),f_{s+1};x) = 1 - \chi^{\delta}_{k,s}(U_{s}(\delta),f_s;x)  \).  As \( Z_{s+1}(x) = 1 - Z_s(x) \) this is enough to establish the conclusion.

\end{proof}

The next lemma verifies that the requirements \( \req{P}{\sigma} \) are satisfied.

\begin{lemma}\label{lem:P-req-met}
If \( S \in \mathbb{I} \) and \( S \geq_U \delta \) and \( \forall*(m)\left(\pi(m) \nsupfun \delta \right) \) then \( \theta_{\delta}(f^{\Box S}) \sagree*0 Z  \)     
\end{lemma} 
\begin{proof}
By \hyperref[lem:finite-injury]{finite injury} we may pick \( s_0 > \lh{\delta} \) large enough  that no triple with \( \lh{\sigma} \leq \lh{\delta} \) leaves \( R \) after stage \( s_0 \).  By assumption there are only finitely many strings \( \sigma \supfun \delta \) appearing in a tuple that ever enters \( R \) so we appeal to \hyperref[lem:finite-injury]{finite injury} again to pick \( s_1 > s_0 \) large enough that no all such tuples have permanently settled into \( R \) or \( R^{S} \).  Thus, if \( s \geq s_1 \) and  \( x < l_{s-1} \) then the assumptions  of lemma \ref{lem:theta-sigma-preserve-agreement} are met.  

For \( i < \lh{\delta} \) the first stage at which \( Y_{i,s}(\pair{\delta\restr{i}}{x})\conv \) is defined in step \ref{step:satN:ext-U} it will be \( 0 \) and it will keep that value until we change \( Z_s(x) \).   By definition \ref{def:viable-tuple} if \( x \) is a witness for viability at stage \( s \)  then \( x < s \).  Thus, if \( x \geq s \) we have \( Z_s(x) = Z_0(x) = 0 \).   As we can assume that for all \( s > 2 \) we have \( l_{s-2} \geq s  \).  Thus, for almost all \( x \) we have some \( s \geq s_1 \) with \( x < l_{s -1} \) and \( \theta_{\delta,s}(f^{\Box U_s(\delta)};x) \boxeq 0 \boxeq Z_s(x) \).  Thus, inductively applying lemma \ref{lem:theta-sigma-preserve-agreement} and appealing to the fact that in lemmas \ref{lem:fZ-defined} and \ref{lem:bbU-well-defined} we've proved that \( f,Z \) and \( U(\delta) \) are equal to their stagewise limits it is straightforward to see that \( \theta_{\delta}(f^{\Box U(\delta)}) \sagree* Z \).  

Now suppose \( S \geq_U \delta \).  As \( U(\delta) \geq_U \delta \) by lemma \ref{lem:compute-Yi} we have that  \(  \theta_\delta(f^{\Box S}) \sagree \theta_{\delta}(f^{\Box U(\delta)}) \sagree* Z \).   Finally,  lemma \ref{lem:theta_sigma-sdom-one} shows that \( \udensity(\set{x}{\theta_{\delta}(f^{\Box S};x) = \Box}) = 0 \) establishing that  \( \theta_\delta(f^{\Box S}) \sagree*0 Z \).       
\end{proof}

We can now put the pieces together to show if \( \alpha \in \kleeneO \) then  we've built a non-uniform effective dense reduction. 

\begin{lemma}\label{lem:WF-implies-nedgeq}
If \( \alpha \in \kleeneO \) then \( f \NEDgeq Z \)  
\end{lemma}  
\begin{proof}
Suppose \( S \in \mathbb{I} \) and  \( \alpha \in \kleeneO \).  Let \( X \in \cantor \) be the  unique set such that \( S \geq_U X \).  As \( \alpha \in \kleeneO \), by lemma \ref{lem:piii-normal-form}, there must be some \( \sigma \subfun X \) such that no \( \pi(n) \supfun \sigma \).  By lemma \ref{lem:P-req-met} we have that  \( \theta_{\sigma}(f^{\Box S}) \sagree*0 Z  \) showing that every effective dense description of \( f \) computes an effective dense description of \( Z \).                          
\end{proof}

This leaves us only to verify that no such reduction exists when \( \alpha \nin \kleeneO \).  The next three lemmas work together to show that if \( (n, i, \sigma) \) permanently leaves \( R \) then we preserve the resulting disagreement.

\begin{lemma}\label{lem:hatU-changes}
\[
\hat{U}_{\sigma, s} = \begin{cases}
                        \mathbb{U}_t(\tau) \setminus U_t(\sigma) & \text{if } (n, i, \sigma,x) \entersat{t+1} R^{S} \land \forall(s' \in [t+1, s])\left((n, i, \sigma,x) \in R^S_{s'}  \right) \\
                        \eset & \text{if } \forall(n, i, x)\left((n, i, \sigma,x) \nin R^S_{s}   \right)
                      \end{cases}
\]
Where we assume in the first case that \( t+1 \leq s \).
% If no quadruple \( (n, i, \sigma, x) \in R^{S}_s \) for any \( n, i, x \) then \( \hat{U}_{\sigma, s} = \eset \).  If \( (n, i, \sigma, x) \entersat{t} R^{S} \) and doesn't leave during \( [t, s] \)  then \[
% \hat{U}_{\sigma, s} = \hat{U}_{\sigma, t} =  \mathbb{U}_t(\tau) \setminus U_t(\sigma)
% \] 
% where \( x, \tau \) was the witness to viability at stage \( t \).  
\end{lemma}
\begin{proof}
 % The rightmost equality in the main claim is just the definition of \( \hat{U}_{\sigma, t}  \).  To verify \( \hat{U}_{\sigma, s} = \hat{U}_{\sigma, t}  \) we note 
This follows from the fact that the only times we modify \( \hat{U}_\sigma \) is when either removing \( (n, i, \sigma) \) from \( R \) \mydash in which case we use the top definition \mydash  or when removing some \( (m', i', \sigma') \) from \( R \) the  with \( m' < n \) \mydash in which case, by step \ref{step:satN:add-injured-to-R}, we put \( (n, i, \sigma) \) back into \( R \) and reset \( \hat{U}_\sigma \).  
\end{proof}

\begin{lemma}\label{lem:U-changes-large}
Suppose that \( (n, i, \sigma) \) permanently leaves \( R \) at stage \( s+1 \), \( \tau \supfun \sigma  \) and \( t > s \) then \( U(\tau)\restr{s} = U_t(\tau)\restr{s} =  U_s(\sigma)\restr{s}  \).  Moreover, if \( \tau \supfunneq \sigma \) then \( \hat{U}_\tau\restr(s) = \eset \).   
\end{lemma}
\begin{proof}
As step \ref{step:satN:ext-U} only adds values to \( U \) during stage \( t \)  that are larger than \( l_t \geq s \) we can ignore its impact for this claim.   By lemma \ref{lem:piii-normal-form} no \( (n', i', \sigma') \) with \( \lh{\sigma'} < \lh{\sigma} \) leaves \( R \) after stage \( s \).  Thus, we don't reset any \( U_j \) with \( j < \lh{\sigma} \) after stage \( s \).  As we reset \( U_j \) for \( j \geq \lh{\sigma} \) we have \( U_{s+1}(\tau)\restr{s} = U_s(\sigma)\restr{s}  \) for the main claim it is enough to show that after stage \( s \) step  \ref{step:satN:add-error-to-U} only adds values to \( U_j \) larger than \( s-1 \).  But if \( (n', i', \sigma') \) leaves \( R \) at stage \( t+1 > s + 1 \) then \( n' > n \) so any witnesses \( (x, \tau') \) to the viability of \( (n', i', \sigma') \) must satisfy  \( x > \overbar{r}_t(m) \geq r_{s+1}(n) = l_s \geq s \).  Therefore any values added to \( U \) after stage \( s \) must be larger than stage \( s -1 \).  For the moreover claim it is enough to further note that lemma \ref{lem:hatU-changes} guarantees \( \hat{U}_{\tau, s+1}  = \eset\) and the conclusion follows from the argument above and the guarantee from lemma \ref{lem:hatU-takes-from-U} that elements only enter \( \hat{\tau} \) from some \( U_j \) with \( j \geq \lh{\tau} \).          
\end{proof}

\begin{lemma}\label{lem:computation-preserved}
If \( (n, i, \sigma) \) permanently leaves \( R \) at stage \( s+1 \) with witnesses to viability \( x, \tau \) and \( X \supfun \sigma \) then \( \mathbb{U}(X)\restr{s} = \mathbb{U}_s(\tau)\restr{s} \) and   
\[
\Box \neq \recfnl{i}{f^{\Box \mathbb{U}(X)}}{x}\conv = \recfnl[s]{i}{f^{\Box \mathbb{U}_s(\tau)}}{x} \neq Z(x)
\] 
\end{lemma}
\begin{proof}
Suppose the assumptions of the lemma are satisfied.  We first note that it is enough to show that  \( \mathbb{U}(X)\restr{s} = \mathbb{U}_s(\tau)\restr{s} \).  As \( x, \tau \) witnesses the viability of \( (n, i, \sigma) \)  we must have \( \recfnl{i}{f^{\Box \mathbb{U}_s(\tau)}}{x}\conv[s] \neq \Box \).  By the operation of step \ref{step:satN:change-Z} we know that \( Z_{s+1}(x) \) must disagree with \( \recfnl{i}{f^{\Box \mathbb{U}_s(\tau)}}{x}\conv[s]  \) and since  \( (n, i, \sigma) \) never reenters \( R \) we must respect its restraint ensuring that \( Z_{s+1}(x) = Z(x) \).  

To verify \( \mathbb{U}(X)\restr{s} = \mathbb{U}_s(\tau)\restr{s} \), we first note that if \( \sigma' \subfunneq \sigma \) then, as no tuple including \( \sigma' \) can enter or leave \( R \) after stage \( s \),   lemma \ref{lem:hatU-changes} guarantees that \( \hat{U}_{\sigma', s} = \hat{U}_{\sigma'} \) and therefore \( \hat{U}(\sigma^{-}) = \hat{U}_s(\sigma^{-}) \).  Applying the same to \( \sigma \) after stage \( s+1 \) yields \( \hat{U}(\sigma) = \hat{U}_{s+1}(\sigma) \).  By construction we set \( \mathbb{U}_{s+1}(\sigma)\restr{s} = \mathbb{U}_s(\tau)\restr{s} \) and the desired conclusion now follows from lemmas \ref{lem:hatU-changes} and \ref{lem:U-changes-large}.     
\end{proof}

Our last major hurdle is to verify that triples which remain in \( R \) permanently result in some other kind of victory.  This will require we prove that computations whose strong domain has density \( 1 \) always give us a chance to diagonalize.  To that end we make the following definition (where \( \godelnum{\sigma_j} \) is as defined in \eqref{eq:coding-bstrs}).  

\begin{equation}\label{eq:S-simul-equal}
S^{\sigma_0, \ldots, \sigma_{n-1}} = \set{y}{\forall(j < n)\left(f_0(\pair{j}{y}) = \godelnum{\sigma_j} \right)}
\end{equation} 

We will eventually relate this set to potential witnesses to viability but first we show that it has non-zero density.

\begin{lemma}\label{lem:any-bstrs-positive-density}
If \(  \sigma_i \in 2^{l+1} \) for all  \( i < n \) then \( \udensity(S^{\sigma_0, \ldots, \sigma_{n-1}}) = 2^{-(l+1)(n+1)} \). 
\end{lemma}
\begin{proof}
% \begin{pfequations}{lemma}
Recall that we defined in  \eqref{eq:fzero}  
\begin{equation*}
f_0(\pair{i}{\pair{l}{k}}) =  (\floor{\frac{k}{2^{i(l+1)}}} \bmod 2^{l+1} ) + 2^{l+1} - 1
\end{equation*}
Fixing \( \sigma_i \in 2^{l+1} \) let  \( S_m = S^{\sigma_0, \ldots, \sigma_{m-1}} \) for \( 1 \leq m \leq n \) and \( k_i = \godelnum{\sigma_i} - 2^{l+1} + 1 \).  Note that \( 0 \leq k_i < 2^{l+1} \) since, by \eqref{eq:coding-bstrs}, \( k_i \) will be the value whose base \( 2 \) representation is \( \sigma_i \).  Combining these definitions with \eqref{eq:coding-bstrs} gives us (for all \( k \) and \( m \in [1,n] \))
\begin{equation}\label{eq:Sm-membership}
\pair{l}{k} \in S_m \iff \forall(i < m)\left(\floor{\frac{k}{2^{i(l+1)}}} \bmod 2^{l+1} = k_i  \right)
\end{equation}
Recall we write \( \bmod 2^{l+1} \) inside the equation to denote the operation which returns the representative in \( [0, 2^{l+1}) \).  Thus, when \( m \in [1, n) \)  we have 
\begin{equation}\label{eq:Smplus-membership}
\pair{l}{k} \in S_{m+1} \iff \pair{l}{k} \in S_{m} \land \floor{\frac{k}{2^{m(l+1)}}} \bmod 2^{l+1} = k_m
\end{equation}
We argue, by way of induction, that (for all \( k \) and \( m \in [1,n] \))
\begin{equation}\label{eq:any-bstrs-positive-density:kmod}
\pair{l}{k} \in S_m \iff  k \equiv  \sum_{i < m} k_i 2^{i(l+1)} \pmod{2^{(l+1)m}} 
\end{equation}
We note, for later use, that the above sum is in \(  [0, 2^{(l+1)m}) \).  When \( m = 1 \) we note that both eqs. \eqref{eq:Sm-membership}, \eqref{eq:any-bstrs-positive-density:kmod} simplify to \( \pair{l}{k} \in S_1 \) iff \( k \equiv k_0 \pmod{2^{l+1}} \).  Now suppose, the claim holds for some \( m < n \).  Unpacking \eqref{eq:Smplus-membership} we have \( y \in S_{m+1} \)  iff \( y = \pair{l}{k} \in S_m \) and for some integer  \( d \)  \[
2^{m(l+1)}k_m + d2^{m(l+1)}2^{l+1} \leq  k <  2^{m(l+1)}(k_m+1) + d2^{m(l+1)}2^{l+1} 
\]
Note that for any given \( d \), each equivalence class modulo \( 2^{(l+1)m} \) has exactly one representative in the interval above.  Thus, combining the above equation with our inductive hypothesis   we see that \( k \) satisfies these inequalities  iff 
\begin{multline*}
k = 2^{m(l+1)}k_m  + \sum_{i < m} k_i 2^{i(l+1)} +  d2^{(m+1)(l+1)}  \iff \\
k \equiv \sum_{i < m+1} k_i 2^{i(l+1)} \pmod{2^{(l+1)(m+1)}} 
\end{multline*}
This vindicates our inductive assumption and therefore implies that \( \density(\set{k}{\pair{l}{k} \in S_n}) = 2^{-(l+1)n}  \).  As it is easy to see from \eqref{eq:fzero} that \( y \in S_n \) iff \( y  \) has the form \(  \pair{l}{k} \)  it follows from observation  \ref{obs:pairing} that \[
\density(S^{\sigma_0, \ldots, \sigma_{n-1}}) = 2^{-(l + 1)}2^{-(l+1)n}   = 2^{-(l+1)(n+1)}  
\]
 % \( y \in S_n \) iff \( y = \pair{l}{k} \) and \( k \)  has a particular value modulo \( 2^{(l+1)n} \).  Thus, \( \udensity(\set{k}{\pair{l}{k} \in S}) = 2^{-(l+1)n} \) and by observations \ref{obs:density-multiply} and \ref{obs:pairing}  we have \( \udensity(S) = 2^{-(l+1)(n+1)} \) as desired.     
 % \end{pfequations} 
\end{proof}

The lemma above only tells us about \( f_0 \).  To ensure our changes to \( f \) don't cause a problem we show that we only change \( f \) on a set of density \( 0 \).    

\begin{lemma}\label{lem:f-equal-fzero-on-density-one}
For all \(n \) 
\[
\udensity(\set{y}{\exists(i < n)\exists(s)\left(f_s(\pair{i}{y}) \neq f_0(\pair{i}{y})\right)}) = 0
\]
\end{lemma}
\begin{proof}
We first note that \( \udensity(\set{x}{\exists(s)\left(f_s(x) \neq f_0(x)\right)}) = 0 \).  This holds as the only place we change \( f \) is step \ref{step:satN:flip-elements-of-U} which only changes \( f \) on elements currently in \( U \) and by lemma \ref{lem:U-is-in-I} this is a set of density \( 0 \).  By observation \ref{obs:pairing} we see that for any \( i \) the set \( \set{y}{\exists(s)\left(f_s(\pair{i}{y}) \neq f_0(\pair{i}{y})\right)} \) must also have density \( 0 \) and therefore so is a finite union of such sets.   
\end{proof}  

We can now show the desired claim about computations.

\begin{lemma}\label{lem:desc-eventually-diagonalized}
If \( (n, i, \sigma) \in R \), \( X \supfun \sigma \) then either  \( \recfnl{i}{f^{\Box \mathbb{U}(X)}}{} \) is partial or \( \udensity(\set{x}{\recfnl{i}{f^{\Box \mathbb{U}(X)}}{x} = \Box }) > 0 \).   
\end{lemma}
\begin{proof}
Suppose, for a contradiction, that  the claim fails with witnesses \( X \supfun \sigma \) and  \( (n, i, \sigma) \in R \).   For \( j < \lh{\sigma} \) let \( \xi_j = Y_j\restr{n+1} \) and, following \eqref{eq:S-simul-equal}, define 
\begin{align*}
S &= \set{y}{\forall(j < \lh{\sigma})\left(f_0(\pair{j}{y}) = \godelnum{\xi_j} \right)  } \\
\hat{S} &= \set{y}{y \in S \land  \forall(s)\forall(j< \lh{\sigma})\left(f_s(\pair{j}{x}) = f_0(\pair{j}{x})\right) \land \recfnl{i}{f^{\Box \mathbb{U}(X)}}{x}\conv \neq \Box   } 
\end{align*}
By lemma \ref{lem:any-bstrs-positive-density} we have \( \udensity(S) = 2^{-(n+1)(\lh{\sigma} +1)} \).   By our assumption and lemma \ref{lem:f-equal-fzero-on-density-one} we see that \( \hat{S} \) is the intersection of \( S \) with two sets of density \( 1 \) so by observation \ref{obs:density} we also have \( \udensity(\hat{S}) = 2^{-(n+1)(\lh{\sigma} +1)} > 0  \).  Thus, \( \hat{S} \) must be infinite.     

Now let \( x \in \hat{S} \) be larger than both \( \overbar{r}(n) \) and \( l_n \).  Let \( s_0 \) be large enough that \( \recfnl{i}{f^{\Box \mathbb{U}(X)}}{x}\conv[s_0] \).  Now let \( s_1 > s_0 \) be large enough that for all \( t \geq s_1 \) we have \(  \mathbb{U}_t(X)\restr{s_0} = \mathbb{U}(X)\restr{s_0} \) and choose \( \tau  \) such that \( \sigma \subfun \tau \subfun X \) such that \( \mathbb{U}_{s_1}(\tau)\restr{s_0} = \mathbb{U}_{s_1}(X)\restr{s_0}  \) (this must be possible as only finitely many \( \hat{U}_\tau \) are non-empty at any stage).  Note that, by lemma \ref{lem:hatU-takes-from-U} if any element left \( \mathbb{U}(\tau)\restr{s_0} \) at stage \( t \geq s_1 \) it would also leave \( \mathbb{U}_t(X)\restr{s_0} \) so we have  \( \mathbb{U}_t(\tau)\restr{s_0}= \mathbb{U}(X)\restr{s_0}  \) for all \( t \geq s_1 \).  Finally, choose \( s \) large enough that all of the following obtain for all \( j < \lh{\sigma} \) and \( t \geq s \)  
\begin{itemize}
    \item \( s_1,  n, x, \godelnum{\tau}  \) are all less than \( s \)
    \item  \( \pair{j}{\xi_j} < l_{s-1} \)
    \item  \( Y_{j,t}\restr{n+1} = Y_j\restr{n+1} \) 
    \item \( (n, i, \sigma) \in R_t \) 
    \item No triple \( (n', i', \sigma') \) with \( n' < n \) leaves \( R \) at stage \( t \).    
\end{itemize}

As \( s > s_1 > s_0 \) we have  \( \mathbb{U}_s(\tau)\restr{s_0}= \mathbb{U}(X)\restr{s_0}   \) and therefore  \( \recfnl{i}{f^{\Box \mathbb{U}_s(\tau)}}{x}\conv[s] \in \omega \).  As \( x \in \hat{S} \) we have that \( \forall(j< \lh{\sigma})\left(f_s(\pair{j}{x}) = f_0(\pair{j}{x})\right) \) and as \( x \in S \) we have that \( x \) is compatible with \( Y_j \) for \( j < \lh{\sigma} \).  Thus, by definition \ref{def:viable-tuple}, \( x, \tau \) witness that \( (n, i, \sigma) \) is a viable triple  at stage \( s \).  As no triple \( (n', i', \sigma') \) with \( n' < n \) leaves \( R \) at stage \( s \) the triple \( (n, i, \sigma) \) must leave \( R \) at stage \( s \) contradicting our assumption.          
\end{proof}

We are at last ready to complete the verification.

\begin{lemma}\label{lem:nWF-implies-not-nedgeq}
If \( \alpha \nin \kleeneO \) then \( f \nNEDgeq Z \)  
\end{lemma} 
\begin{proof}
By lemma \ref{lem:piii-normal-form} if \( \alpha \nin \kleeneO \) there is some \( X \) which extends \( \pi(n) \) for infinitely many \( n \).  We claim \( f^{\Box \mathbb{U}(X)} \) is an effective dense description of \( f \) that doesn't compute an effective dense description of \( Z \).

By lemma \ref{lem:U-is-in-I} we have that \( \mathbb{U}(X) \in \mathbb{I} \) so \( f^{\Box \mathbb{U}(X)} \) is an effective dense description of \( f \).  Now suppose, for a contradiction, that \( \recfnl{i}{f^{\Box \mathbb{U}(X)}}{} \) is total and \( \recfnl{i}{f^{\Box \mathbb{U}(X)}}{} \sagree0 Z \). Since \( X \) extends \( \pi(n) \) for infinitely many \( n \) and \( \pi(n) \) is injective for some \( n \) we have \( i^{\pi}_n = i \).  If \( \sigma = \pi(n) \) then at stage \( n+1 \) we enumerate \( (n, i, \sigma) \) into \( R \).  By lemma \ref{lem:finite-injury},  \( (n, i, \sigma) \) eventually either leaves \( R \) permanently or stays in \( R \) permanently.  

In the former case, lemma \ref{lem:computation-preserved} ensures that \( \recfnl{i}{f^{\Box \mathbb{U}(X)}}{} \) disagrees with \( Z \).  In the later case, lemma \ref{lem:desc-eventually-diagonalized} ensures that \( \recfnl{i}{f^{\Box \mathbb{U}(X)}}{} \) is either partial or the strong domain of \( \recfnl{i}{f^{\Box \mathbb{U}(X)}}{}   \) doesn't have density \( 1 \).  This provides the desired contradiction proving the lemma.              
\end{proof}

This, plus the observation that this construction relativizes, shows that can uniformly produce \( f, Z \in \deltazn(X){2} \) such that \( f \NEDgeq Z \) iff \( P(X) \).  To finish the proof of theorem \ref{thm:ned-piii-complete} it is enough to prove the following lemma.

\begin{lemma}\label{lem:ned-construction-non-uniform}
When the \( \piin{1} \) predicate \( P(X) \) holds we have \( f \NEDgeq Z \) but not \( f \nUEDgeq Z \).  Indeed, there is a strictly monotonic computable function \( l_k \) such that there is no computable functional \( \Gamma \) satisfying \[
\forall(S \subset \omega)\left(\forall(k)\left[\udensity[l_k](S) \leq 2^{-k} \right] \implies \Gamma(f^{\Box S}) \in (\omega \union \set{\Box})^{\omega} \land   \Gamma(f^{\Box S}) \sagree0 Z  \right)
\] 
\end{lemma}
\begin{proof}
It is enough to prove the indeed claim.  It is clear that we can define a computable monotonic sequence \( l_k \) such that \( \udensity[l_k](\mathbb{U}) \leq 2^{-k} \).  We now prove that this sequence witnesses the truth of the lemma.   Given \( i \) by the proof of lemma \ref{lem:piii-normal-form} there will be some \( \sigma = \str{0}^k\concat[1] \) such that for some \( n \)  we enumerate \( (n, i, \sigma) \) into \( R \).  If we set \( S = \mathbb{U}(\sigma) \) then, as \( S \subset \mathbb{U} \), we have  \( \forall(k)\left[\udensity[l_k](S) \leq 2^{-k} \right] \).   Applying lemmas \ref{lem:finite-injury}, \ref{lem:computation-preserved} and \ref{lem:desc-eventually-diagonalized} entails that \( \recfnl{i}{f^{\Box S}}{} \) isn't an effective dense description of \( Z \). 
\end{proof}

\section{Effective Dense Degrees Without Sets}\label{sec:edd-no-sets}

In this section we prove the following theorem.  

\begin{theorem}\label{thm:effective-dense}
If  \( f \in \baire \) is \( 3 \)-generic then there is no set \( X \) such that \( f \NEDequiv X \).  
\end{theorem}

Ideally, this result could be proved by simply showing that \( f \) forces that every candidate \( X \in \cantor \) isn't non-uniformly effective dense equivalent to \( f \).  However, unlike coarse reducibility, there are going to be some functions \( g \) with \( f \NEDequiv g \) but with \( g \) extremely non-definable from \( f \) \mydash if \( V \neq L \) we could take \( f  \) to be the \( 0 \) function and \( g \) to be the \( 0 \) function except on a computable set of density \( 0 \) but with \( g \nin L \).  It is tempting to think that this kind of pathology can only occur by taking a well-behaved element of the equivalence class and modifying it on a set of density \( 0 \).  If so, we could exclude consideration of such highly non-definable sets \( X \) by arguing that if \( f \NEDequiv Y \) with \( Y \subset \omega \) then there will be some hyperarithmetic in \( f \)  set \( X \) with \( f \NEDequiv X \).  While we still hypothesize this is true, we have been unable to prove this (see question \ref{q:ned-definable-set-element}).  Indeed, theorem \ref{thm:ned-piii-complete} grew out of a a failed attempt to prove this claim so we never tried to address the complexity of effective dense reducibility on \( \cantor \).  However, the important difference we observe in this section between the behavior of effective dense reductions on  \( \baire \) and \( \cantor \) raises doubts about whether effective dense reducibility remains \( \piin{1} \) complete on \( \cantor \).  The fact that we can't assume we must have an \( f \)-definable set \( X \) witnessing \( f \NEDequiv X \) if there is any such set will render the proof of theorem \ref{thm:effective-dense} more complicated than the basic intuition might suggest so we can avoid quantifying over \( X \).

\def\e{\underline{e}}
\def\i{\underline{i}}
\def\j{\underline{j}}

However,  the basic intuition is relatively simple.  At the highest possible level the idea is that if we had some description of \( X \) in terms of \( f \) and we changed \( f \) on some finite interval \( F \) and \( U \) covered all the places \( X \) changed as a result then \( X^{\Box U} \) couldn't tell the difference between these versions of \( f \).   If we knew that \( \recfnl{\i}{f^{\Box S}}{} \) was guaranteed to compute an effective dense description of \( X \) then we could use this to find a specific density \( 0 \) set \( U_i \) which covered the changes to \( X \) as a result of modifying \( f \) on \( F_i \) \mydash roughly, we will argue that if we set \( S = F_i \) and look at the locations  \( \recfnl{\i}{f^{\Box S}}{} \) is \( \Box \) this set must cover any location that \( X \) could change as a result of changing \( f \) on \( F_i \).  If we ensured that each interval \( F_i \) was large enough we could ensure that \( X^{\Box U} \) with \( U = \Union U_i  \) didn't compute an effective dense description of \( f \) since we could ensure that any such computation was either incorrect on some \( F_i \) \mydash pick the version of \( f \) which disagrees with the computation on \( F_i \) \mydash or returned \( \Box \) on almost all of the intervals \( F_i \) violating the density condition.  However, to ensure \( f \nNEDequiv X \) we need \( F = \Union F_i \) to have positive density but \( U =  \Union U_i \) to have \( 0 \) density \mydash hence why this fails if \( f = X \).  The key idea we use is that for some sufficiently large \( l_i \) we can ensure that \( \Union U_i\restr{[l_i, \infty)} \) has density \( 0 \) and since there are infinitely many values \( f \) can taken on \( F_i \) we can use the pigeonhole principle to find versions of \( f \), \( f^y, f^z \)  which disagree on \( F_i \) but which ensure \( X^y, X^z \) agree up to \( l_i \).    

The role of forcing will be to let us assume we have specific computations that are guaranteed to produce effective dense descriptions.  We now sketch a simplified version of this argument in a bit more detail.  Suppose we build \( f \) and \( S, U \in \mathbb{I} \) via finite extensions to be sufficiently generic \mydash in \S\ref{ssec:effective-dense-conditions} we will supplement Cohen conditions with density bounds formally define generic elements of \( \mathbb{I} \).  If we could assume that \( X \) was sufficiently definable from \( f \) then we could assume that if \( f \NEDequiv X \) then some finite condition \( q_0 \)  forced \( \recfnl{\i}{f^{\Box S}}{} \) to be an effective dense description of \( X \) and \( \recfnl{\j}{X^{\Box U}}{} \) to be an effective dense description of \( f \).  For simplicity, assume that \( q_0 \) is actually the empty condition.

For some condition \( q_1  \) extending \( q_0 \) and some \( l_0 \)  we must have \[ 
q_1 \forces \udensity[l_0](\set{y}{\recfnl{\j}{X^{\Box U}}{y} = \Box}  ) \leq \frac{1}{4} 
\] 
We can assume that \( l_0 \geq \lh{\tau_1} \) where \( \tau_1 \) is the \( f \) component of \( q_1 \) and define  \( F = [l_0, 2l_0) \).  Using the fact that \( q_0 \) forces that  \( \recfnl{\i}{f^{\Box S}}{} \) is an effective dense description of \( X \) we can argue that there is some \( l_1 \) such that changing \( f \) on \( F \) only changes \( \recfnl{\i}{f^{\Box S}}{} \) on a set whose density above \( l_1 \) is much smaller than the bound \( q_1 \) places on \( U \).  We can now argue that, as we have infinitely many options for \( f \) on \( F \),  we can find two versions of some extension \( q_2 \), \( q^y_2 \) and \( q^z_2 \), which differ only in that \( \tau^y_2 \) is \( y \) on \( F \) and \( \tau^z_2 \) is \( z \neq y \) on \( F \), give equal values for   \( \recfnl{\i}{f^{\Box S}}{}\restr{l_1} \) and above \( l_1 \) cover any potential disagreement between the two versions of this computation with \( U \).  If \( X^y, X^z \) refer to the corresponding versions of \( X \) we have that \( X^y \symdiff X^z \subset U \) \mydash at least on the part of \( U \) specified by \( q^y_2 \) and \( q^z_2 \).   As a result, if either \( q^y_2 \) or \( q^z_2 \) ensured \( \recfnl{\j}{X^{\Box U}}{y}\conv \in \omega \) for some  \( y \in F \) we could switch to the other condition without altering the computation.  On the other hand, if either condition forced \( \recfnl{\j}{X^{\Box U}}{} \) to be \( \Box \)  on \( F \)  this would violated the assumption we made about \( q_1 \) limiting the upper density of this set above \( l_0 \). 

Unfortunately, this argument simplifies the actual situation in a number of ways.  First, rather than taking \( F = [l_0, 2l_0) \) it will need to have density (below it's endpoint) that is below a bound imposed by \( q_0 \) and this will require some bookkeeping.  More seriously, we have the issue mentioned above that we can't assume we have some definition of \( X \) in terms of \( f \).  We deal with this by noting that if \( f \NEDequiv X \) for any \( X \) we can assume that we have some \( \e \) and \( \i \) so that \( X_{\e} \eqdef \recfnl{\e}{f}{} \) and \( \recfnl{\i}{f^{\Box S}}{}  \) are both forced to have strong domain of density \( 1 \) and agree on that strong domain.  Since the union of two sets of density \( 0 \) have density \( 0 \) we can use \( X^{\Box U}_{\e} \) in place of \( X^{\Box U} \) with a bit of extra bookkeeping.  Finally, there is the issue that even if \( q_0 \) forces \( \recfnl{\j}{X_{\j}^{\Box U}}{} \) to be total \mydash as forcing only equals truth on a dense set \mydash it may refuse to converge as long as \( U \) contains all the locations above \( l_1 \) on which changes to \( f \) on \( F \) can affect \( X_{\e} \).  We solve this by considering conditions for \( U \) that not only commit to extending some finite initial segment and obey some density bound but also commit to containing some finite number of \( f \Tplus S \) computable sets.  Unfortunately, this will require a fair bit of bookkeeping as we can only commit to \( U \) containing a particular \( f \Tplus S \) computable set if we've forced that set to be compatible with the density bound.

Before we present the full proof, we briefly remark on why this argument is successful for effective dense reducibility despite the fact that we will see in   \S\ref{sec:coarse-sets} that every uniform coarse degree contains a set.  After all, the uniform functional \( \Gamma\maps{\baire}{\cantor} \)  we produce in that section satisfying \( \Gamma(f)  \UCequiv f \) will have the property that changing \( f \) on a finite set \( F \) results in a change to \( \Gamma(f) \) on a set of density \( 0 \).  The key difference is that for effective dense reducibility we can consider \( f^{\Box F} \) and argue that we have a \textit{single} set \( U \) of density \( 0 \) which covers \( \Gamma(f^y) \symdiff \Gamma(f^z) \) for \textit{any} values \( y, z \) that we might change \( f(x) \) to on \( F \).  In contrast, the reduction \( \Gamma \)  we construct in \S\ref{sec:coarse-sets} merely ensures that for any particular values of \( y,z \)  we have  \( \Gamma(f^y) \symdiff \Gamma(f^z) \) is density \( 0 \) while \( \Union_{z \in \omega}  \Gamma(f^y) \symdiff \Gamma(f^z)   \) will be a set of positive density.  Indeed, we were lead to the solution in \S\ref{sec:coarse-sets} by considering the kind of forcing argument given in this section and noticing that we needed to, but couldn't, ensure that the set of \( x \) in \( \Gamma(f') \symdiff \Gamma(f) \) for some \( f' \agree[F] f \) had \( 0 \) density.

\subsection{Zero Density Conditions}\label{ssec:effective-dense-conditions}

We now define modified Cohen conditions to produce sets of density \( 0 \).

\begin{definition}\label{def:zero-density-condition}
\( \Icond \) is the set of pairs \( \sigma, k \) with \( k \in \omega \), \( \sigma \in \bstrs \) such that  \( \udensity[\lh{\sigma}][\lh{\sigma}](\sigma) \leq 2^{-k} \).  If \( p = (\sigma, k), \hat{p} = (\hat{\sigma}, \hat{k}) \) then
\begin{align*}
&p \Ileq \hat{p} \iffdef p, \hat{p} \in \Icond \land \hat{k} \geq k \land \hat{\sigma} \supfun \sigma \land \udensity[\lh{\sigma}][\lh{\hat{\sigma}}](\hat{\sigma}) \leq 2^{-k} \\
&S \Igeq p \iffdef p \in \Icond \land S \subset \omega \land \sigma \subfun S \land \udensity[\lh{\sigma}](S) \leq 2^{-k} \\
&\Pcond \eqdef \set{(\tau, p)}{\tau \in \wstrs \land p \in \Icond} \\
&(\tau, p) \Pleq (\hat{\tau}, \hat{p}) \iffdef \tau \subfun \hat{\tau} \land p \Ileq \hat{p} \\
&(f, S) \Pgeq (\tau, p) \iffdef f \supfun \tau \land S \Igeq p
\end{align*} 
\end{definition}

As every condition \( p_i = (\sigma_i, k_i) \in \Icond \) can be extended with \( p_{i+1} \Igeq p_i \) and \( k_{i+1} > k_i \) a set \( S \) that is generic with respect to \( \Icond \) \mydash that is satisfies \( S \Igeq p_i \) for all \( i \) in some generic sequence \mydash will be in \( \mathbb{I} \).  Thus, the forcing notion associated with \( \Pcond \) produces a pair \( f, S \) with \( f \in \baire \) and \( S \in \mathbb{I} \).  We observe that interpreting \( \hat{\sigma} \Igeq p \) consistent with our practice of identifying sets with their characteristic functions means  \( \hat{\sigma} \Igeq p \) iff \( (\hat{\sigma}, k) \Igeq p \).   We will drop unnecessary parentheses writing conditions, e.g., we write \( (\tau, \sigma, k) \) rather than \( (\tau, (\sigma, k)) \), and will adopt the convention that, unless specifically stated otherwise, that accents, subscripts and superscripts are inherited by the (standardly denoted) components of a condition, e.g., the condition \( \hat{p}_0 \) has components \( (\hat{\tau}_0, \hat{\sigma}_0, \hat{k}_0) \).

We define forcing and genericity for conditions in \( \wstrs \) or \( \Pcond \) standardly (e.g., see \cite{Odifreddi1999Classical}) with \( f \) and \( f, S \) the designated constant symbols for \( \wstrs \) and \( \Pcond \) respectively.  We tweak the standard definition slightly in that we assume  \( \tau \forces \exists(x)\psi(f,x) \) requires \( \tau \forces \psi(f,x) \) for some \( x < \lh{\tau} \) and similarly for our other forcing notions\footnote{For \( \Pcond \) we will require \( x < \min(\lh{\tau}, \lh{\sigma}) \) and for \( \Qcond \) \( x < \min(\lh{\tau}, \lh{\sigma}, \lh{\xi}) \).  However, we will assume that if \( S \) isn't mentioned in \( \psi(f) \)  then \( (\tau, p) \forces \psi(f) \) iff \( \tau \forces \psi(f) \) and likewise for \( U \) and \( \Pcond \).}. This change improves the normal result about the complexity of forcing to make the forcing relation for \( \sigmazn(f){n+1} \) sentences \( \pizn{n} \) (or \( \deltazn{1} \) when \( n = 1 \)).  However, the primary advantage isn't so much to improve the complexity of forcing but in conjunction with the convention we adopt that we can treat conditions like partial oracles in computations when context makes clear what part of the condition is being consulted.  For instance, if in context we are discussing \( \recfnl{\i}{f^{\Box S}}{} \) and \( p = (\tau, \sigma, k) \in \Pcond \) we will simply write \( \recfnl{\i}{p}{} \) in place of \( \recfnl{\i}{\tau^{\Box \sigma}}{} \).  In conjunction with this convention this rule for existential forcing allows us to assume that a computations are forced by conditions iff the computation applied to the condition converges, e.g., \( p \forces \recfnl{\i}{f^{\Box S}}{x} = y  \) iff \( \recfnl{\i}{p}{x}\conv = y \).

% To avoid unnecessary verbosity we will adopt the convention that we will replace oracles with conditions when context makes clear what part of the condition is being consulted.  For instance, if in context we are discussing \( \recfnl{\i}{f^{\Box S}}{} \) and \( p = (\tau, \sigma, k) \in \Pcond \) we will simply write \( \recfnl{\i}{p}{} \) in place of \( \recfnl{\i}{\tau^{\Box \sigma}}{} \).  The convention we adopted above that existential claims require the witness to be bounded by an element of the condition allow us to assume that a computations are forced by conditions iff the computation applied to the condition converges, e.g., \( p \forces \recfnl{\i}{f^{\Box S}}{x} = y  \) iff \( \recfnl{\i}{p}{x}\conv = y \). 

It is easy to check that all the usual results about forcing and the complexity of the forcing relation  with \( \wstrs \) translate\footnote{The only material difference that the inclusion of density bounds in \( \Icond \) makes to the usual results about forcing and genericity is that we can no longer assume that a set \( S \in \mathbb{I} \) is associated with a unique chain in \( \Icond \).  However, as we will only build generic sets using \( \Icond \) and \( \Pcond \) and won't try to characterize the class of sets that are generic with respect to these notions this issue won't matter.}. to forcing with  \( \Icond \) and \( \Pcond \) and that the resulting forcing notions satisfy  monotonicity, consistency and quasi-completeness (see see \cite{Odifreddi1999Classical}).   We also observe that  the pair \( (f, S) \) will the \( 3 \)-generic with respect to \( \Pcond \) iff they force  every \( \sigmazn(f, S){3} \) sentence or it's negation  iff they extend an infinite chain \( p_i \in \Pcond \) which meets or strongly avoids every  \( \sigmazn{3} \) of \( \Pcond \).  This implies that if \( f \) is \( 3 \)-generic then the standard argument regarding iterated/mutual genericity proves that there is some \( S \) such that the pair \( (f, S) \) is \( 3 \)-generic with respect to \( \Pcond \).

We also want to build \( U \in \mathbb{I} \) but, as discussed in the introduction to this section, we need to be able to commit to \( U \) containing certain zero density sets computable in \( f \Tplus S \).   To that end, we offer a modified version of the conditions \( \Icond \) which allows us to commit to \( U \)  containing \( 0 \) density \( X \) computable sets.  To express this definition we adopt the convention that when \( E \) is a finite set \( \recset(X){E} \) is defined by    \[
\recset(X){E} =  \Union_{e \in E} \recset(X){e} \qquad \recset(X){E}(x) \eqdef  \begin{cases}
                            \diverge & \text{unless } \forall(e \in E)\left(\recset(X){e}(x)\conv \right) \\
                            1 & \text{if } \exists(e \in E)\left(\recset(X){e}(x)\conv = 1 \right) \\
                            0 & \text{if } \forall(e \in E)\left(\recset(X){e}(x)\conv = 0 \right) 
                        \end{cases}  
\]
where we recall \( \recset(X){e} \) is just a suggestive way to write the \( e \)-th \( 0-1 \) valued functional to which we extend our convention about totality, i.e.,  \( \recset(X){E}\conv \) means \( \forall(y)\left(\recset(X){E}(x)\conv \right) \).     

 \begin{definition}\label{def:I-cond-star}
Assuming \( r =  (\xi, w, E)  \) and \( r' =  (\xi', w', E') \)  then  
\begin{align*}
&\Icond* \eqdef \set{(p, r)}{p \in \Icond \land E \subset \omega \land \card{E} < \omega } \\
&\Icond[X] \eqdef \set{r}{ r \in \Icond* \land \recset(X){E}\conv \land \udensity(\recset(X){E}) = 0 \land \udensity[\lh{\xi}]({ \xi \union  \recset(X){E}\restr{[\lh{\xi}, \infty)} }) \leq 2^{-w}  } \\
&r \Ileq[X] r' \iffdef r, r' \in \Icond[X]  \land (\xi, w) \Ileq (\xi', w') \land E \subset E' \land \phantom{X} \\
& \qquad \qquad \forall(y < \lh{\xi'})\Bigl(y \geq \lh{\xi} \land \recset(X){E}(y) = 1 \implies \xi'(y) = 1 \Bigr) \\
% & \qquad \qquad \Bigl({\forall y \in [\lh{\xi}, \lh{\xi'})}\Bigr)\Bigl(\recset(X){E}(y)\conv \land \recset(X){E}(y) = 1 \implies \xi'(y) = 1 \Bigr) \\
&U \Igeq[X] (\xi, w, E) \iffdef U \Igeq (\xi, w) \land U \supset \recset(X){E}\restr{[\lh{\xi}, \infty)}
\end{align*}  
\end{definition}

We clarify that when discussing a set \( Y \) that is a r.e. or computable relative to \( X \) (such as \( \recset(X){E} \) above) we understand \( \udensity(Y) = 0 \) to abbreviate the following \( \pizn(X){3} \) sentence   
\begin{equation}\label{eq:density-for-re}
\forall(k)\exists(l)\forall(l' > l)\forall(s)\left( \frac{\card{Y_s\restr{l'}}}{l'} \leq 2^{-k}  \right)
\end{equation}  
It is easy to see that, as with \( \Icond \), any sequence of conditions in \( \Icond[X] \) which meet the dense set of conditions \( \set{(\xi, w, E)}{(\xi, w, E) \in \Icond[X] \land w \geq n} \) for every \( n \) define a set \( U \in \mathbb{I} \).  It will be useful to adopt the viewpoint of building \( f, S, U \) simultaneously therefore we make the following definition.

\begin{definition}\label{def:q-condition}
\begin{align*}
&\Qcond \eqdef \left\{(p, r) \mid p \in \Pcond \land r = (\xi, w, E) \in \Icond* \land p \forces r \in \Icond[f \Tplus S] \land p \forces \lh{\recset(f \Tplus S){E}} \geq \lh{\xi}   \right\} \\
% & \qquad \qquad \left. w \leq k \land \lh{\xi} \leq \lh{\sigma} \right\} \\
&(p, r) \Qleq (p', r') \iffdef (p, r), (p', r') \in \Qcond \land p \Pleq p' \land p' \forces r \Ileq[f \Tplus S] r'
\end{align*} 
\end{definition}  

This definition reflects the usual way we can turn iterated forcing into simultaneous forcing by considering pairs of conditions where the first condition forces the appropriate relationship between the second conditions relative to the object built by the first condition. This has only been modified to ensure that \( p \) must specifically force the convergence of \( \recset(f \Tplus S){E} \) on the domain of \( \xi \) not merely be densely extended by conditions which do so.    We will now unpack this definition into a slightly more useful form which we will effectively treat as the definition for \( \Qcond \) going forward.  

% Essentially, \( \Qcond \) are the pairs \( (p, r) \) with \( p \in \Pcond \) such that \( p \) forces \( r \in \Icond[f \Tplus S] \) and extension is defined similarly.  The definition above only adds the additional requirement that the condition \( r \) must `lag behind' the condition for \( S \) and makes explicit the form \( r \) takes to avoid any confusion, e.g., no replacing \( r \) with some \( f \) computable function evaluating to a condition in \( \Icond[f \Tplus S] \).  We make the following observation.

\begin{lemma}\label{lem:q-condition-equiv}
Suppose that  \( r =  (\xi, w, E) \in \Icond* \) and \( p =  (\tau, \sigma, k) \in \Pcond \) then \( (p, r) \in \Qcond \) iff all of the following obtain 
% ,  \( w \leq k \), \( \lh{\sigma} \leq \lh{\xi} \)
\begin{enumerate}
    \item\label{lem:q-condition-equiv:length} \( \displaystyle p \forces \lh{\recset(f \Tplus S){E}} \geq \lh{\xi}  \)  
    \item\label{lem:q-condition-equiv:sets-total}   \( \displaystyle p \forces \forall(y)\exists(s)\recset(f \Tplus S){E}(y)\conv[s] \)
    \item\label{lem:q-condition-equiv:respects-w}  \( \displaystyle p \forces \udensity[\lh{\xi}]({ \xi \union \recset(f \Tplus S){E}\restr{[\lh{\xi}, \infty)} }) \leq 2^{-w} \)
    \item\label{def:q-condition:density-zero}  \( \displaystyle p \forces \udensity({  \recset(f \Tplus S){E}  }) = 0 \)
\end{enumerate} 
Furthermore, if \( (p,r), (p', r') \in \Qcond \) then
\begin{equation}\label{eq:q-extension-equiv}
\begin{aligned}
(p, r) \Qleq (p', r') \iff& p \Pleq p' \land (\xi, w) \Ileq (\xi', w') \land E \subset E' \land \\
& \forall(y < \lh{\xi'})\left(y \geq \lh{\xi} \land \xi'(y) = 0 \implies \recset(p'){E}(y)\conv = 0  \right)
\end{aligned}
\end{equation}
\end{lemma} 
\begin{proof}
This is immediate from unpacking the definition of forcing and definition \ref{def:q-condition} given that we treat bounded quantification as if it was either conjunction or disjunction and that our convention ensures that computations are forced to converge just if the condition witnesses the convergence.  
\end{proof}      

We will also use \( \Qcond \) to define a forcing relation via the standard inductive definition (requiring existential witnesses to be below \( \lh{\xi} \)), using \( U \) as the constant symbol for the set built by \( \xi \).  However, unlike \( \Pcond \), it is not clear (indeed it is likely false) that  meeting or strongly avoiding every \( \sigmazn{3} \) set of conditions is enough to guarantee that every \( \sigmazn(f, S, U){3} \) sentence or its negation is forced.  This will not pose any problems for us as we will build \( U \) as the limit of an explicitly constructed sequence of conditions and we will directly prove the complexity of forcing for the particular cases our construction depends on.  Before we start proving the utility lemmas about forcing that we will need we introduce one last piece of notation that will help us talk about modifying conditions on a set \( F \).

\begin{definition}\label{def:condition-agree}
For \( p, p' \in \Pcond \) or \( q ,q' \in \Qcond \) we define
\begin{align*}
p \agree[F] p' &\iffdef \tau \agree[F] \tau' \land \sigma = \sigma' \land k = k' \\
p' \Pgeq[F] p &\iffdef \exists(p'')\left(p'' \Pgeq p \land p' \agree[F] p''  \right) \\
q \agree[F] q' &\iffdef  p \agree[F] p' \land \xi = \xi' \land w = w' \land E = E' \\
q' \Qgeq[F] q &\iffdef \exists(q'')\left(q'' \Qgeq q \land q' \agree[F] q''  \right) \\
\end{align*} 
\end{definition}
Recall that \( \tau \agree[F] \tau' \) held just if \( \tau\restr{\setcmp{F}} = \tau'\restr{\setcmp{F}} \) and \( \lh{\tau} = \lh{\tau'} \).  Thus, \( q \agree[F] q' \) if the only difference between \( q \) and \( q' \) is a possible change of values for \( f \) on \( F \) and \( q' \Qgeq q \) only if \( q' \) agrees modulo \( F \) with some extension of \( q \).

\subsection{Forcing Lemmas}

In this subsection we prove a number of helpful results about our conditions and the forcing relation for \( \Pcond \).  We start by proving some results about the complexity of forcing over \( \Qcond \).

\begin{lemma}\label{lem:q-forcing-lvl-one}
If we understand the claim that \( A \in \Gamma \) assuming \( P(z) \) to mean that there is some \( B \in \Gamma \) such that for all \( z \) satisfying \( P(z) \) we have \( A(z) \iff B(z) \) then all of the following obtain     
\begin{enumerate}
    \item\label{lem:q-forcing-lvl-one:qcond} \( \Qcond   \) is \( \pizn{3} \) 
    \item\label{lem:q-forcing-lvl-one:qgeq}  \( q \Qleq q'  \) is \(  \pizn{3} \)
    \item\label{lem:q-forcing-lvl-one:qcond-assume-forces} Assuming \( p \in \Pcond \) forces both   \(  \udensity(\recset(f \Tplus S){E}) = 0 \) and \( \forall(x)\exists(s)\recset(f \Tplus S){E}\conv[s] \) then \( (p, r) \in \Qcond  \)  is \(  \pizn{1} \)
    \item\label{lem:q-forcing-lvl-one:qgeq-assuming} Assuming \( (p, \xi, w, E) \in \Qcond \) then \( (p', \xi', w', E) \Qgeq (p, \xi, w, E)  \) is \( \pizn{1} \). 
    \item\label{lem:q-forcing-lvl-one:force-sigmaone} Assuming \( q \in \Qcond \), the relation   \( q \forces \psi(f, S, U)  \)  is \( \deltazn{1} \) for \( \psi \in \sigmazn{1} \)
    \item\label{lem:q-forcing-lvl-one:force-pione} Assuming \( q \in \Qcond \), the relation   \( q \forces \psi(f, S, U)  \) is \( \pizn{2} \) for \( \psi \in \pizn{1} \).
\end{enumerate} 
  
\end{lemma} 
\begin{proof}
% \begin{pfequations}{lemma}
Part \ref{lem:q-forcing-lvl-one:qcond} follows from lemma \ref{lem:q-condition-equiv} and the complexity of forcing for \( \Pcond \) \mydash where we understand  \( \udensity({  \recset(f \Tplus S){e}  }) = 0  \) as per equation \eqref{eq:density-for-re}.  \( q \Qleq q' \) holds iff  \eqref{eq:q-extension-equiv}, \( q \in \Qcond \) and \( q' \in \Qcond \) so by part \ref{lem:q-forcing-lvl-one:qcond} the relation  \( q \Qleq q' \) is  \( \pizn{3} \) establishing part \ref{lem:q-forcing-lvl-one:qgeq}.  Part \ref{lem:q-forcing-lvl-one:force-sigmaone} follows directly from the definition for forcing for existential sentences given that we adopt the convention that to force an existential claim the existential witnesses must be bounded by \( \lh{\xi} \). 

For part \ref{lem:q-forcing-lvl-one:qcond-assume-forces}, by lemma \ref{lem:q-condition-equiv} and our assumptions, it is enough to observe that  \( p \forces \lh{\recset(f \Tplus S){E}} \geq \lh{\xi}  \) is \( \deltazn{1} \) (recall our convention for forcing existential claims)  and that the relation on the left-hand side of the following observation is \( \pizn{1} \).  
\begin{multline*}
p \forces \udensity[\lh{\xi}]({ \xi \union \recset(f \Tplus S){E}\restr{[\lh{\xi}, \infty)} }) \leq 2^{-w} \iff \\
\forall(p' \Pgeq p)\forall(l\geq \lh{\xi})\forall(s)\left(p' \nforces  \udensity[l][l]({ \xi \union \recset(f \Tplus S)[s]{E}\restr{[\lh{\xi}, l)} }) > 2^{-w}     \right)
\end{multline*}
By complexity of forcing for \( \Pcond \) the right-hand side above is \( \pizn{1} \) establishing part \ref{lem:q-forcing-lvl-one:qcond-assume-forces}.  For part \ref{lem:q-forcing-lvl-one:qgeq-assuming} we observe that by \eqref{eq:q-extension-equiv} 
\begin{align*}
(p', \xi', w', E) \Qgeq (p, \xi, w, E) &\iff (p', \xi', w', E) \in \Qcond \land p \Pleq p' \land (\xi, w) \Ileq (\xi', w') \land \phantom{X} \\
& \forall(y < \lh{\xi'})\left(y \geq \lh{\xi} \land \xi'(y) = 0 \implies \recset(p'){E}(y)\conv = 0  \right)
\end{align*}
To prove the right-hand side is \(  \pizn{1} \) it suffices to show that \( (p', \xi', w', E) \in \Qcond \) can be decided by a \( \pizn{1} \) question.  Since we are assuming  \( (p, \xi, w, E) \in \Qcond \) we have that \( p \) and hence  \( p' \) forces both  \(  \udensity(\recset(f \Tplus S){E}) = 0 \) and the totality of \( \recset(f \Tplus S){E} \).   By part \ref{lem:q-forcing-lvl-one:qcond-assume-forces}, assuming those two facts,  \( (p', \xi', w', E) \in \Qcond \)  is  \(  \pizn{1} \).  This establishes part \ref{lem:q-forcing-lvl-one:qgeq-assuming}.   Finally, for part \ref{lem:q-forcing-lvl-one:force-pione} we note that since \( E \) plays no role in the forcing of \( \sigmazn(f, S, U){1} \) formulas we have 
\begin{align*}\label{eq:q-forcing-lvl-one:pi}
(p, \xi, w, E) \forces& \lnot \exists(x)\psi(f, S, U, x) \iff \forall(x)\forall(p',\xi', w')\Bigl( \\
& (p', \xi', w', E) \nQgeq (p, \xi, w, E) \lor (p', \xi', w', E) \not\models \psi(f, S, U, x)\Bigr) 
\end{align*} 
Part \ref{lem:q-forcing-lvl-one:qgeq-assuming} says, assuming \( (p, \xi, w, E) \in \Qcond \),   \( (p', \xi', w', E) \nQgeq (p, \xi, w, E) \) is \( \sigmazn{1} \).  Hence,  assuming \( q \in \Qcond \), forcing for \( \pizn(f, S, U){1} \) sentences is \( \pizn{2} \). 
% \end{pfequations}    
\end{proof}

% The above lemma captured the idea that if changes to our forcing conditions are somehow screened off from affecting the sentence being forced then they don't make a difference.  Unfortunately, the interaction between extensions of \( \sigma \) and our density bounds complicates proving a similar result about extensions with \( \Icond \) s  

% We'd like to  prove an analogous result with respect to modifications to the conditions for \( S \) and show that \( (\tau, \sigma, k) \forces \psi(f, \sigma' \triangleright S) \) iff  \( (\tau, \sigma' \triangleright \sigma, k) \forces \psi(f,  S) \).  However, the bounds in conditions in \( \Icond \) don't allow for such a simple result.  For instance, maybe \( \psi(f, S) \) asserts \( \lnot\lnot (m \in S) \) and \( p' \) forces this because \( \sigma' \) adds sufficiently many elements to \( \sigma \) that  no extension of \( \sigma' \) can place \( m \in S  \) without violating the bound imposed by \( k \).  Roughly speaking, it turns out that the analogous result requires that our bounds allow a given extension to \( \sigma' \triangleright \sigma \) iff they allow us to extend \( \sigma \) in the same way.  However, formulating this generally is somewhat ugly so we simply avoid the problem by proving the result only for a specific pair of conditions in lemma \ref{lem:conditions-equiv-between}. 

We now show that if we force a set to have density \( 0 \) then we can force the density above some \( l \) to be below any positive bound.

\begin{lemma}\label{lem:density-zero-force-small}
Suppose that \( p \in \Pcond \), \( k \in \omega \),  and \( p \forces \udensity({ \set{x}{\psi(x)}  }) = 0 \) then there is some \( l, p' \Pgeq p \) such that \( p' \forces \udensity[l]({ \set{x}{\psi(x)}  }) \leq 2^{-k}   \).  %Moreover, \( p' \forces \udensity[0]({ \set{x}{\psi(x) \land x \geq l}  }) \leq 2^{-k}  \).    
\end{lemma}
\begin{proof}
Observe that  \[
 p \forces \udensity({ \set{x}{\psi(x)}  }) = 0  \iffdef p \forces \forall(k)\exists(l)\forall(l' \geq l)\left(\udensity[l][l']({ \set{x}{\psi(x)}  }) \leq 2^{-k}  \right)
\] 
The desired result now follows from the definition of forcing and quasi-completeness.  
\end{proof}

Our last forcing lemma is that we can always properly extend conditions in \( \Qcond \). 

\begin{lemma}\label{lem:q-condition-extendable}
Given \( q = (p, (\xi, w, E)) \in \Qcond \) and \( w' \geq w \)  then there are \( p' \Pgeq p \) and \( \xi' \supfunneq \xi \) such that  \[
p' \forces \udensity[\lh{\xi'}]({ \xi' \union \recset(f \Tplus S){E}\restr{[\lh{\xi'}, \infty)}  }) \leq 2^{-w'} 
\]
and  \( q' = (p', \xi', w', E) \Qgeq q \).  Moreover, we can assume \( \tau' \supfunneq \tau\), \( \sigma' \supfunneq \sigma \) and either \( k' > k  \)  or \( k' = k \). 
\end{lemma}
\begin{proof}
Apply lemma \ref{lem:density-zero-force-small} to \( p \)  using \( w' + 2 \) in place of \( k \) and \( \recset(f \Tplus S){E} \) as \( \psi(x) \).  If \( l \) and  \( p_0 \) are given by that lemma we have 
\begin{equation}\label{eq:q-condition-extendable:pzero-bound}
p_0 \forces \udensity[l](\recset(f \Tplus S){E} ) \leq 2^{-w' -2}  
\end{equation}
Choose \( l' > l \) large enough that \( \frac{l}{l'} \leq 2^{-w' -1} \).    Since part \ref{lem:q-condition-equiv:sets-total} of \hyperref[def:q-condition]{the definition of \( \Qcond \)}   guarantees that \( p \) forces \( \recset(f \Tplus S){E} \) is total we can find an extension \( p' \Pgeq p_0 \Pgeq p \) such that \( \lh{\recset(p'){E}} \geq l'  \).  Now let \( \xi' \) be the unique string of length \( l' \) extending \( \xi \) and equal to \( \recset(p'){E} \) on the interval \( [\lh{\xi}, l') \). We now observe   
\begin{align*}
\udensity[\lh{\xi'}]({ \xi' \union \recset(f \Tplus S){E}\restr{[\lh{\xi'}, \infty)}  }) \leq&  \frac{\card{\xi}}{\lh{\xi'}} + \udensity[\lh{\xi'}][l']({\xi'\restr{[\lh{\xi}, l')}  }) + \udensity[\lh{\xi'}]({ \recset(f \Tplus S){E}\restr{[\lh{\xi'}, \infty)}  }) \leq \phantom{X} \\
&\frac{l}{l'} + \udensity[l][l']({ \recset(p'){E} }) + \udensity[l']({ \recset(f \Tplus S){E} }) \leq \phantom{X} \\
& 2^{-w' -1} + 2^{-w' -2}  + \udensity[l]({ \recset(f \Tplus S){E} })
\end{align*}    
With the fact that \(  \udensity[l][l']({ \recset(p'){E} }) \leq 2^{-w' -2} \) following from the fact that \( p' \Pgeq p_0 \) and \eqref{eq:q-condition-extendable:pzero-bound}.  Since the same equation shows us that \( p' \) forces \(  \udensity[l]({ \recset(f \Tplus S){E} }) \leq 2^{-w' -2} \) and \( 2^{-w' -1} + 2^{-w' -2}  + 2^{-w' -2} = 2^{-w'} \) we can derive \[
p' \forces \udensity[\lh{\xi'}]({ \xi' \union \recset(f \Tplus S){E}\restr{[\lh{\xi'}, \infty)}  }) \leq 2^{-w} 
\]  
This leaves us only to prove that \(  (p', \xi', w', E) \Qgeq q \).  By construction \( p' \Pgeq p \) and  it is easy to see that we've either already verified all \( 4 \)  parts of lemma \ref{lem:q-condition-equiv} or they follow immediately from the fact that \( (p, \xi, w, E) \in \Qcond \) and we haven't changed \( E \).  As it is clear that \( \xi' \) is \( 1 \) whenever \(  \recset(p'){E} \) this leaves us only to verify  \( (\xi', w') \Igeq (\xi, w) \)  and  \( (\xi', w') \in \Icond \).  To verify   \( (\xi', w') \in \Icond \) it suffices to note that    \( \udensity[\lh{\xi'}][\lh{\xi'}]( \xi' ) \leq 2^{-w} \)  since otherwise  \( p' \) in the above equation would be forcing a statement guaranteed to be false.     

To verify \( (\xi', w') \Igeq (\xi, w) \) it suffices to check that \( \udensity[\lh{\xi}][\lh{\xi'}]( \xi' ) \leq 2^{-w} \) and that \( (\xi', w') \in \Icond \).   By part \ref{lem:q-condition-equiv:respects-w} of lemma \ref{lem:q-condition-equiv} since \( q \in \Qcond \)  we have that \[
p \forces \udensity[\lh{\xi}]({ \xi \union \recset(f \Tplus S){E}\restr{[\lh{\xi}, \infty)} }) \leq 2^{-w} 
\]
As \( p' \Pgeq p \) it can't witness the failure of the sentence forced by \( p \) and since we simply copied \( \recset(f \Tplus S){E} \) (as interpreted by \( p' \) to build \( \xi' \) we have that \( \udensity[\lh{\xi}][\lh{\xi'}]( \xi' ) \leq 2^{-w} \).   The moreover claim with \( k' > k \)  follows immediately from the fact that we can always further extend \( p' \) to be a strict extension in each component.  To see that we can have \( k' = k \) note that the only place where we extended \( p \) was to a condition \( p' \) which only needed to force a \( \sigmazn(f, S){1} \) fact and therefore doesn't require a change to \( k \).               
\end{proof}

% For the rest of this subsection, fix some indexes \( \e, \i \) and \( \j \) and assume that we have some  \( p_0 = (\tau_0, \sigma_0, k_0) \in \Pcond  \) satisfying \( p_0 \forces \ED  \) and assume that \( \recfnl{\e}{}{} \) and \( \recfnl{\i}{}{} \) are constrained to only yield values \( \leq 1 \) for all oracles and arguments.  Also fix some index \( \j \) with the intention of showing that every condition \( q_1 = (p_1, r_1) \) with \( p_1 = (\tau_1, \sigma_1, k_1) \Pgeq p_0 \) can be extended in a way that either ensures that \( \udensity[\lh{\tau_1}](\set{y}{\recfnl{\j}{X^{\Box U}_{\e}}{y}\conv = \Box}) \geq 2^{-k_0 -2}  \)  

%Moreover, by extending \( \sigma_0 \) with \( 0 \)s if necessary, we may assume that \( \frac{\card{\sigma_0}}{\lh{\sigma_0}} \leq 2^{-k_0 -2} \).   We now work to show that every generic extending \( p_0 \)  
% We won't truly construct the set \( U \) via forcing over  \( \Qcond \) to give an explicit construction but it will still be convenient to use the notation of forcing.  To that end, we assume that \( q \forces \psi \) is defined as usual for the notion of forcing \( \Qcond \) where, as above, we insist that existential witnesses be less than some computable function of \( q \), e.g., \( \lh{\tau} \).  However, this will be a mere notational convenience \mydash unlike our use of forcing over \( \Pcond \) \mydash and we will explicitly demonstrate that the set \( U \) constructed this way has the desired properties.  Finally, before we continue we adopt a few useful notational conventions.

\subsection{A Sufficient Condition}\label{ssec:extending-conditions}

We  identify an arithmetic sentence asserting that we have witnesses for the fact that  \( f \) and \( f^{\Box S} \) both compute effective dense descriptions of some set \( X \).  
% First, however, we introduce the following notation.

% \begin{definition}\label{def:merge-functions}
% Given \( f, g \in \left(\omega \union \set{\Box} \right)^{\omega} \) define 
% \end{definition}

\begin{definition}\label{def:iiprime-good}
\begin{subequations}\label{eq:def-ED}
\begin{align}
\ED(e, i) \eqdef & \forall(y)\exists(s)\left( \recfnl{e}{f}{y}\conv[s] \boxeq \recfnl{i}{f^{\Box S}}{y}\conv[s] \right) \land \phantom{X} \label{eq:def-ED:boxeq} \\
& \udensity({ \set{y}{\recfnl{e}{f}{y} = \Box   }   }  ) = 0 \label{eq:def-ED:Xe-density-zero} \land \phantom{X} \\
& \udensity({ \set{y}{\recfnl{i}{f^{\Box S}}{y}  = \Box   }   }  ) = 0  \label{eq:def-ED:i-box-S-density-zero} \\
& \notag \\
X_e \eqdef& \recfnl{e}{f}{} \notag
\end{align}
\end{subequations}
% \begin{flalign*}
% \qquad \qquad X_e \eqdef& \recfnl{e}{f}{} &&
% \end{flalign*}
% And we define \( X_e \) to abbreviate \( \recfnl{e}{f}{} \).
\end{definition}

  To avoid the complication of adding sentences to \( \ED(e,i) \) asserting that \( \recfnl{e}{f}{} \) and \( \recfnl{i}{f^{\Box S}}{} \) are sets we will tacitly assume that the indexes we use for \( \ED(e, i) \) are  set-indexes, i.e., indexes from a computable set guaranteed to only return \( 0, 1 \) or \( \Box \).  Using \eqref{eq:density-for-re} to expand the assertions of \( 0 \) density in equations \eqref{eq:def-ED:Xe-density-zero} and \eqref{eq:def-ED:i-box-S-density-zero} we can take   \( \ED(e,i) \) to be a  \( \pizn(f \Tplus S){3} \) sentence.  We now identify a condition that will suffice to prove theorem \ref{thm:effective-dense}.

\begin{proposition}\label{prop:extendable-conditions}
If \( \e, \i \) are set-indexes \( \j \in \omega \),  \( q_0 = (p_0, r_0) \in \Qcond \),  \( p_0  \forces \ED  \) and \( q_1 \Qgeq q_0 \)  then there is a condition \( q^{\ddagger}  \Qgeq q_1 \) such that one of the following obtains.
\begin{enumerate}
    \item\label{prop:extendable-conditions:partial}    \( \displaystyle \exists(x)\left( q^{\ddagger} \forces \lnot \exists(s)\recfnl{\j}{X_{\e}^{\Box U}}{x}\conv[s] \right) \)
    \item\label{prop:extendable-conditions:not-small}   \( \displaystyle q^{\ddagger} \forces \exists(l > \lh{\tau_1})\exists(s)\left(\udensity[l][l](\set{y}{\recfnl{\j}{X_{\e}^{\Box U}}{y}\conv[s] = \Box}) \geq 2^{-k_0 -2}\right)  \) 
    \item\label{prop:extendable-conditions:disagree}  \(\displaystyle q^{\ddagger} \forces \exists(x)\exists(s)\left(\recfnl{\j}{X_{\e}^{\Box U}}{x}\conv[s] \nboxeq f(x) \right)  \) 
\end{enumerate} 
Moreover, given \( q_0, q_1 \) and \( p_2 \Pgeq p_1 \), \( \zerojj \) can compute a condition \( q^{\ddagger} \) satisfying the above with \( p^{\ddagger} \Pgeq p_2 \).  
\end{proposition}

We now prove theorem \ref{thm:effective-dense} on the assumption the above proposition holds.

\begin{proof}
To that end, assume we have a function \( f \in \baire \) that is \( 3 \)-generic.  By our discussion above, we know that there is some \( S \in \mathbb{I} \) such that the pair \( (f, S) \) is \( 3 \)-generic with respect to \( \Pcond \) \mydash where this means both forcing every \( \sigmazn(f,S){3} \) sentence or its negation and meeting or strongly avoiding every \( \sigmazn{3} \) subset of \( \Pcond \).  % \mydash that is for every \( \sigmazn(f \Tplus S){3} \) sentence \( \psi(f, S) \),   \( f, S \) extend a condition in \( \Pcond \) forcing \( \psi(f, S) \)  or \( \lnot \psi(f, S) \) and for every \( \sigmazn{3} \) subset of \( \Pcond \) either \( (f, S) \) extend a condition in that set or extend a condition with no extensions in that set. 
We first argue that if  the assumptions of proposition \ref{prop:extendable-conditions} are satisfied with \( p_0 \Pleq p_1 \Pleq (f, S) \) then we can find \( q^{\ddagger}  \) satisfying the conclusion of the proposition with \( p^{\ddagger} \Pleq (f, S) \).  This follows since the set of \( p^{\ddagger} \) such that there is some \( p_2 \Pgeq p_1 \) such that \( q^{\ddagger} \) is equal to our \( \zerojj \) computable function of \( q_0, q_1 \) and \( p_2  \) is a \( \sigmazn{3} \) set of conditions and therefore must be met  or strongly avoided by some condition \( p \Pleq (f, S) \) and it can't be strongly avoided.

We will build \( U \) via a sequence of conditions \( \overline{q}_n = (\overline{p}_n, \overline{r}_n) = (\overline{\tau}_n, \overline{\sigma}_n, \overline{k}_n, \overline{\xi}_n, \overline{w}_n, \overline{E}_n) \) with each \( \overline{p}_n \Pleq (f, S) \).  We start with \( \overline{q_0} \) as the trivial condition (extended by all other conditions).  We note that we can tacitly assume that every time we define a new condition \( \overline{q}_n  \) every component, save perhaps \( \overline{E}_n \), differs from the corresponding component in \( \overline{q}_{n-1} \).  To see this, we observe that by lemma \ref{lem:q-condition-extendable} the following set is dense above \( \overline{q}_n  \)   \[
\set{\ddot{q}}{\ddot{q} \Qgeq \overline{q}_n \land \ddot{\tau} \supfunneq \overline{\tau}_n \land \ddot{\sigma} \supfunneq \overline{\sigma}_n \land \ddot{k} > \overline{k}_n \land  \ddot{\xi} \supfunneq \overline{\xi}_n \land \ddot{w} > \overline{w}_n \land \ddot{E} = \overline{E}_n  }
\] 
By part \ref{lem:q-forcing-lvl-one:qgeq-assuming} of lemma  \ref{lem:q-forcing-lvl-one}, since \( \overline{q}_n \in \Qcond \), the set of \( \ddot{p} \) such that \( \ddot{q} \) is in the above set is a \( \pizn{1} \) dense set above \( \overline{p}_n \) and therefore, by the \( 3 \) genericity of \( f, S \),  must be met by some extension \( p' \)  of \( \overline{p}_n \) with \( p' \Pleq (f, S) \).  We will silently assume that the definition we give for \( \overline{q}_n \) below is replaced with such an extension.

 % \( \overline{q}_n  \) we might have otherwise produced to be a strict extension (excepting \( E \)) of \( \overline{q}_{n-1}  \) while retaining the property that \( \overline{p}_n \Pleq (f, S) \).   

Now suppose that we've defined \( \overline{q}_{n-1} \) and that \( n \) codes for the quadruple \( (\e, \i, \j, k) \) and execute the following cases.

\begin{pfcases*}

\case[\( k = 0 \)]  As \( (f, S) \) is \( 3 \)-generic and \( \overline{p}_{n-1} \Pleq (f, S) \) there must be some \( p \Pgeq \overline{p}_{n-1} \) with \( p \Pleq (f, S)  \)  that either forces \( \ED \) or forces \( \lnot \ED \).  In either case choose \( \overline{q}_n \) to be  \( (p, \overline{r}_{n-1}) \) but if \( p \forces \lnot \ED \) mark \( (\e, \i, \j) \) finitely satisfied.   

% On the other hand, suppose that \( p \forces \ED \).  In this case, we apply proposition \ref{prop:extendable-conditions} using \( (p, \overline{r}_{n-1}) \) as both \( q_0 \) and \( q_1 \)  to get \( q^{\ddagger}_n \Qgeq (p, \overline{r}_{n-1}) \) and set  \( \overline{q}_n \) to extend   \( q^{\ddagger}_n \) with \( \overline{w}_n > \overline{w}_{n-1} \).   If proposition \ref{prop:extendable-conditions} was satisfied in this case via either part \ref{prop:extendable-conditions:partial} or \ref{prop:extendable-conditions:disagree} we also mark \( (\e, \i, \j) \) finitely satisfied. 

\case[\( k \neq 0 \)]  If we've marked  \( (\e, \i, \j) \) finitely satisfied we let \( \overline{q}_n = \overline{q}_{n-1} \).  Otherwise, let \( m \) code the quadruple \( (\e, \i, \j, 0) \) and apply proposition \ref{prop:extendable-conditions} using \( \overline{q}_m \) as \( q_0 \) and \( \overline{q}_{n-1} \) as \( q_1 \) and let \( \overline{q}_n \) be the result of that lemma.  If the proposition is satisfied  via either disjunct \ref{prop:extendable-conditions:partial} or \ref{prop:extendable-conditions:disagree} mark \( (\e, \i, \j) \) finitely satisfied.   

\end{pfcases*}

This completes our construction and by our discussion above we are guaranteed that  \( U =  \lim_{n \to \infty} \overline{\xi}_n  \in \mathbb{I} \).  Now assume, for a contradiction, that \( f \NEDequiv X \) for some set \( X \).  As \( f, f^{\Box S} \) are effective dense descriptions of \( f \) they must both compute effective dense descriptions of \( X \).  Thus, there must be some \( \e, \i \) such that  both \( X_{\e} \) and \( \recfnl{\i}{f^{\Box S}}{} \) are effective dense descriptions of \( X \) and \( \ED \) holds.  Thus, \( X^{\Box U}_{\e} \) is an effective dense description of \( X \) and, by assumption, computes some effective dense description of \( f \).  Let \( \recfnl{\j}{X^{\Box U}_{\e}}{} \) witnesses this.  

When \( n  \) codes \(  (\e, \i, \j, 0) \) we must have that \( \overline{p}_n \forces \ED \) since forcing equals truth for generics and \( (f, S) \) is \( 3 \)-generic.   Suppose for some \( k \) and  \( n  \) coding \(  (\e, \i, \j, k) \) we marked \( (\e, \i, \j) \) finitely satisfied then we either saw  \( \overline{q}_n \) witness a finite computation showing that \( \recfnl{\j}{X^{\Box U}_{\e}}{}\conv \in \omega \) which disagrees with \( f \) or for some \( x \),   \( \overline{q}_n \forces \lnot \exists(s)\recfnl{\j}{X_{\e}^{\Box U}}{x}\conv[s]  \).  Applying the  definition of forcing, we see that \( \recfnl{\j}{X_{\e}^{\Box U}}{} \) is partial. In either case, we've contradicted or assumption that  \( \recfnl{\j}{X^{\Box U}_{\e}}{} \) is an effective dense description of \( f \).

Now suppose we never mark \( \e, \i, \j \) finitely satisfied.  In this case we claim that for all \( l \), \( \udensity[l]( { \set{x}{\recfnl{\j}{X^{\Box U}_{\e}}{x} = \Box } }) \geq 2^{-k_m -2} \) where \( m  \)  codes \( (\e, \i, \j, 0) \).  To see this, note that for sufficiently large \( n \) coding   \( (\e, \i, \j, k) \) (for some \( k \))  we  have, via our discussion above about proper extension,  \( \lh{\overline{\tau}_{n-1}} > l_0 \) so our application of proposition \ref{prop:extendable-conditions} ensures that for some \( l > l_0 \) we have \( \udensity[l][l](\set{y}{\recfnl{\j}{X_{\e}^{\Box U}}{y} = \Box}) \geq 2^{-k_0 -2} \).  Thus \( \udensity(\set{y}{\recfnl{\j}{X_{\e}^{\Box U}}{y} = \Box}) \geq 2^{-k_0 -2}  \)   This again contradicts our assumption that \( \recfnl{\j}{X_{\e}^{\Box U}}{} \) was an effective dense description of \( f \) and completes the proof. 

\end{proof}

We devote the rest of this section to proving proposition \ref{prop:extendable-conditions}.  We therefore will assume that the assumptions in that proposition are satisfied.  That is, for the rest of this section, we fix  set-indexes \( \e, \i \), \( \j \in \omega \),  \( q_0  \in \Qcond \) and \( q_1  \Qgeq q_0 \) and assume  \( p_0 \forces \ED  \) where \( q_i = (p_i, r_i) \), \( p_i = (\tau_i, \sigma_i, k_i) \) and \( r_i = (\xi_i, w_i, E_i) \) as per our convention.

\subsection{Building \texorpdfstring{\( F \)}{F} and Forcing Translation}

In this subsection, we define our interval \( F \) above the condition \( q_1 \in \Qcond \) and prove lemmas letting us translate between modifications made to constant symbols in a sentence and modifications made to the conditions forcing them.  First, we  extend \( q_1 \) in multiple steps to build conditions with certain technical features we will need later.

% We first show that every \( q_1 \in \Qcond \) can be extended in a way that leaves enough room between our density bounds and the density of the elements/sets we have already added to \( S, U \) to prepare us to later be able to add our desired index to \( E \) and to arrange some details to help us define \( F \). 

\begin{lemma}\label{lem:extend-q-small}
There is a condition \( q_2 \Qgeq q_1 \) with \( E_2 = E_1, w_2 = w_1 \) such that 
\begin{subequations}
\begin{align}
&\lh{\sigma_2} \geq \lh{\tau_2} \label{eq:extend-q-small:sigma-long} \\
&\!\frac{\card{\sigma_2}}{\lh{\sigma_2}}, \frac{1}{\lh{\sigma_2}} \leq 2^{-k_0 -2}\label{eq:extend-q-small:sigma-small} \\
&\, p_2 \forces \udensity[\lh{\xi_2}]({ \xi_2 \union \recset(f \Tplus S){E_1}\restr{[\lh{\xi_2}, \infty)}  }) \leq 2^{-w_1-4} \label{eq:extend-q-small:q-condition-small} \\ 
&\, p_2 \forces  \udensity[\lh{\tau_2}](\set{y}{X_{\e}(y) = \Box}) < 2^{-w_1-4}  \label{eq:extend-q-small:i-box-small} 
\end{align} 
\end{subequations} 
\end{lemma}
\begin{proof}
To satisfy \eqref{eq:extend-q-small:sigma-small} it is enough to just extend \( \sigma_1 \) with sufficiently many \( 0 \)s.  Acting on top of this, we apply lemma \ref{lem:q-condition-extendable} to guarantee \eqref{eq:extend-q-small:q-condition-small} holds and then apply lemma \ref{lem:density-zero-force-small} (since \( p_0 \) forces \( \set{y}{X_{\e}(y) = \Box} \) to have density \( 0 \)) to extend the condition generated by the previous step so that for some \( l \) we have 
\[
p_2 \forces  \udensity[l](\set{y}{X_{\e}(y) = \Box}) < 2^{-w_1-4} 
\] 
Finally, we simply extend \( \tau_2 \) to have length \( l \) and then we can extend \( \sigma_2 \) with more \( 0 \)s to satisfy \eqref{eq:extend-q-small:sigma-long} without upsetting any of the other equations.  
\end{proof}   

We now define our interval \( F \).

\begin{definition}\label{def:F-interval}
\begin{align*}
 l_F &= \murec{l}{\frac{l}{\lh{\sigma_2}+l} \geq 2^{-k_0 - 2}} \\
 F &= [\lh{\sigma_2}, \lh{\sigma_2}+l_F) \\
 \delta_0(y) &= \begin{cases}
                    \sigma_0(y) & \text{if } y < \lh{\sigma_0}\\
                    0 & \text{if }  \lh{\sigma_0} \leq y < \lh{\sigma_2} \\
                    1 & \text{if } \lh{\sigma_2} \leq y < \lh{\sigma_2}+l_F \\
                    \diverge & \text{if } y \geq \lh{\sigma_2}+l_F
               \end{cases}
\end{align*}
\end{definition}

As \( 2^{-k_0 -2} < 1 \) and \( \lim_{l \to \infty} \frac{l}{\lh{\sigma}+l} = \infty \) we have that \( l_F \) is well-defined.  We now observe some properties of this definition.   

\begin{lemma}\label{lem:build-F-good}
\( (\tau_2, \delta_0, k_0)   \Pgeq p_0  \) and \( 2^{-k_0 - 2}  \leq  \frac{l_F}{\lh{\sigma_2}+l_F} \leq 2^{-k_0 - 1}  \).
\end{lemma}
\begin{proof}
We prove the inequality first. The top line below holds by definition and the rest by algebraic manipulation and \eqref{eq:extend-q-small:sigma-small}.
\begin{align*}
 &\frac{l_F}{\lh{\sigma_2}+l_F}   \geq 2^{-k_0 - 2} > \frac{l_F-1}{\lh{\sigma_2}+l_F -1} > \phantom{X} \\
& \frac{l_F}{\lh{\sigma_2}+l_F} - \frac{1}{\lh{\sigma_2}} \geq \frac{l_F}{\lh{\sigma_2}+l_F} -  2^{-k_0 -2} \\
\end{align*}
Adding \(  2^{-k_0 -2} \) to the second and final terms yields the desired inequality. 
\begin{equation*}
2^{-k_0 - 2} \leq \frac{l_F}{\lh{\sigma_2}+l_F} < 2^{-k_0 - 2} + 2^{-k_0 - 2} = 2^{-k_0 - 1}
\end{equation*}
As \( \tau_2 \supfun \tau_0 \) to verify \( (\tau_2, \delta_0, k_0)   \Pgeq p_0  \) requires showing only that 
\( \udensity[\lh{\sigma_0}][\lh{\delta_0}](\delta_0) \leq 2^{-k_0} \).  The definition of \( \delta_0 \) makes it clear that 
\[
\udensity[\lh{\sigma_0}][\lh{\delta_0}](\delta_0) \leq \udensity[\lh{\delta_0}][\lh{\delta_0}](\delta_0) =  \frac{\card{\delta_0}}{\lh{\delta_0}} = \frac{\card{\sigma_0}}{\lh{\sigma_2}+l_F} + \frac{l_F}{\lh{\sigma_2}+l_F} \leq 2^{-k_0 -2} + 2^{-k_0 - 1} < 2^{-k_0} 
\] 
With the second to last inequality justified by the inequality we proved immediately above and \eqref{eq:extend-q-small:sigma-small}.
\end{proof}

Before we can define our set \( C \) we need to extend \( q_2 \) in a way that will make it convenient to translate between conditions extending \( q_3 \) and conditions extending \( q_0 \) that place \( F \) into \( S \).     

\begin{lemma}\label{lem:build-delta}
There is some \( p_3 \Pgeq p_2 \) and \( \delta \supfun \delta_0 \) such that \( \lh{\delta} = \lh{\sigma_3} \), \( \card{\delta} = \card{\sigma_3} \), \( \lh{\tau_3} \leq \lh{\sigma_2} \) and \( (\tau_3, \delta, k_3) \Pgeq p_0 \).  Moreover, if \( \hat{\tau}_2 \supfun \hat{\tau}_1 \supfun \tau_3 \), \( \hat{\sigma}_2 \supfun \hat{\sigma}_1 \supfun \sigma_3 \) and \( \hat{k}_2 \geq \hat{k}_1 \)       
\end{lemma}
\begin{proof}
Clearly we can extend \( \delta_0 \) to \( \delta' \)  and \( \sigma_2 \) to \( \sigma' \) by only adding \( 0 \)s so that \( (\sigma_2, k_2) \Ileq (\sigma', k_2) \), \( \lh{\sigma'} = \lh{\delta'} \) and  \( (\sigma_0, k_0) \Ileq (\delta', k_2) \).   Now extend \( \sigma' \) with \( \card{\sigma'} - \card{\delta'} \) many \( 1 \)s and \( \delta' \) with the same number of \( 0 \)s to produce \( \sigma_3 \) and \( \delta \)   \mydash it does follow from the above inequalities that \( \card{\sigma'} - \card{\delta'} \geq 0 \) but since the same argument would work in the other direction it's not necessary to verify.  It is easy to check that since \( (\delta', k_2) \in \Icond \) that we have \( (\sigma_3, k_2) \in \Icond \) and that \( \udensity[\lh{\sigma'}][\lh{\sigma_3}](\sigma_3) \leq 2^{-k_2} \).  Let \( p_3 = (\tau_2, \sigma_3, k_2) \) and observe that, by \eqref{eq:extend-q-small:sigma-long},  \( \lh{\sigma_2} \leq \lh{\tau_2} = \lh{\tau_3} \)  to complete the proof.        
\end{proof}

Fix \( p_3 \) and \( \delta \) as in the above lemma and let \( q_3 = (p_3, r_2) \).  We now observe that we can swap out \( \sigma_2 \) and \( \delta \) in any conditions.  

\begin{lemma}\label{lem:swap-delta-sigma}
\[
\hat{p}_1 \Pgeq p_3 \implies \hat{p}_2 \Pgeq \hat{p}_1 \iff (\hat{\tau}_2, \delta \triangleright \hat{\sigma}_2, \hat{k}_2) \Pgeq  (\hat{\tau}_1, \delta \triangleright \hat{\sigma}_1, \hat{k}_1) 
\] 
\end{lemma}
\begin{proof}
This is immediate from the fact that if \( \hat{\sigma}_2 \supfun \sigma_3 \) then for all \( l \geq \lh{\sigma_3}  \) both \( \hat{\sigma}_2 \) and   \( \delta \triangleright \hat{\sigma}_2  \) have the same number of elements below \( l \). 
\end{proof}

% In light of the above lemma we adopt the following notation for \( \zeta \in \wstrs \) and \( \hat{p} \in \Pcond \)  
% \begin{equation}

% \end{equation}
We now  prove out main technical lemma letting us translate between modifications made to our conditions and modifications made to the constant symbols in the sentences being forced.    

\begin{lemma}\label{lem:forcing-translation}
If \( \zeta \in \wstrs \), \( p, p', \hat{p} \in \Pcond \),  \( f, S \) don't appear in \( \psi \) and we define
\begin{alignat*}{2}
\hat{p}^{*} &\eqdef (\hat{\tau}, \delta \triangleright \hat{\sigma}, \hat{k}) & \qquad \qquad {\hat{p}}^{\zeta \triangleright f} &\eqdef (\zeta \triangleright \hat{\tau},  \hat{\sigma}, \hat{k})  
\end{alignat*}
then
\begin{subequations}
\begin{align}
&\hat{p} \Pgeq p_3 \implies \hat{p} \forces \psi(f, \delta \triangleright S)  \iff \hat{p}^{*}  \forces \psi(f, S)  \label{eq:forcing-translation:S} \\
&\lh{\zeta} \leq \lh{\hat{\tau}} \implies  \hat{p} \forces \psi(\zeta \triangleright f, S) \iff \hat{p}^{\zeta \triangleright f} \forces \psi(f, S)\label{eq:forcing-translation:f} \\
& p \agree[{[0, \lh{\zeta})}] p' \land \lh{\tau} \geq \lh{\zeta} \implies p \forces \psi(\zeta \triangleright f, S) \iff p' \forces \psi(\zeta \triangleright f, S) \label{eq:forcing-translation:shield}
\end{align}
\end{subequations}
Moreover, \( p_3 \forces \ED[f, \delta \triangleright S] \).
\end{lemma}
\begin{proof}
We prove the main claim by induction on the complexity of \( \psi \).  We first prove \eqref{eq:forcing-translation:S}.  This claim is evident for \( \psi  \in \sigmazn{1} \) and  it is easy to see that the class of \( \psi \) the claim holds for is closed under \( \land, \lor \) and existential quantification. We now prove the claim holds for  \( \lnot \psi \) assuming that it holds for \( \psi \).  

For the \( \implies \) direction assume \( \hat{p} \Pgeq p_3 \) and  \( \hat{p} \forces \lnot \psi(f, \delta \triangleright S) \). Suppose, by way of contradiction, that \( \hat{p}^{*}  \nforces \lnot \psi(f, S) \).  Thus, for some \( p  \Pgeq \hat{p}^{*}  \) we have \( p \forces \psi(f, S) \). By lemma \ref{lem:swap-delta-sigma} we can assume that \( p = \hat{p}_1^{\delta \triangleright S}  \) with \( \hat{p}_1 \Pgeq \hat{p} \) and by the inductive hypothesis we have that \( \hat{p}_1 \forces  \psi(f, \delta \triangleright S) \) contradicting the assumption that \( \hat{p} \forces \lnot \psi(f, \delta \triangleright S) \).

The argument in the \( \impliedby \) direction is the same as the argument for \( \implies \) just with the sides flipped. Thus, this suffices to prove  \eqref{eq:forcing-translation:S} for all arithmetic sentences \( \psi \).  The argument for \eqref{eq:forcing-translation:f} is identical but works for arbitrary modifications \( \zeta \) not just a specially constructed \( \delta \). We observe that \eqref{eq:forcing-translation:shield} follows from \eqref{eq:forcing-translation:f} and the fact that \( p^{\zeta \triangleright f} = {p'}^{\zeta \triangleright f} \).  The moreover follows by noting that \( p_3^{*} \Pgeq p_0 \) by lemma \ref{lem:build-delta} \mydash so \( p_3^{*} \forces \ED \) and applying \eqref{eq:forcing-translation:S}.
\end{proof}

We now show  that if a formula ignores \( f \) on some set \( F \) then the forcing relation also ignores the values of conditions on that set.

\begin{lemma}\label{lem:hidden-no-effect}
Suppose that \( \psi \) doesn't contain \( f, S \), \( p \in \Pcond \), \( p \agree[F] p' \) then 
\begin{align*}
F &\subset \sigma \implies p \forces \psi(f^{\Box  S}) \iff p' \forces \psi(f^{\Box  S}) \\
F &\subset \delta \implies p \forces \psi(f^{\Box  \delta \triangleright S}) \iff p' \forces \psi(f^{\Box \delta \triangleright S}) 
\end{align*}
\end{lemma}
Recall \( F \subset \sigma \) requires that \( \sigma(x)\conv = 1 \) for all \( x \in F \).   
\begin{proof} 
All of these are obvious if \( \psi \in \sigmazn{1} \) as no part of \( f\restr{F} \) is consulted by the sentence and the general result follows by straightforward induction on formula complexity using the definition of forcing just like the last lemma.
\end{proof}

\subsection{Building Alternatives}

Our goal in this subsection is to build two extensions \( p^y_5, p^z_5 \)  of \( p_4 \) which disagree on \( F \)  and then show that we can add the set of locations  the resulting versions of \( X_{\e} \) might differ on to \( U \).   First, however, we must define the set \( C \) which we claim will contain any strong ( non-\( \Box \)) disagreement.  This might seem backwards, but we must proceed in this order so we know the point at which the density of \( C \) gets small enough to add to \( U \) and therefore the initial segment on which our two versions of \( X_{\e} \) must agree.

\begin{definition}\label{def:C-set}
Define the \( f \Tplus S \) computable set \( C \) as follows
\[
C(x) = \begin{cases}
        0 & \text{if } \recfnl{\i}{f^{\Box \delta \triangleright S}}{x}\conv \neq \Box \\
        1 & \text{if } \recfnl{\i}{f^{\Box \delta \triangleright S}}{x}\conv = \Box
        \end{cases}
\]
\end{definition}  

% When necessary, we will explicitly denote the reliance on \( f, S \) by \( C^{f, S} \) and, following our usual convention\footnote{Divergence is required when \( x \geq \min(\lh{\hat{\tau}}, \lh{\hat{\sigma}})  \) as a computation from a string is only defined up to the length of that string and  \( \lh{\tau^{\Box \sigma}} = \min(\lh{\hat{\tau}}, \lh{\hat{\sigma}}) \). However, this detail is not important to the proofs and is included only to remind the reader how our conventions fit together.}, use \( C^{\hat{p}} \) for \( \hat{p} \in \Pcond \) to abbreviate \[
% C^{\hat{p}}(x) = \begin{cases}
%                     \diverge & \text{if } x \geq \min(\lh{\hat{\tau}}, \lh{\hat{\sigma}}) \\
%                     \Box & \text{o.w. if } \delta(x)\conv = 1 \lor \hat{\sigma}(x)\conv = 0 \\
%                     \hat{\tau}(x) & \text{o.w. if } \hat{\tau}(x)\conv \land \delta(x) = 0 \lor x \geq \lh{\delta} \land \hat{\sigma}(x)\conv = 0
%                  \end{cases}
% \] 

We now note that that we can force \( \udensity[l](C) \) to be small.

\begin{lemma}\label{lem:C-density-above-l}
There is some condition \( p_4 \Pgeq p_3 \) and some \( l_0 \) such that if we define \( q_4  \) to be \(  (p_4, r_2) = (p_4, r_3) \)  \[
 p_4 \forces \udensity[l_0](C) \leq 2^{-w_1-4} \land \udensity(C) = 0 \land \forall(x)\exists(z)C(x)\conv[s] \land \lh{C} \geq l_0
\] 
Moreover, if we define \( q_4  \) to be \(  (p_4, r_2) = (p_4, r_3) \) then we can also assume that \( l_0 > \lh{\tau_4} > \lh{\sigma_3} \geq \lh{\delta_0} \) and  \(  \frac{\lh{\xi_4}}{l_0} < 2^{-w_1 -4}  \) and \( w_4 = w_1 \).
\end{lemma}
\begin{proof}
We observe that \( p_3 \) forces the totality and \( 0 \) density of \( C \) follows by unpacking the \hyperref[eq:def-ED]{definition of} \( \ED[f, \delta \triangleright S]  \) which, by lemma \ref{lem:forcing-translation}, is forced by \( p_3 \).  Using this fact, apply lemma \ref{lem:density-zero-force-small} to find \( p_4 \Pgeq p_3  \) and \(  l_0 \) with \(  p_4 \forces \udensity[l_0](C) \leq 2^{-w_1-4} \) and note that we define \( w_4 = w_2 = w_1 \).  We can further extend \( p_4 \)  as needed to ensure \( C \) is forced to converge on \( [0, l_0) \)   and the other inequalities can be satisfied simply by extending \( \tau_4 \) and increasing \( l_0 \) until the inequalities are satisfied as \( \lh{\sigma_3} \geq \lh{\delta_0} \) was already guaranteed in lemma \ref{lem:build-delta}.  
\end{proof}

Fix \( q_4 \) as defined in the above lemma.  For the first time,  \( \tau_4 \) is defined on \( F \) so we define modified versions of \( \tau_4 \) we can use which take on different values on \( F \).

\begin{definition}\label{def:tau-prime}
\begin{equation*}%\label{eq:ramsey:hat-tau}
\tau^v_{i}(x) = \begin{cases}
                    \tau_i(x) & \text{if } x < \lh{\tau_i} \land x \nin F \\
                    v & \text{if } x \in F \land x \leq \lh{\tau_i} \\
                    \diverge & \text{otherwise}
                \end{cases}
\end{equation*}
Further define \( f^v = \tau^v_4 \triangleright f \), \( X^v_{\e} = \recfnl{\e}{f^v}{} \),   \( p^v_i = (\tau^v_i, \sigma_i, k_i) \) and \( q^v_i = (p^v_i, r_i) \). 
\end{definition} 

This definition will apply to conditions we have yet to define as well as \( p_4 \).  The conditions \( q_1, q_2, q_3 \) as well as their components are unaffected by this definition as \( \lh{\tau_3} = \lh{\tau_2} \leq \lh{\sigma_2} \) and \( F = [\lh{\sigma_2}, \lh{\sigma_2} + l_F) \) while \( \lh{\tau_4} > \lh{\delta_0} = \lh{\sigma_2} + l_F \).  Also, recall that in terms of the notation from definition \ref{def:condition-agree} we have \( p^r_i \agree[F] p^{r'}_i \).   We now show that we can find \( y, z \) and condition \( p_5 \) so that \( p^y_5 \) and \( p^z_5 \) agree on \( X_{\e}\restr{l_0} \).    

\begin{lemma}\label{lem:ramsey}
There is a condition \( p_5 \Pgeq p_4 \) and \( y, z \in \omega \)  such that  if \( q_5 = (p_5, r_4) \) then  for any \( v \in \set{y,z} \)
\begin{subequations}\label{eq:ramsey}
\begin{align}
& \lh{\recfnl{\e}{p^v_5}{}} \geq l_0 \label{eq:ramsey:Xe-long} \\
& \lh{\recset(p^v_5){E_5}} \geq l_0 \label{eq:ramsey:recset-long} \\
& \recfnl{\e}{p^y_5}{}\restr{l_0} = \recfnl{\e}{p^z_5}{}\restr{l_0} \label{eq:ramsey:agree}
\end{align}
\end{subequations}
\end{lemma}
This proof requires a fair bit of bookkeeping but the basic intuition is very simply \mydash we build an increasing sequence of compatible strings \( \nu^v \) which we concatenate with the strings \( \tau^v_4 \) to make \( \recfnl{\e}{\tau^v_4 \concat \nu^v}{} \) converge on a long enough initial segment.  We then use the pigeonhole principle to find \( y, z \) giving compatible computations.  By a similar extend and copy process we can produce the conditions \( p^y_5, p^z_5 \).
\begin{proof}
% First we appeal to quasi-completeness, and the fact that \( p_4 \) forces the totality of the computations in question,  to find \( p \Pgeq p_4 \) such that \( \lh{\recfnl{\e}{p}{}} \geq l_0 \)  and  \( \lh{\recset(p){E_4}} \geq l_0 \).  

Let \( \hat{\tau}^{0} \) be the shortest extension of \(  \tau^{0}_4  \)  satisfying \( \lh{\recfnl{\e}{\hat{\tau}^{0}}{}} \geq l_0 \)   \mydash note that only the unhatted versions are within the scope of definition \ref{def:tau-prime}.  Now assume that \( \hat{\tau}^{n} \supfun \tau^n_4 \) is defined and \( \nu^{n} \) is the unique string satisfying \( \hat{\tau}^{n}  = \tau^n_4\concat \nu^{n}  \) and let \( \hat{\tau}^{n+1} \) be the least extension of \( \tau^{n+1}_4\concat \nu^{n} \) such that \( \lh{\recfnl{\e}{\hat{\tau}^{0}}{}} \geq l_0 \).  The string \( \hat{\tau}^n \) must exist since \( \tau^{n}_4 \supfun \tau_3  \) and \( p_3  \) \mydash and hence \( \tau_3 \) since the sentence doesn't mention \( S \) \mydash forces the totality of \( \recfnl{\e}{f}{} \).       

As there are only \( 3^{l_0} \) possible elements of \( \set{0, 1, \Box}^{l_0} \) but there are infinitely many strings \( \hat{\tau}^n \) by the pigeonhole principle there must be \( y < z \) such that \( \recfnl{\e}{\hat{\tau}^{y}}{}\restr{l_0} = \recfnl{\e}{\hat{\tau}^{z}}{}\restr{l_0} \).  Now pick some \( p = (\hat{\tau}^z\concat \theta, \sigma, k) \) such that  \( p \Pgeq (\hat{\tau}^z, \sigma_4, k_4)  \) and  \( \lh{\recset(p){E_5}} \geq l_0 \).  As \( E_5 = E_1 \) and \( (\hat{\tau}^z, \sigma_4, k_4) \Pgeq p_1 \) such an extension must exist as \( p_1 \) forces the totality of \( \recset(f \Tplus S){E_1} \).  Define \( \dot{p} \) to be a condition satisfying \( \dot{p} \Pgeq (\hat{\tau}^{y}\concat \nu^z \concat \theta, \sigma, k)  \) such that \( \lh{\recset(\dot{p}){E_5}} \geq l_0 \) \mydash this is guaranteed to exist by the same argument above. 

Finally, define \( \tau_5 \) to be the unique extension of \( \tau_4 \) satisfying \( \tau_5 \agree[F] \dot{\tau} \) and \( p_5 = (\tau_5, \dot{\sigma}, \dot{k}) \).  A careful, but uninteresting, analysis shows the resulting conditions \( p^y_5, p^z_5 \) have the desired features.        
\end{proof}

Fix \( p_5, y, z \) and \( q_5 = (p_5, r_4) \) and continue in the next subsection.

\subsection{Hiding Disagreement}   

While, we will shortly prove, that \( C \) covers any differences between \( X^y_{\e} \) and \( X^z_{\e} \) on their strong domain we also need to ensure that  \( X_{\e}^{y, \Box U}  \) and \(  X_{\e}^{z, \Box U} \) have the same strong domain.  We deal with this by defining the set \( O \) to contain \( C \) and all \( x \) such that either \( X_{\e}^{z, \Box U}(x) = \Box \) or  \( X_{\e}^{y, \Box U}(x) = \Box \). 

\begin{definition}\label{def:obscuring-set}
\[
O \eqdef \recset(f \Tplus S){o} \eqdef \set{x}{\recfnl{\e}{f^y}{x} = \Box  } \union \set{x}{\recfnl{\e}{f^z}{x} = \Box  } \union C \union \recset(f^y \Tplus S){E_1} \union \recset(f^z \Tplus S){E_1}
\]
% We write \( O^{f, S} \) to make the dependence on \( f, S \) explicit and \( O^{p} \) to denote the computation from the parts of \( f, S \) specified by \( p \in \Pcond \).    
\end{definition}

We now verify that the set \( O \) define above has sufficiently small density above \( l_0 \) for us to add it to \( U \).   

\begin{lemma}\label{lem:o-small}
The conditions \( p_5, p^y_5 \) and \( p^z_5 \) all force the following sentence.  
\[
\udensity(O) = 0 \land \forall(x)\exists(s)O(x)\conv \land \udensity[l_0]({O}) \leq 2^{-w_1 -1} = 2^{-w_5 -1}
\]
\end{lemma}
\begin{proof}
We first verify that totality and \( 0 \) density are forced by \( p_5 \).   It is enough to consider each component of \( O \) separately.   The claims for \( C \) with respect to \( p_5 \) is proved in lemma \ref{lem:C-density-above-l} and since \( C \) is defined in terms of \( \delta \triangleright S \) we can apply lemma \ref{lem:hidden-no-effect} to infer that these claims are also forced by \( p^y_5 \) and \( p^z_5 \).  Since \( q_1 \in \Qcond \) and \( p_0 \forces \ED \) we have \[
p_3 \forces \udensity(\recset(f \Tplus S){E_1}) = 0  \land \udensity(\set{x}{\recfnl{\e}{f}{x} = \Box  }) = 0 \land \recset(f \Tplus S){E_1}\conv \land \recfnl{\e}{f}{}\conv  
\] 
Appealing to \eqref{eq:forcing-translation:shield} lets us infer from the above equation that \( p_5, p^y_5 \) and \( p^z_5 \) all force the following for \( v = y,z \).  
\[
\udensity(\recset(f^v \Tplus S){E_1}) = 0  \land \udensity(\set{x}{\recfnl{\e}{f^v}{x} = \Box  }) = 0 \land \recset(f^v \Tplus S){E_1}\conv \land \recfnl{\e}{f^v}{}\conv 
\]
This leaves only the claims about bounding to prove. By  lemma \ref{lem:C-density-above-l}  we have \( p_5 \forces \udensity[l_0](C) \leq 2^{w_1 -4} \)  and appealing to \eqref{eq:extend-q-small:q-condition-small} and \eqref{eq:extend-q-small:i-box-small} we see that (as \( \lh{\xi_2}, \lh{\tau_2} \leq l_0 \)) that  \( p_3 \) forces \( \udensity[l_0](\set{y}{X_{\e}(y) = \Box}) < 2^{-w_1-4} \) and \(  \udensity[l_0]({ \recset(f \Tplus S){E_1}  })  < 2^{-w_1-4}  \).  By the same appeal above to \eqref{eq:forcing-translation:shield} we can parlay this into the fact that all of \( p_5, p^y_5 \) and \( p^z_5 \) force the versions of these claims where \( f \) is replaced with \( f^y \) or \( f^z \).  Putting this together reveals that the conditions \( p_5, p^y_5  \) and \( p^z_5  \) all force \( \udensity[l_0]({O}) \leq 5 \cdot 2^{-w_1 -4} \leq 2^{-w_1 -1} = 2^{-w_5 -1} \) completing the proof. 
\end{proof}

We now show that we can build extensions \( q^v_6 \) of \( q^v_5 \)  that place \( O\restr{[l_0, \infty)} \) into \( U \) \mydash adopting the convention that \( v \) always ranges over \( \set{y,z} \).       
.   

\begin{lemma}\label{lem:add-o-to-cond}
If \( \xi_6 \in \bstrs \) is defined by  \[
\xi_6(x) = \begin{cases}
                    \xi_2(x) & \text{if } x < \lh{\xi_2} \\
                    1 & \text{if } \lh{\xi_2} \leq x < l_0 \land x \in \recset(p^y_5){E} \union \recset(p^z_5){E} \\
                    0 &  \text{if } \lh{\xi_2} \leq x < l_0 \land x \nin \recset(p^y_5){E} \union \recset(p^z_5){E}\\
                    \diverge & \text{if } x \geq l_0
                \end{cases}
\]
and \( p_6 = p_5 \)  \( w_6 = w_5 = w_1 \) and \( E_6 = E_1 \union \set{o} \) then \( q^v_6 \Qgeq q^v_5 \).
\end{lemma}
We  \textbf{don't} claim that \( q_6 \) is even in \( \Qcond \) but we define it for notational convenience.  
\begin{proof}
First we check that \( (\xi_6, w_1) \Igeq (\xi_2, w_1) \).  By our definition above we see that 
\begin{align*}
\udensity[\lh{\xi_2}][\lh{\xi_6}]({\xi_6}) &= \udensity[\lh{\xi_2}][l_0]({\xi_2 \union \recset(p^y_5){E}\restr{[\lh{\xi_2}, l_0)} \union \recset(p^z_5){E}\restr{[\lh{\xi_2}, l_0)} }) \leq \phantom{X}\\
&\udensity[\lh{\xi_2}][l_0]({\xi_2 \union \recset(p^y_5){E}\restr{[\lh{\xi_2}, l_0)} }) + \udensity[\lh{\xi_2}][l_0]({\xi_2 \union \recset(p^y_5){E}\restr{[\lh{\xi_2}, l_0)} }) \leq 2 \cdot 2^{-w_1-4}  \leq  2^{-w_1}
\end{align*}
where the last line follows from the fact that \eqref{eq:extend-q-small:q-condition-small} tells us \[
p_2 \forces \udensity[\lh{\xi_2}]({ \xi_2 \union \recset(f \Tplus S){E_1}\restr{[\lh{\xi_2}, \infty)}  }) \leq 2^{-w_1-4} 
\]
and appealing to \eqref{eq:forcing-translation:shield}  of lemma \ref{lem:forcing-translation} shows us that \(  p^v_4 \) forces the versions with \( f \) replaced with \( f^y \) and \( f^z \).   As \( w_1 \) is shared by both \( (\xi_6, w_1) \) and \( (\xi_2, w_1) \) this is enough to demonstrate both that \( (\xi_6, w_1) \in \Icond \) and \( (\xi_6, w_1) \Igeq (\xi_2, w_1) \).  We clearly have \( \xi_6 \supset \recset(p^v_5){E}  \) showing we satisfy \eqref{eq:q-extension-equiv} of lemma \ref{lem:q-condition-equiv}.  Parts \ref{def:q-condition:density-zero} and \ref{lem:q-condition-equiv:sets-total} of that lemma's conditions for extension are guaranteed by lemma \ref{lem:o-small} and the fact that \( q_5 \in \Qcond \).  In lemma \ref{lem:ramsey} we explicitly built \( p^v_5 \) to force convergence of \( \recset(f \Tplus S){E_1} \) on an initial segment of length \( l_0 \).  To see that \( O \) also converges on \( [0, l_0) \) we put together the fact that \( p_4 \) ensures  \( C \) does in lemma \ref{lem:C-density-above-l} \mydash since \( C \) ignores \( f \) on \( F \) so does \( p^v_6 \) \mydash  and that the other parts of \( O \) were guaranteed to have sufficent length in lemma \ref{lem:ramsey}.  This verifies part \ref{lem:q-condition-equiv:length} of lemma \ref{lem:q-condition-equiv} leaving us only to verify that \[ 
p^v_6 \forces \udensity[l_0]({ \xi_6 \union \recset(f \Tplus S){E_6}\restr{[l_0, \infty)} }) \leq 2^{-w_6} 
\]  
Finally, we observe that \mydash letting \( v' \in \set{y,z} \) with \( v' \neq v \) and understanding the entire equation to be evaluated in the context of forcing by \( p^v_6 \) \mydash 
\begin{align*}
p^v_6 \forces &\udensity[l_0]({ \xi_6 \union \recset(f \Tplus S){E_6}\restr{[l_0, \infty)} }) \leq  \udensity[l_0]({\xi_2 \union \recset(p^y_5){E_1}\restr{l_0} \union \recset(p^z_5){E_1}\restr{l_0} \union \recset(f \Tplus S){E_1}\restr{[l_0, \infty)} \union O\restr{[l_0, \infty)} })\leq \phantom{X} \\
&\frac{\card{\xi_2}}{l_0} + \udensity[l_0]({ \recset(f \Tplus S){E_1} }) + \udensity[l_0]({ \recset(p^{v'}_5){E_1}\restr{l_0} }) + \udensity[l_0]({ O\restr{[l_0, \infty)} }) \leq \phantom{X} \\
& 2^{-w_1 -4} + 2^{-w_1 -4}  + 2^{-w_1 -4} + 2^{-w_1 -1} \leq 2^{-w_1} = 2^{-w_6}
\end{align*}
where the final line is (in order) by lemma \ref{lem:C-density-above-l}, equation \eqref{eq:extend-q-small:q-condition-small} (as \( p^v_6, p^{v'}_6 \Pgeq p_2 \)) and lemma \ref{lem:o-small}.  This completes the proof. 
% \udensity[\lh{\xi_2}]({ \xi_2 \union \recset(f \Tplus S){E_1}  }) \leq 2^{-w_1-4} \label{eq:extend-q-small:q-condition-small
\end{proof}

We fix \( q^y_6, q^z_6 \) having all but finished building our conditions.  

\subsection{Verification}

In this subsection we show that we really have hidden all disagreement in \( U \) and use \( q^y_6 \) and \( q^z_6 \) to produce the condition required by proposition \ref{prop:extendable-conditions}.  First, we observe that our sets \( O \) and \( C \) have the same members when defined with respect to \( f^y \) or \( f^z \).     

\begin{lemma}\label{lem:same-X-U}
If \( q^y_7 \Qgeq q^y_6 \) and \( q^y_7 \forces X_{\e}^{\Box U}(x) = a \) then there is some \( q^y_8 \Qgeq q^y_7 \) with \( q^z_8 \forces X_{\e}^{\Box U}(x)\conv = a \). 
\end{lemma}
This asserts \( X^{y, \Box U}_{\e} = X^{z, \Box U}_{\e} \).  The asymmetry is purely a matter of presentation as the same result holds with \( y \) and \( z \) swapped.
\begin{proof}
We first note that the symmetry of \( O \) and the fact that we placed  \( \recset(f^y \Tplus S){E_1} \union \recset(f^z \Tplus S){E_1} \) into  \( O \) along with lemma \ref{lem:forcing-translation} ensures that \( q^y_7 \Qgeq q^y_6 \) iff \( q^z_7 \Qgeq q^z_6 \) and, likewise, \( q^z_8 \Qgeq q^z_7 \) iff \( q^y_8 \Qgeq q^y_7 \).  Thus, we can simply extend \( q^z_7 \) to \( q^z_8 \) so that it forces \( q^z_8 \forces X_{\e}^{\Box U}(x)\conv =b \) and argue that we can't have \( a \neq b \).  Now if \( x \geq l_0 \) and either \( a \) or \( b \) was \( \Box \) then our definition of \( O \) and the requirement that if \( q^v_8 \) forces \( \recset(f \Tplus S){E_6}(x) = 1 \) then \( \xi_8(x) = 1 \) \mydash recall that to force our existential fact we need \( x < \min(\lh{\xi_8}, \lh{\tau_8}, \lh{\sigma_8}) \) \mydash  and therefore \( a = b = \Box \).  If \( x < l_0 \) then the result is guaranteed by lemma \ref{lem:ramsey}.  Thus, it is enough to prove that if \( a, b \in \omega \) and \( a \neq b \) then \( x \in C \) \mydash note that by lemma \ref{lem:hidden-no-effect} if either \( q^z_8 \) or \( q^y_8 \) forces \( x \in C \) so does the other.  

So suppose that \( a \neq b \) and both are in \( \omega \) and that neither \( q^z_8 \) or \( q^y_8 \) forces \( x \in C \).  Since \( C \subset O \) we can assume they both force \( C(x)\conv = 0 \).      Thus we have both \( q^y_8 \) and \( q^z_8 \)  force   \( \recfnl{\i}{f^{\Box \delta \triangleright f}}{x}\conv \) and by lemma \ref{lem:hidden-no-effect} they both force  \( \recfnl{\i}{f^{\Box \delta \triangleright f}}{x}\conv = c \).  However, we must have \( c = a \) and \( c = b \) since both \( q^y_8 \) and \( q^z_8 \) extend \( q_0 \) which prevents disagreement between \( \recfnl{\i}{f^{\Box S}}{} \) and \( \recfnl{\e}{f}{} \).                        
\end{proof}  

We are finally in a position to give our proof of proposition \ref{prop:extendable-conditions} which, by our argument in \S\ref{ssec:extending-conditions}, completes our proof of theorem \ref{thm:effective-dense}.

\begin{proof}

Recall we need to show that if \( \e, \i \) are set-indexes, \( q_0 \forces \ED \) and \( q_1 \Qgeq q_0 \) then there  is some \( q^{\ddagger}  \Qgeq q_1  \) such that one of the following obtains.  
\begin{enumerate}
    \item\label{prop:extendable-conditions:partial-repeat}    \( \displaystyle \exists(x)\left( q^{\ddagger} \forces \lnot \exists(s)\recfnl{\j}{X_{\e}^{\Box U}}{x}\conv[s] \right) \)
    \item\label{prop:extendable-conditions:not-small-repeat}   \( \displaystyle q^{\ddagger} \forces \exists(l > \lh{\tau_1})\exists(s)\left(\udensity[l][l](\set{y}{\recfnl{\j}{X_{\e}^{\Box U}}{y}\conv[s] = \Box}) \geq 2^{-k_0 -2}\right)  \) 
    \item\label{prop:extendable-conditions:disagree-repeat}  \(\displaystyle q^{\ddagger} \forces \exists(x)\exists(s)\left(\recfnl{\j}{X_{\e}^{\Box U}}{x}\conv[s] \nboxeq f(x) \right)  \) 
\end{enumerate} 

And we also need to show that \( \zerojj \) can compute \( q^{\ddagger} \) from \( q_0 \) and \( q_1 \).  Suppose that there is some \( q \Qgeq q^y_6 \) such that for some \( x \in F \) and \( a \in \omega \)   \[ 
q \forces  \exists(x)\exists(s)\left(\recfnl{\j}{X_{\e}^{\Box U}}{x}\conv[s] = a \right)
\]     
Since this is a \( \sigmazn(f, S, U){1} \) sentence then we must also have this sentence is forced by \( (\tau, \sigma, k_6, \xi, w_6, E_6) \) so therefore some condition of the form \( q^y_7 \).  If \( a \neq y \) then we can take \( q^{\ddagger} = q^y_7   \) which clearly satisfies \ref{prop:extendable-conditions:disagree-repeat} above.  If \( a = y \) then, by lemma \ref{lem:same-X-U}, there is some \( q^z_8 \Qgeq q^z_6 \) which also forces the same and we can take \( q^{\ddagger} = q^y_8   \) to satisfy  \ref{prop:extendable-conditions:disagree-repeat}.  

Now suppose that there is some condition \( q \Qgeq q^y_6  \) which, for all \( x \in F \) \[ 
q \forces  \exists(x)\exists(s)\left(\recfnl{\j}{X_{\e}^{\Box U}}{x}\conv[s] = \Box \right)
\]         
Again, this sentence is forced by a condition of the form \( q^y_7 \).  By definition \ref{def:F-interval} if we take \( l = \lh{\sigma_2} + l_F \) this condition satisfies  part \ref{prop:extendable-conditions:not-small-repeat} above.  This leaves only the case where for some \( x \in F \) no condition \( q \Qgeq q^y \) forces  \[ 
q \forces \forces \exists(x)\exists(s)\left(\recfnl{\j}{X_{\e}^{\Box U}}{x}\conv[s] = \Box \right)
\]  
In this case we take \( q^{\ddagger} = q^y_6 \), using the definition of forcing, we have that \[
 q^{\ddagger} \forces \lnot \exists(s)\recfnl{\j}{X_{\e}^{\Box U}}{x}\conv[s]
\]
and therefore satisfies part \ref{prop:extendable-conditions:partial-repeat}.  

To verify that \( \zerojj \) can compute \( q^{\ddagger} \) it is enough to note that lemma \ref{lem:q-forcing-lvl-one} an examination of our construction shows that \( \zerojj \) can build both \( q^y_6 \) and \( q^z_6 \).  To see this, note that nowhere in that construction did we need to check for some arbitrary index \( e \) if we forced \( \recset(f \Tplus S){e} \) to be total and have density \( 0 \), i.e., every time we need to check if something is forced we can take advantage of part \ref{lem:q-forcing-lvl-one:qcond-assume-forces} of lemma \ref{lem:q-forcing-lvl-one}.  Finally, we note that in this proof we can determine which option we take merely by asking about the existence of conditions of the form \( q^y_7 \) forcing \( \sigmazn(f, S, U){1} \) sentences.  Thus, \( \zeroj \) can determine which of the cases obtains and, therefore, given \( q_1 \Qgeq q_0 \) we can produce \( q^{\ddagger} \) computably in \( \zerojj \) as required.        
\end{proof}

\section{Coarse Degrees Contain Sets}\label{sec:coarse-sets}

We now shift our attention to coarse degrees.  We first show that, in contrast to effective dense degrees, even every uniform coarse degree contains a set \mydash and, indeed, that they do so in the most uniform way possible.  

\begin{theorem}\label{thm:coarse}
There are computable functionals \( \Gamma\maps{\baire}{\cantor}  \) and \(  \hat{\Gamma}\maps{\cantor}{\baire} \) such that for all \( f \in \baire \), \( \Gamma \) is a uniform coarse reduction of \( \Gamma(f) \) to \( f \), \(  \hat{\Gamma} \circ \Gamma = \id  \)  and \( \hat{\Gamma} \) is a uniform coarse reduction of \( f \) to \( \Gamma(f) \).         
\end{theorem}

It might not be initially obvious that it makes sense to look for such a uniform solution to this problem.  However, we were lead to this solution by an attempt to prove that sufficiently generic functions would fail to be equivalent to a set via an argument similar to what we used in section \ref{sec:edd-no-sets} (only later after realizing it failed for coarse degrees did we realize it might succeed for effective dense degrees) and it isn't hard to see that any functionals witnessing that some sufficiently generic function \( f \) is of the same uniform coarse degree as a set actually produces functionals like those above which work for every sufficiently generic function \( f \).  

Also, as we prove in \S\ref{sec:coarse-sets:ssec:splitting-into-pieces},
\( S \in \mathbb{I} \) is equivalent to demanding the fraction of elements in the finite interval \( I_n \) that are also in \( S \) go to \( 0 \).  This suggests that, if there really is some way to code every function \( f \) into a set in a way that preserves coarse descriptions, it should be possible to define such a function based only on the behavior of \( S \) on some finite intervals \( I_n \) partitioning \( \omega \setminus \set{0} \).  This strongly suggests that either there is some way to define \( f \) on the intervals \( I_n \) which can't be coded in a set in a way that preserves the coarse degree, e.g., like the construction we saw in section \ref{sec:edd-no-sets}, and therefore this equivalence should fail for sufficiently generic \( f \) or that whatever functionals work for all generic functions \( f \) work for all functions.     

The motivation for first trying to prove this for uniform coarse equivalence rather than starting with the, seemingly easier, case of non-uniform coarse equivalence was the vague outline of the argument later developed in section \ref{sec:coarse-complexity} where we show that \mydash unlike non-uniform effective dense reductions \mydash all coarse reductions are uniform when restricted to \( \set{\hat{f}}{\udensity[l](f \symdiff \hat{f})} \leq 2^{-k} \) for some \( k, l \).  Given that we are both building the reduction and defining the set \( X \) it seems plausible that either there is or isn't a way to encode \( f \) on \( I_n \) into a set in a way that respects density and that therefore it makes sense to look for an entirely uniform solution.  

The reason we bother explaining this motivation \mydash after all we do prove it so what does it matter why we tried to prove the theorem in this form \mydash is that knowing that we were looking for such a uniform solution will motivate looking for functionals with certain properties in the next subsection.  This motivation therefore explains why it makes sense to think there might be such functionals.

\subsection{What Success Looks Like}\label{ssec:challenge}

Imagine \( \Gamma \) encoded all the information about the behavior of \( f \) on a positive density set into \( 0 \) density subset of \( X \).  This would allow a density \( 0 \) change in \( X \) to cause a positive density change in \( \hat{\Gamma} \) (or failure to be an inverse).  Thus,  \( f_0 \symdiff f_1 \) will be positive density iff \( X_0 \symdiff X_1 \) is and vice-versa.  However, as  \( \Gamma, \hat{\Gamma} \) must be computable they must decide how to encode information about \( f, X \) with only finite information.  It isn't hard to see that this means \( \Gamma, \hat{\Gamma} \) must try to  keep  \(\udensity[l][l](f_0 \symdiff f_1) \) and   \( \udensity[\hat{l}][\hat{l}](X_0 \symdiff X_1)  \) somehow proportional, i.e., every time we see   \(\udensity[l_n][l_n](f_0 \symdiff f_1)   > \epsilon \) for some large \( l_n \)  we should ensure that  \( \udensity[\hat{l}_n][\hat{l}_n](X_0 \symdiff X_1) > \delta \).  If we define the local density of \( X \)  to be \( \udensity[0](X) \) we will describe this intuition as saying that changes to \( f \) should result in changes to \( X \) of proportional local  density with the understanding that proportional doesn't mean linear \mydash though in the actual construction it will be \mydash and that any function which goes to \( 0 \) as its input goes to \( 0 \) would suffice.         

However, this is easier said than done.  Changing   \( f(x) \) should change \( X \) on a set of local density proportional to  \( \frac{1}{x+1} \) but we need to use an infinite subset of \( X \) to encode \( f(x) \). One natural idea is to use disjoint class of sets \( C_x \) and encode \( f(x) = y \) by encoding \( y \) into the bits of  \( X \isect C_x \), e.g., by setting the first \( r(y) \) values of \( C_x \) to be \( 1 \). While this initially seems to work, it falls apart once we consider changing \( f \) at multiple locations simultaneously.  

The problem is that if we make both \( f(n) \neq f'(n) \)  very large we can delay the point at which we observe any resulting difference in \( \Gamma(f) \) and \( \Gamma(f') \) arbitrarily long.  Thus, if \( f  \) is sufficiently fast growing and \( f'(n) = f(n)+1 \) for all \( n \) then, by the time \( l \) is large enough to see any  difference between  \( \Gamma(f)\restr{l} \) and \( \Gamma(f')\restr{l} \)  in  \( C_{n+1} \)  the number of differences in  \( C_{m}, m < n  \) is a negligible fraction of \( l \).   Thus, even though \( f, f' \) might disagree everywhere they can so thoroughly space out the locations at which that disagreement occurs that \( X \symdiff X' \) could be a set of density \( 0 \).  Indeed, this kind of consideration shows that \textit{no} method that encodes \( f \) at each \( x \) independently can work.   Ultimately, this motivates our `borrowing' strategy in \S\ref{ssec:define-encoding} however, we first show that while we can't deal with \( f \) separately on each individual argument we can do so with appropriately sized intervals.

% While this might seem like a problem with one particular way to define \( \Gamma \) it actually works to show that \textit{no} strategy\footnote{The idea is essentially that described in the previous paragraph but we inductively build up \( f, f' \) from equal length strings \( \sigma_n, \sigma'_n \).  By using an argument like that used to prove Ramsey's theorem we can find an infinite set \( I \) such that \( \Gamma(\sigma_n\concat[x]\concat\str{0}^{\infty}) \) approaches some limit as \( x \in I \) goes to infinity and likewise for \( \sigma'_n \).  As above, this lets us ensure that each location \( f \) and \( f' \) disagree at is isolated from all the others   and since \( \Gamma \) can't tell the difference between \( \sigma_n\concat[x] \) and \( \sigma'_n\concat[y] \) both followed by all \( 0 \)s and both followed by an arbitrary function this is enough to bound \( \udensity(\Gamma(f) \symdiff \Gamma(f')) \).} in which \( \Gamma(f) \) works by considering the value \( f \) takes on each \( n \) separately \mydash indeed, \( \Gamma(f) \symdiff \Gamma(f') \) can't be a function only of \( S = f \symdiff f' \), \( f\restr{S} \) and \( f'\restr{S} \).     

% Since \( \Gamma \) is computable this means that we  

% Thus, we .  While we don't need to preserve any exact balance between 

\subsection{Splitting Into Pieces}\label{sec:coarse-sets:ssec:splitting-into-pieces}

We start by showing that having density \( 0 \) can be defined in terms of the behavior of that set on a collection of disjoint intervals.   Such intervals can neither be too large  nor too small with the right balance struck by having each interval the same size as all previous intervals as follows.
\begin{definition}\label{def:In}
\begin{alignat*}{2}
\symbf{I}_n &= [0, 2^{n+1}) \hskip 8em\relax  && I_n = [2^{n}, 2^{n+1}) 
\end{alignat*}
\end{definition}
To help talk about the behavior of a function on an interval we also introduce the following notation.
\begin{definition}
If \( I = [a, a+n) \) then \( f/I \) is the length \( n \) defined by  \( (f/I)(x) = f(x + a) \).
\end{definition}
In other words, \( f/I \) takes the restriction of \( f \) to \( I \) and turns it into a string by shifting its domain to start at \( 0 \).  We now prove that the intervals \( I_n \)  have the property claimed above.   

\begin{lemma}\label{lem:frac-to-density}
 \[ \udensity(S) = 0 \iff \limsup_{n \to \infty} \frac{\card{S \isect I_n}}{\card{I_n}} = 0 \] 
\end{lemma}
\begin{proof}
We  observe that    
\[ 
 \limsup_{n \to \infty} \frac{\card{S \isect I_n}}{\card{I_n}} =\limsup_{n \to \infty} \frac{\card{S \isect I_n}}{2^{n+1}} \leq \frac{1}{2} \limsup_{n \to \infty} \frac{\card{S\restr{2^{n+1}}}}{2^{n+1}}  \leq \frac{\udensity(S)}{2} \]
so it is enough to show that
\[ 
    \udensity(S) = r > 0 \implies \limsup_{n \to \infty} \frac{\card{S \isect I_n}}{\card{I_n}} > 0  
\]
Suppose the antecedent holds and let \( k > 2 \) such that \( r > 2^{-k+2} \).  Given an arbitrary \( n_0 \geq 0 \) choose \( n_1 > n_0 \) such that \( \udensity[2^{n_1}](S) \in [r, r + 2^{-k}) \) and choose \( l, n \) with  \( n >  n_1 \) and \( l \in I_{n+1} \) such that \( \frac{\card{S\restr{l}}}{l} > r - 2^{-k-1} \).  Since \( n_0 \) is arbitrary, it is enough to show that either \( \frac{\card{S \isect I_n}}{\card{I_n}}  \) or \( \frac{\card{S \isect I_{n+1}}}{\card{I_{n+1}}}  \) is greater than \( 2^{-k-2} \).  We start by splitting up \( S\restr{l} \) into \( 3 \) pieces.
\begin{equation}\label{eq:frac-to-density:belowi}
r - \frac{1}{2^{k+1}} < \frac{\card{S\restr{l}}}{l} \leq  \frac{\card{S \isect \symbf{I}_{n-1}}}{l} + \frac{\card{S \isect I_{n}}}{l} +  \frac{\card{S \isect I_{n+1}}}{l}
\end{equation}
Since \( n -1  \geq n_1 \) and \( l \geq 2^{n+1} \) we have 
\begin{equation}\label{eq:frac-to-density:abovei}
\frac{1}{2}(r + \frac{1}{2^{k}}) > \frac{1}{2}\udensity[2^{n - 1}](S) \geq \frac{1}{2}\frac{\card{S \isect \symbf{I}_{n-1}}}{2^n} \geq \frac{\card{S \isect \symbf{I}_{n-1}}}{l}
\end{equation}

Flipping \eqref{eq:frac-to-density:abovei} and subtracting it from \eqref{eq:frac-to-density:belowi} yields
\[
\frac{\card{S \isect I_{n}}}{l} +  \frac{\card{S \isect I_{n+1}}}{l} > r - \frac{1}{2^{k+1}} - (\frac{r}{2} + \frac{1}{2^{k+1}})
\]
Simplifying and using the fact that \( r > \frac{1}{2^{k-2}} \) gives us 
\begin{equation}\label{eq:frac-to-density:IandIplus}
\frac{\card{S \isect I_{n}}}{l} +  \frac{\card{S \isect I_{n+1}}}{l} > \frac{1}{2^{k}} 
\end{equation}
Now if \( \frac{\card{S \isect I_{n}}}{l} > \frac{1}{2^{k+1}}   \) then since \( l \geq 2^{n+1} > 2^n = \card{I_n} \) we would have
\[
\frac{\card{S \isect I_{n}}}{\card{I_n}} \geq \frac{\card{S \isect I_{n}}}{l} > \frac{1}{2^{k+1}}  > \frac{1}{2^{k+2}} 
\]
as desired.  Thus, we may assume that  \( \frac{\card{S \isect I_{n}}}{l} \leq \frac{1}{2^{k+1}}   \).  Substituting this into \eqref{eq:frac-to-density:IandIplus} gives us
\[
\frac{1}{2^{k+1}}  +  \frac{\card{S \isect I_{n+1}}}{l} > \frac{1}{2^{k}} \implies \frac{\card{S \isect I_{n+1}}}{l} > \frac{1}{2^{k+1}}
\]
Using the fact that \( l \geq 2^n \) and dividing by \( 2 \) lets us infer the desired inequality
\[
\frac{\card{S \isect I_{n+1}}}{2l} \geq  \frac{\card{S \isect I_{n+1}}}{2^{n+1}} = \frac{\card{S \isect I_{n+1}}}{\card{I_{n+1}}} > \frac{1}{2^{k+2}}
\]

\end{proof}

We will use the above lemma to define \( \Gamma(f) \) in terms of the application of \( \Gamma \) to  \( f\restr{I_n} \).  Specifically, we will define \( X \) on the intervals \( L_{\pair{n}{s}} \) solely in terms of \( f\restr{I_n} \).  We will space out the intervals \( L_i \) to ensure that there is no `crosstalk' between the coding intervals \( L_i \) and \( L_j \) by defining long fallow intervals \( L^{-}_{i+1} \) on which nothing happens.  We will also need to define a way to repeat information the appropriate number of times so we can have a similar sized effect on the local density of \( X \) regardless of what interval \( i \) we are changing.   That is the purpose of the next definition

\begin{definition}\label{def:on-interval}
\begin{align*}
\modr[r](\sigma) &\eqdef \tau \text{ where } \lh{\tau} = \lh{\sigma}\cdot 2^r \land \forall(x < \lh{\sigma})\forall(k < 2^r) \left[\tau(x + k\lh{\sigma}) = \sigma(x)\right] \\
\modrinv[r](\tau;x) &\eqdef \begin{cases}
                        \diverge & \text{if } x \geq l \\
                        1 & \text{o.w. if } \card{\set{k < 2^r}{\tau(x + k l) = 1}} > 2^{r -1} \text{ where } \lh{\tau} = l2^r \\
                        0 & \text{otherwise } \\
                      \end{cases} 
\end{align*}
\end{definition}
The definition of \( \modrinv[r](\tau) \) presumes  \( \tau \) is a binary string of length  \(l\cdot 2^r \) for some \( l \) and returns a binary string of length \( 2^l \).

We observe, for later use, the following facts.
\begin{lemma}\label{lem:modr}
If \( \sigma, \tau \in 2^n \) and \( \chi \in 2^{n\cdot 2^r} \) 
\begin{enumerate}
    \item\label{lem:modr:symdiff}  \( \modr[r](\sigma) \symdiff \modr[r](\tau) = \modr[r](\sigma \symdiff \tau) \).
    \item\label{lem:modr:times-two-r} \( \card{\modr[r](\sigma)} = 2^r\card{\sigma} \) 
    \item \( \modrinv[r](\modr[r](\sigma)) = \sigma \).
    \item\label{lem:modr:modinv-times} \( \card{\chi} > 2^{r -1}\card{\modrinv[r](\chi)} \) 
    \item\label{lem:modr:inv-symdiff} \( 2^{r-1}\card{\sigma \symdiff \modrinv[r](\chi)} \leq  \card{\modr[r](\sigma) \symdiff \chi} \)    
\end{enumerate}
\end{lemma}
\begin{proof}
Only the last two items aren't immediate from the definitions.  For the second to last item, observe that if \( x \in  \modrinv[r](\chi) \) then \( \chi \) must be \( 1 \) on more than \( 2^{r-1} \) values all equal to \( x \) modulo \( n \).  To see the last item holds, suppose that \( x \in  \sigma \symdiff \modrinv[r](\chi) \) and note this requires \( \modr[r](\sigma)  \) disagrees with \(  \chi \) on at least \( 2^{r-1} \) locations (equality can only hold if \( \sigma \) is \( 1 \) everywhere). 
\end{proof}

To interpret the next lemma, think of \( \xi_i \) as some binary string we wish to encode (whose length is \( 2^{n_i} \)) and  \( 2^{-b_i} \) as  the allowable error, \( L_i  \) as the interval on which we code \( \xi_i \), \( L^{-}_i \) as the fallow area we skip between \( L_{i-1} \) and \( L_{i} \)  and \( r_i \) as specifying how many times we need to repeat \( \xi_i \) so that changes to a certain fraction of the bits of \( \xi_i \) change the code on a similar density of bits in \( L_i \).  

\begin{lemma}\label{lem:seperate-and-repeat}
If for all \( i \), \( b_i, n_i \in \omega \), \( b_i > 0 \)  and \( \xi_i \) is a binary string of length \( 2^{n_i}  \)   and we define
\begin{alignat*}{2}
l^{-}_i &\eqdef 2^{b_i + n_i + l_{i-1}} = 2^{r_i - b_i} & \qquad \qquad r_i &\eqdef 2b_i + n_i + l_{i-1}  \\
l_{-1} &\eqdef 0 & l_i &\eqdef l^{-}_i  + 2^{n_i + r_i} \\
L_i &\eqdef [l^{-}_i, l_i) &  L^{-}_i &\eqdef [l_{i-1}, l^{-}_i)
\end{alignat*}
then for all \( i \geq 0 \),  \( l_{i-1} < l^{-}_i < l_i \), the sets \( L_i, L^{-}_i \) form a partition of \( \omega \) and the set \( X \) uniquely defined by \[
 X \isect L^{-}_i = \eset  \qquad \qquad  X/L_i = \modr[r_i](\xi_i)
\]
satisfies 
\[ \frac{\card{\xi_i}}{2^{n_i}} - \frac{1}{2^{b_i}} < \udensity[L_i](X) < \frac{\card{\xi_i}}{2^{n_i}} + \frac{1}{2^{b_i}}   \]
\end{lemma}
Recall that \( \udensity[L_i](X) \) is defined to be \( \udensity[l^{-}_i][l_i-1](X) \).  Thus, we can think of this lemma as showing us that we can compensate for the location we are encoding \( \xi_i \) at so it has the same effect on the local density of \( X \) as if it was encoded at any other location.  While we will only cite this lemma for the upper bound the lower bound will be used implicitly, i.e., derived from the definitions of our intervals.  
\begin{proof}
It is easy to see that for \( i \in \omega \),  \( r_i, l_{i-1}, l^{-}_i, b_i, n_i \) are always non-negative.  
Thus, as \( b_i > 0 \), we have \( l^{-}_i \geq 2^{1 + l_{i-1}} > l_{i-1} \) and \( l_i \geq l^{-}_i + 2^0 > l^{-}_i \).  As \( l_{-1} = 0 \) this entails that the sets \( L_i, L^{-}_i \) form a partition of \( \omega \).   The fact that \( X \) is uniquely defined follows from the fact that the sets \( L_i, L^{-}_i \) together form a partition of \( \omega \).  This leaves only the bounds to verify. 

We first verify the lower bound.  To that end, we observe that, by part \ref{lem:modr:times-two-r} of lemma \ref{lem:modr}
\begin{equation}\label{eq:seperate-and-repeat:below-start}
\udensity[L_i](X)  \geq \frac{\card{\xi_i} 2^{r_i}}{l_i} = \frac{\card{\xi_i} 2^{r_i}}{2^{r_i - b_i} + 2^{n_i + r_i} } = \frac{\card{\xi_i}}{2^{- b_i} + 2^{n_i} }
\end{equation}
We also note that \( \card{\xi_i} \leq 2^{n_i} \) and \( n_i \geq 0 \) so 
\[
0 > \frac{\card{\xi_i}2^{- b_i}}{2^{n_i}} - \frac{1}{2^{2b_i}} - \frac{2^{n_i}}{2^{b_i}} 
\] 
Adding \( \card{\xi_i} \) to both sides and doing some algebra gives us
\[
\card{\xi_i}  > \card{\xi_i}  + \frac{\card{\xi_i}2^{- b_i}}{2^{n_i}} - \frac{1}{2^{2b_i}} - \frac{2^{n_i}}{2^{b_i}} = \frac{\card{\xi_i}(2^{- b_i} + 2^{n_i})}{2^{n_i}} - \frac{2^{- b_i} + 2^{n_i}}{2^{b_i}} \] 
Dividing both sides by \( 2^{- b_i} + 2^{n_i} \) yields
\[
\frac{\card{\xi_i}}{2^{- b_i} + 2^{n_i} } > \frac{\card{\xi_i}}{2^{n_i}} - \frac{1}{2^{b_i}}
\]
Combining with \eqref{eq:seperate-and-repeat:below-start} yields the desired lower bound
\begin{equation*}
 \frac{\card{\xi_i}}{2^{n_i}} - \frac{1}{2^{b_i}} < \udensity[L_i](X) 
\end{equation*}

To verify the upper bound, suppose that \( l \in L_i  \) witnesses the value of \( \udensity[L_i](X)  \), i.e., \( \frac{\card{X\restr{l}}}{l} =  \udensity[L_i](X) \).  Choose  \( a < 2^{n_i} \) and \( b < 2^{r_i} \) such that \( l = l^{-}_i  + a + b2^{n_i} \).  This entails

\begin{equation}\label{eq:seperate-and-repeat:above-start}
\udensity[L_i](X) = \frac{\card{X\restr{l_{i-1}}} + \card{\modr[r_i](\xi_i)\restr{(a+b2^{n_i})}}}{ l^{-}_i  + a + b2^{n_i}} \leq \frac{l_{i-1}}{l^{-}_i  + a + b2^{n_i}}  + \frac{a + b\card{\xi_i} }{l^{-}_i  + a + b2^{n_i}} 
\end{equation}

Continuing this chain of inequalities 
\[
\leq \frac{l_{i-1}}{l^{-}_i}  + \frac{a}{l^{-}_i  + a + b2^{n_i}} +  \frac{b\card{\xi_i} }{l^{-}_i  + a + b2^{n_i}} \leq \frac{l_{i-1}}{l^{-}_i}  + \frac{a}{l^{-}_i} + \frac{\card{\xi_i} }{2^{n_i}} 
\]
Substituting in our definitions to the right-hand side of the above equation and using the fact that \( a < 2^{n_i} \) and \( n_i \geq 0 \)  gives us
\begin{equation*}
< \frac{\card{\xi_i} }{2^{n_i}} +  \frac{l_{i-1}}{2^{b_i + n_i + l_{i-1}} } + \frac{2^{n_1}}{2^{b_i + n_i + l_{i-1}}} \leq \frac{\card{\xi_i} }{2^{n_i}} +  \frac{l_{i-1}}{2^{b_i + l_{i-1}} } + \frac{1}{2^{b_i + l_{i-1}}}
\end{equation*}
It is easy to see that if \( l_{i-1} = 0 \) then the right-hand side is just \( \frac{\card{\xi_i} }{2^{n_i}} + \frac{1}{2^{b_i}} \).  So we assume that \( l_{i-1} \geq 1 \) \mydash noting this entails \( l_i \leq 2^{l_i - 1} \) \mydash and observe this becomes
\[
= \frac{\card{\xi_i} }{2^{n_i}} +  \frac{l_{i-1}}{2^{l_{i-1}}}\frac{1}{2^{b_i + 1} } + \frac{1}{2^{b_i + 1}} \leq \frac{\card{\xi_i} }{2^{n_i}} + \frac{2}{2^{b_i + 1}} = \frac{\card{\xi_i} }{2^{n_i}} + \frac{1}{2^{b_i}}
\]
So tracing the chain of inequalities back to \eqref{eq:seperate-and-repeat:above-start} yields the desired inequality.
\[
\udensity[L_i](X) < \frac{\card{\xi_i} }{2^{n_i}} + \frac{1}{2^{b_i}}
\]
\end{proof}

\subsection{Encoding Finite Strings}

Taking these two results together will allow us to construct \( \Gamma \) in pieces.  Specifically, for each interval \( I_n \) we will define an infinite sequence of intervals with the \( s \)-th such interval denoted by \( \encstr[s](f/I_n) \) and  use the above lemma to encode this into \( X \in \cantor \), i.e., if \( i = \pair{n}{s} \) then \( L_i \) encodes  \( \encstr[s](f/I_n) \).  As discussed in \S \ref{ssec:challenge} we want to ensure that changes to \( f \in \baire \) roughly translate to a similar density of changes in \( \Gamma(f) \).  In light of the lemmas above, we want to try to ensure that if \( \sigma = f/I_n \) and \( \sigma' = f'/I_n \) then   
\begin{equation*}\label{eq:encstr-sigma-diff}
\frac{\card{\sigma \symdiff \sigma'}}{\card{I_n}} \approx \sup_{s \in \omega}  \frac{\card{\encstr[s](\sigma) \symdiff \encstr[s](\sigma)}}{\lh{\encstr[s](\sigma)}}
\end{equation*}
While it isn't necessary \mydash almost any relationship ensuring one side uniformly goes to \( 0 \) when the other does would suffice \mydash we will construct our encoding so that there   are  non-zero constants that we can multiply each side by to bound the other.  If we keep \( \lh{\encstr[s](\sigma)}  \) bounded we can avoid the problem discussed in \S\ref{ssec:challenge} of separating the locations at which we encode changes too far apart.  While this doesn't force any particular size on us it will be convenient to fix \( \lh{\encstr[s](\sigma)} = \lh{\sigma} \).    Thus, if \( \lh{\sigma} = m \)  we will think of  \( \encstr(\sigma) \) as an element of \( (2^m)^\omega \), i.e., a function taking \( s \in \omega \) and yielding \( \encstr[s](\sigma)  \in  2^m \).   Given this, we make the following definition.

\begin{definition}\label{def:d-dist}
Given \( \theta, \theta' \in (2^n)^\omega  \) define 
\[
\ddist(\theta,\theta') = \sup_{s \in \omega} \card{\theta(s) \symdiff \theta'(s)}
\]
We extend this to \( \theta, \theta' \in (2^n)^{< \omega}  \) by assuming \( \theta(s) = \eset \) whenever \( s \nin \dom \theta \)  
\end{definition}

Since we will have  \( \encstr(\sigma) \in (2^n)^\omega  \) when \( \lh{\sigma} = n \) we can ignore the denominator and target the following goal.   
\begin{equation}\label{eq:finite-goal}
\ddist(\encstr(\sigma), \encstr(\sigma')) \approx \card{\sigma \symdiff \sigma'} 
\end{equation}
Since this is only an informal goal, we won't define \( \approx \) other than to say it suffices  to (and we will) find a constant \( c \) such that each side is less than or equal to \( c \) times the other.  Of course, this only helps us build a set that is of the same coarse degree as \( f \) if we also build  a computable inverse function \( \encinv(\theta) \) such that  
\begin{equation}\label{eq:finite-inv-goal}
\theta \in (2^n)^{\omega}, \sigma \in \omega^n \implies \card{\sigma \symdiff \encinv(\theta)} \approx  \ddist(\encstr(\sigma), \theta) 
\end{equation}
So far, this is a relatively easy demand to satisfy.  For instance, if we used \( s \) as the code of some \( \tau \in \omega^{\lh{\sigma}} \)  and understood \( \encstr[\tau](\sigma) \) to abbreviate \( \encstr[s](\sigma) \) where \( s \) is the code of \( \tau \) we could define (for \( x < \lh{\sigma} \))   
\begin{equation}\label{eq:encstr-fail-simple}
\encstr[\tau](\sigma;x) = \begin{cases}
                            1 & \text{if } \sigma(x) > \tau(x) \\
                            0 & \text{if } \sigma(x) \leq \tau(x)
                          \end{cases}
\end{equation}    
If we make the following definition 
\begin{definition}\label{def:coarse-set-notations}
 \(  \sigma \wwedge \tau \) is the componentwise minimum of the strings \( \sigma \) and \( \tau \) 
\end{definition}
then we can see this manages to satisfy \eqref{eq:finite-goal} since 
\[
\encstr[\sigma \wwedge \sigma'](\sigma) \symdiff \encstr[\sigma \wwedge \sigma'](\sigma') =  \sigma \symdiff \sigma'
\] 
However, it is not enough to ensure that changing \( f \) on a set of local density \( r \) causes \( X \) to change on a set of roughly proportional  density.  We also need to ensure that the impact of any finite change to \( f \) results in a change of overall density \( 0 \) to \( X \).  Or, in terms of our coding we need to ensure  that applied to strings of any particular length \( n \) we have 
\[
\lim_{s \to \infty} \card{\encstr[s](\sigma) \symdiff \encstr[s](\sigma')} = 0
\]
 Since this implies that for all large enough \( s \), \( \encstr[s](\sigma) \) is independent of \( \sigma \) we may as well define our encoding to satisfy
\begin{equation}\label{eq:encstr-to-zero}
\lim_{s \to \infty} \encstr[s](\sigma) = \eset
\end{equation}
This is a tricky requirement to satisfy since, as we saw in section \ref{sec:edd-no-sets}, this limit can't be independent of \( \sigma \) but it's not obvious how to do this without violating \eqref{eq:finite-goal}.  For instance, if we used some global condition like we stopped considering \( \tau \) once its maximum value was larger than the maximum value of \( \sigma \) we would satisfy \eqref{eq:encstr-to-zero} and we would still be guaranteed that if \( \sigma \) and \( \sigma' \) differed at many locations so would \( \encstr[\tau](\sigma) \) and \( \encstr[\tau](\sigma') \) for  \( \tau = \sigma \wwedge \sigma' \).  However, we can now build \( \sigma \) and \( \sigma' \) which differ only at \( 1 \) location and yet, because this changes when we compare them to \( \tau \), we have  \( \encstr[\tau](\sigma) \) and \( \encstr[\tau](\sigma') \) differ at \( \lh{\sigma} -1 \) locations for some \( \tau \).   On the other hand, if we try to only register whether  \( \sigma(x) > \tau(x)  \) when \( \sigma(x) \) is above the maximum value of \( \tau \) we again satisfy  \eqref{eq:encstr-to-zero} but now we can't guarantee that many differences between \( \sigma' \)  and \( \sigma \) must result in many differences between  \( \encstr[\tau](\sigma) \) and \( \encstr[\tau](\sigma') \) for some \( \tau \).

\subsubsection{Defining The Encoding}\label{ssec:define-encoding}    

Our solution will be to ensure that any change in \( \sigma \) at a single point only changes  \( \encstr[\tau](\sigma) \) at a single point while still ensuring that \( \ddist(\encstr(\sigma), \encstr(\sigma'))  \) grows with \( \card{\sigma \symdiff \sigma'}  \) we introduce the idea of borrowing.  Specifically,  we start with the idea that we only register whether \( \sigma(x) > \tau(x) \) iff the maximum of \( \tau \) is bounded by \( \sigma(x) \) but instead of only looking at \( \sigma(x) \) we allow borrowing the value of \( \sigma \) at some other location \( x' \).  By only allowing a single location \( x \) to borrow the value of \( \sigma(x') \) we avoid the problem we ran into above when we compared the maximum value of \( \sigma \) to the maximum value of \( \tau \).   

To this end, we allow the index \( s \) of \( \encstr[s](\sigma) \) to encode two objects, our string \( \tau \in \omega^n \) and a function \( q \in n^n \), i.e., a map from \( \set{0, \ldots, n -1} \) to itself.  We now make the following definition  \mydash recalling that \( q \) is idempotent if  \( \forall(x)\left(q(q(x)) = q(x)\right) \) and that our convention is to assume \( \sigma \in \omega^n \) implies \( n > 0 \).  

\NewDocumentCommand{\bndb}{s}{\IfBooleanTF{#1}{\vec{\symbf{\mathfrak{b}}}}{\symbf{\mathfrak{b}}}}
\begin{definition}\label{def:encstr-real}
Given \( \tau, \sigma \in \omega^n  \) and \( q \in n^n \) we define
\begin{align*}
\bndb(q, \sigma, x) &= \max \set{-1} \union \set{\sigma(y)}{y < n \land q(y) = x} \\ 
\encstr[\tau][q](\sigma;x) &= \begin{cases}
                            1 & \text{if } \sigma(x) > \tau(x) \land  \max_{i < n} \tau(i) < \bndb(q, \sigma, x)   \land  q \text{ is idempotent }\\
                            0 & \text{otherwise } 
                          \end{cases} %\label{eq:encstr-real}
\end{align*}  
\end{definition}

We turn this definition into \( \encstr[s](\sigma) \) as follows.

\begin{definition}\label{def:encstr-encoding} 
We define \( \encstr[s](\sigma) = \encstr[\tau][q](\sigma) \) when \( s = \code{\tau, q} \) where  
\begin{equation}\label{eq:code-tau-q-def}
\code{\tau, q} = \godelnum{\tau}\cdot n^n + \godelnum{q}  \text{ where } \godelnum{q} =-1 + \prod_{i = 0}^{n-1} q(i) 
\end{equation}
And we understand \( \godelnum{\tau} \) to be given by a bijection of \( \omega^{\lh{\tau}} \) with \( \omega \) that is monotonic in   \( \max_{i < n} \tau(i) \).  
\end{definition}
In other words, \( \code{\tau, q} \) if \( \max_{i < n} \tau(i) < \max_{i < n} \tau'(i) \) then \( \code{\tau, q} < \code{\tau', q'} \).  This has the following implication.
 %if \( \max_{i < n} \tau(i) < \max_{i < n} \tau'(i)  \) then  \( \code{\tau, q} < \code{\tau', q'} \) and
\begin{lemma}\label{les:weird-code-prop}
If \( \tau \in \omega^n \) then 
\begin{equation}\label{eq:weird-code-prop}
\code{\tau, q} < n^n\cdot (m+1)^n \iff \max_{i < n} \tau(i) \leq m
\end{equation}
\end{lemma}
\begin{proof}
Recall that we specified that if \( \max_{i < n} \tau(i) < \max_{i < n} \tau'(i) \) then \( \godelnum{\tau} < \godelnum{\tau'} \).  Since \( \card{\set{\tau}{ \max_{i < n} \tau(i) \leq m }} = (m+1)^n \) 
and our coding function is a bijection of \( \omega^n \) and \( \omega \) it follows that \( \godelnum{\tau} < (m+1)^n  \) (\( 0 \) is a valid code)  iff \( \max_{i < n} \tau(i) \leq m  \).  The claim now follows from \eqref{eq:code-tau-q-def}. 
\end{proof}

If  \( q \) is idempotent then \( q(x) = x \) iff \( q(y) = x \) for some \( y \).  This allows us to categorize each \( x < n \) as either a (potential) borrower (\( q(x) = x \) ) or lender (\( q(x) \neq x \)) but never both.  It also guarantees that \( \bndb(q, \sigma, x) \geq 0  \) iff \( q(x) = x \).   Since \( \encstr[\tau][q](\sigma) = \eset \) unless \( q \) is idempotent from this point on  we will tacitly assume \( q \) only ranges over idempotent functions.

\begin{lemma}\label{lem:encstr-to-zero}
If \( \sigma \in \omega^n \) then \( \forall*(s)\left(\encstr[s](\sigma) = \eset\right) \).  In particular, if \( m = \max_{i < n} \sigma(i) \) and \( s \geq n^n\cdot m^n \) then \( \encstr[s](\sigma) = \eset \) .   
\end{lemma}
\begin{proof}
By \eqref{eq:weird-code-prop} (applied to \( m -1 \)) if \( \code{\tau, q} = s \geq  n^n\cdot m^n \) then \( \max_{i < n} \tau(i) \geq m \).   For all \( x, q \), \( \bndb(q, \sigma, x) \leq  \max_{i < n} \sigma(i)  \)  and thus \( \bndb(q, \sigma, x) \leq \max_{i < n} \tau(i)  \).  Hence  \( \encstr[s](\sigma) = \eset \).
\end{proof}

We now prove a quick utility lemma.

\begin{lemma}\label{lem:idemp-conseq}
If \( \encstr[\tau][q](\sigma;x) \neq \encstr[\tau][q](\sigma';x) \) then \( \sigma(y) \neq \sigma'(y) \) for some \( y \in q^{-1}(x) \).   
\end{lemma}
\begin{proof}
WLOG assume that \( \encstr[\tau][q](\sigma;x) < \encstr[\tau][q](\sigma';x)   \).      On this assumption,  \( \sigma'(x) > \tau(x) \) and \( \max_{i < n} \tau(i) < \bndb(q, \sigma', x)  \) \mydash implying \( q(x) = x \) \mydash  and either \( \sigma(x) \leq \tau(x) \) or \( \max_{i < n} \tau(i) \geq \bndb(q, \sigma, x)  \).  Thus, we either have \( \sigma(x) \neq \sigma'(x) \) or \( \bndb(q, \sigma, x) \neq \bndb(q, \sigma', x)  \) and \mydash as \( x \in q^{-1}(x) \) \mydash  in either case \( \sigma \) and \( \sigma' \)  disagree on some  \( y  \in q^{-1}(x) \). 
\end{proof}

\begin{lemma}\label{lem:bound-encstr}
Suppose that \( \sigma, \sigma' \in \omega^n  \) and that \( \card{\sigma \symdiff \sigma'} = d \) then
\[
\ceil{\frac{d}{2}} \leq \ddist(\encstr(\sigma),\encstr(\sigma')) \leq d
\] 
Moreover, if \( S = \sigma \symdiff \sigma' \),  \( \zeta = \sigma \wwedge \sigma' \)  and  \( \overbar{z} \) is the median of \( \zeta[S] \) then there is some \( q \in n^n \) and \(  \tau \in \omega^n \) with 
\[
\max_{i < n} \tau(i) \leq \overbar{z}  \land \ceil{\frac{d}{2}} \leq \card{\encstr[\tau][q](\sigma) \symdiff \encstr[\tau][q](\sigma')}
\]
\end{lemma}
This bound is actually tight but, as this fact isn't needed, we won't prove it.
% and that \( \sigma \symdiff \sigma' = \set{x_i}{i < d} \).  To verify the upper bound is tight let \( \sigma = \str{0}^n \), \( \sigma' = \str{1}^d\concat\str{0}^{n - d} \), \( q = \id \) and \( \tau = \str{0}^n \).  To verify the lower bound is tight consider \( \sigma_i, i < 2 \) where 
% \[
% \sigma_i(x) = \begin{cases}
%                 0 & \text{if } x < n - d \\
%                 2x + i & \text{otherwise }
%             \end{cases}
% \]
\begin{proof}
This is all trivial to check if \( n = 1 \), and our convention is to assume \( n > 0 \) when talking about \( \omega^n \), so we can assume \( n > 1 \) and that  \( S, \zeta, \overbar{z}, d \) are as in the statement of the lemma. We first verify the bound above.

If \( x \in  \encstr[\tau][q](\sigma) \symdiff \encstr[\tau][q](\sigma')  \) then by lemma \ref{lem:idemp-conseq} we must have \( S \isect q^{-1}(x) \neq \eset \).  Hence, \( q \) is a surjection from \( S \) to  \( \encstr[\tau][q](\sigma) \symdiff \encstr[\tau][q](\sigma')  \).  The bound follows from the assumption that \( \card{S} = d \).

To verify the bound below we partition \( S \) into three sets \( H, E, L \).   with \( H  = \set{x}{x \in S\land \zeta(x) > \overbar{z}} \), \( E  = \set{x}{x \in S\land \zeta(x) = \overbar{z}} \) and \( L = \set{x}{x \in S\land \zeta(x) < \overbar{z}} \).  By the definition of the median we must be able to find disjoint subsets \( E_0, E_1 \)  of  \( E \) so that \( \card{H \union E_1} = \card{L \union E_1} \geq \floor{\frac{d}{2}} \) (with one element leftover if \( d \) is odd).  Therefore   
\[
\min(\card{H}, \card{L}) + \card{E} \geq \ceil{\frac{d}{2}} 
\]
Now define \( q \) so that \( q\restr{\setcmp{H}} = \id \) and \( q[H] \) is a subset of \( L \) of size \( \min(\card{H}, \card{L}) \), i.e., if \( H \) is smaller \( q \) is an injection from \( H \) to \( L \) otherwise a surjection..  Set \( \tau\restr{L \union E} = \zeta\restr{L \union E}\) and let  \( \tau  \) be \( 0 \) everywhere else.  We note that (by construction) \( \max_{i < n} \tau(i) \leq \overbar{z} \) as required by the moreover claim.  We now argue that \( \encstr[\tau][q](\sigma;x) \neq \encstr[\tau][q](\sigma';x) \) whenever \( x \in E \union q[H] \).  Since \( E \union q[H] \subset S \) we can assume, WLOG, that \( \sigma'(x) > \sigma(x)  \) and therefore that \( \encstr[\tau][q](\sigma';x)  = 0 \).  As \( \sigma'(x) > \sigma(x) = \tau(x)  \) it is enough to show that \( \bndb(q, \sigma', x) > \overbar{z} \).  

If \( x \in E \) then \( \sigma'(x) > \overbar{z} \) and since \( q(x) = x \) we automatically have \( \bndb(q, \sigma', x) > \overbar{z} \).  If \( x \in q[H] \) then there is some \( y \in H \) with \( q(y) = x \). As  \( \sigma'(y) \geq \zeta(y) > \overbar{z}  \) we again have \( \bndb(q, \sigma', x) > \overbar{z} \).  Therefore,
\[
\ddist(\encstr(\sigma),\encstr(\sigma')) \geq \card{\encstr[\tau][q](\sigma) \symdiff \encstr[\tau][q](\sigma')} \geq \card{E \union q[H]} \geq \min(\card{H}, \card{L}) + \card{E} \geq \ceil{\frac{d}{2}} 
\]         
vindicating both the main claim and the moreover claim.
\end{proof}

\subsection{Inverting The Encoding}

Now that we've shown how to go from \( \sigma \in \omega^n \) to \( \encstr(\sigma) \in {2^n}^{\omega} \) we now show how to invert this map.  Since we need to be able to define \( \encinv(\theta) \) in a computable fashion even when \( \theta \in (2^n)^{\omega} \) isn't of the form  \( \encstr(\sigma) \) for any \( \sigma \) we need to decide how much of the infinite object \( \theta \) to look at while computing \( \encinv(\theta) \).  This next definition answers that question.

\begin{definition}\label{def:theta-hgt}
Given \( \theta \in (2^n)^{\omega} \) define \( \hgtenc{\theta} \) to be the least \( m \) such that for all \( s \in [n^n\cdot m^n, n^n\cdot (m+1)^n) \), \( \theta(s) = \eset \).  Define \( \overbar{\theta} \) by 
\[
\overbar{\theta}(s) = \begin{cases}
                        \theta(s) & \text{if } s < n^n\cdot (\hgtenc{\theta}+1)^n \\
                        \eset & \text{otherwise}
                    \end{cases}
\]   
\end{definition}
This definition will make more sense in light of the following lemma (where we treat \( (-1)^n \) as the empty set).
\begin{lemma}\label{lem:encstr-hgt}\hfill
\[\hgtenc{\encstr(\sigma)}  = \murec{m}{\forall(q \in n^n)\forall(\tau \in m^n \setminus (m-1)^n)\left(\encstr[\tau][q](\sigma) = \eset  \right)} = \max_{i< n} \sigma(i) \]
Moreover, \( \encstr(\sigma) = \overbar{\encstr(\sigma)} \).  
\end{lemma}
In other words, if we think of \( \theta \) as a garbled guess at some \( \encstr(\sigma) \), we are looking for the point at which \( \tau \), in the form of \( s \), has gotten so big (large maximum value) that we are done coding \( \sigma \) (at least with non-zero bits).   
\begin{proof}
The only part of the main claim that doesn't follow immediately from  (the proof of) lemma \ref{lem:encstr-to-zero} and \eqref{eq:weird-code-prop} is that if  \( m' <  \max_{i< n} \sigma(i) \) then there is some \( \tau \) with \( \max_{i< n} \tau(i) = m' \) and \( q \) such that \( \encstr[\tau][q](\sigma) \neq \eset \).  This is easily seen to hold by letting \( q = \id  \) and letting \( \tau(i) = m' \) for all \( i < n \)  since if \( \sigma(x) > m' \) we have \( \encstr[\tau][q](\sigma;x) = 1 \).  The moreover claim follows by combining the main claim and  \eqref{eq:weird-code-prop} .      
\end{proof}

We also observe the following
\begin{lemma}\label{lem:theta-equal-overbar}
If \( \tau \in \omega^n, q \in n^n, \theta \in (2^n)^{\omega}  \),  \( \max_{i < n} \tau(i) \leq \hgtenc{\theta}\conv \) and \( s = \code{\tau, q} \) then \( \theta(s) = \overbar{\theta}(s) \).  Furthermore,
\begin{equation}\label{eq:truncate-ddist}
\max_{i < n} \sigma(i) \leq \hgtenc{\theta}\conv \implies \ddist(\encstr(\sigma), \overbar{\theta}) = \sup_{s < n^n\cdot m^n} \card{\encstr[s](\sigma) \symdiff \theta(s)}
\end{equation}
\end{lemma}
\begin{proof}
The main claim follows immediately from \eqref{eq:weird-code-prop} and definition \ref{def:theta-hgt} and the furthermore follows from the main claim,  definition \ref{def:theta-hgt} and lemma \ref{lem:encstr-to-zero}.  
\end{proof}

\begin{definition}\label{def:encinv}
Given \( \theta \in (2^n)^{\omega} \) we define \( \encinv(\theta)  \) to be the string \( \sigma \in \omega^n  \) satisfying \( \max_{i < n} \sigma(i) \leq  \hgtenc{\theta}\) that minimizes \( \ddist(\encstr(\sigma), \overbar{\theta}) \) with ties broken in favor of the string with the smaller code.     
\end{definition}

We first note that this is a computable function which is defined for the class of \( \theta \) we are concerned with.   Specifically, those \( \theta \) such that \( \theta(s) = \eset \) for co-finitely many \( s \).  To avoid constantly repeating this condition we extend our convention that \( \theta(s) \) is assumed to be \( \eset \) everywhere it isn't explicitly defined and write \( \theta \in  (2^n)^{< \omega}  \) to mean that \( \theta(s) = \eset \) for co-finitely many \( s \).

\begin{lemma}\label{lem:encinv-is-defined}
\( \encinv \) is a partial computable function of \( \theta \)  which is defined for all  \( \theta \in (2^n)^{< \omega}  \).  
% Moreover, if \( \sigma \in \omega^n \) and \( \theta \in (2^n)^{\omega}   \)  then
% \[
% \lim_{s \to \infty} \frac{\card{\theta(s) \symdiff \encstr[s](\sigma)}}{n} = 0 \implies \theta \in {2^n}^{< \omega}  \land \encinv(\theta)\conv
% \]   
\end{lemma}
\begin{proof}
Clearly if \( \forall*(s)\left(\theta(s) = \eset \right)  \) then both \( \hgtenc{\theta} \) and \( \overbar{\theta} \) are defined and can be computed simply by searching for the least \( m \) such that \( \theta(s) = \eset \) for all  \( s \in [n^n\cdot m^n, n^n\cdot (m+1)^n)  \) with that \( m \) being \( \hgtenc{\theta} \).   If \( \sigma \in m^n \) then, by lemma \ref{lem:encstr-hgt}, \( \hgtenc{\encstr(\sigma)} \leq m \).   By \eqref{eq:truncate-ddist} from lemma \ref{lem:theta-equal-overbar} we see that we only need to consider a bounded number of terms to evaluate \( \ddist(\encstr(\sigma), \overbar{\theta}) \) allowing us to minimize this value by brute force.  %This establishes the main claim and the moreover claim follows immediately from this and lemma \ref{lem:encstr-to-zero}.     
\end{proof}

We now note that \( \encinv \) is truly the inverse of \( \encstr \).  

\begin{lemma}\label{lem:encinv-is-inverse}
For all \( n > 0, \sigma \in \omega^n \),  \( \encinv(\encstr(\sigma)) = \sigma \).
\end{lemma}
\begin{proof}
By lemma \ref{lem:encstr-hgt}, \( \encstr(\sigma) = \overbar{\encstr(\sigma)} \) therefore \( \ddist(\encstr(\sigma), \overbar{\encstr(\sigma)}) = 0 \).  As, by the same lemma, \( \hgtenc{\encstr(\sigma)} = \max_{i< n} \sigma(i) \) it suffices to show that no \( \sigma' \in \omega^n \) with \( \sigma' \neq \sigma \) satisfies  
\[
\ddist(\encstr(\sigma'), \overbar{\encstr(\sigma)}) = \ddist(\encstr(\sigma'), \encstr(\sigma)) = 0 
\]
But if \( \sigma \neq \sigma' \) then \( \ceil{\frac{\card{\sigma \symdiff \sigma'}}{2}} \geq 1 \) and the result follows by lemma \ref{lem:bound-encstr}.  
\end{proof}

We now prove that we lost nothing by only considering \( \sigma \) with \( \max_{i < n} \sigma(i) \leq \hgtenc{\theta} \) when computing \( \encinv(\theta) \).
\begin{lemma}\label{lem:encinv-gives-true-min}
Suppose that \( \sigma \in \omega^n, \theta \in {2^n}^{< \omega} \) then
\[
\ddist(\encstr(\encinv(\theta)), \overbar{\theta}) \leq \ddist(\encstr(\sigma), \theta) 
\]
\end{lemma}
\begin{proof}
Let \( m = \hgtenc{\theta} \), \( \sigma'(x) = \min(\sigma(x), m) \) for all \( x < n \) and \( S = \sigma \symdiff  \sigma' \).  Since definition \ref{def:encinv} chooses \( \encinv(\theta) \in m^n \) to minimize \( \ddist(\encstr(\encinv(\theta)), \overbar{\theta}) \) and \( \sigma' \in m^n \) we must have \( \ddist(\encstr(\encinv(\theta)), \overbar{\theta}) \leq \ddist(\encstr(\sigma'), \overbar{\theta})  \).  Thus, it suffices to show that  \( \ddist(\encstr(\sigma'), \overbar{\theta}) \leq \ddist(\encstr(\sigma), \theta) \).  To that end, we will argue that

\begin{equation}\label{eq::encinv-gives-true-min:agree}
\tau \in (m-1)^n \land q \in n^n \implies \encstr[\tau][q](\sigma) = \encstr[\tau][q](\sigma')
\end{equation}

So assume, for a contradiction, that \(  \encstr[\tau][q](\sigma;x) \neq \encstr[\tau][q](\sigma';x) \)  \( \max_{i < n} \tau(i) = m' < m  \).  As \( \tau(x) < \sigma(x)  \) iff \( \tau(x) < \sigma'(x) \) (consider both \( S \) and \( \setcmp{S} \)) we must have either \( \bndb(q, \sigma, x) > m'  \) or  \(  \bndb(q, \sigma', x) > m' \) but not both.  Yet if \( q^{-1}(x) \isect S = \eset \) then neither hold and if \( q^{-1}(x) \isect S \neq \eset \) then both hold.  Thus, \eqref{eq::encinv-gives-true-min:agree} holds. Applying \eqref{eq:weird-code-prop} to \eqref{eq::encinv-gives-true-min:agree} gives us
\[
s < n^n\cdot m^n \implies \encstr[s](\sigma) = \encstr[s](\sigma')
\]
Combining \eqref{eq:truncate-ddist} from lemma \ref{lem:theta-equal-overbar} with the above result yields 
\[
\ddist(\encstr(\sigma'), \overbar{\theta}) = \sup_{s < n^n\cdot m^n} \card{\encstr[s](\sigma') \symdiff \theta(s)} = \sup_{s < n^n\cdot m^n} \card{\encstr[s](\sigma) \symdiff \theta(s)} 
\]  
As \(   \ddist(\encstr(\sigma), \theta) \) is the supremum over all \( s \) of \( \card{\encstr[s](\sigma) \symdiff \theta(s)}  \) we have  \( \ddist(\encstr(\sigma'), \overbar{\theta}) \leq \ddist(\encstr(\sigma), \theta) \) which completes the proof.
\end{proof}   

We can now finally relate changes in \( \encstr(\sigma) \) to changes to \( \sigma \).

\begin{lemma}\label{lem:encinv-bounds}
Given \( n > 0, \sigma, \sigma' \in \omega^n, \theta \in (2^n)^{< \omega}\) and \( d \in \omega  \) then
\[
\ddist(\encstr(\sigma), \theta) \leq d \land \ddist(\encstr(\sigma'), \theta) \leq d \implies \card{\sigma \symdiff \sigma'} \leq 4d
\]
Moreover, 
\begin{equation}\label{eq:encinv-bounds:moreover}
 \ddist(\encstr(\sigma), \theta) \leq d \implies \card{\sigma \symdiff \encinv(\theta)} \leq 4d 
\end{equation}
\end{lemma}
\begin{proof}
Suppose, for a contradiction, that the main claim fails, i.e., both \( \ddist(\encstr(\sigma), \theta)  \)  \( \ddist(\encstr(\sigma'), \theta)  \) are at most  \( d \)  but \( \card{\sigma \symdiff \sigma'} > 4d \).  We first note that for all \( s \) we have 
\begin{equation}\label{eq:encinv-bounds:2dgeq}
 \card{\encstr[s](\sigma) \symdiff \theta(s) \union \encstr[s](\sigma') \symdiff \theta(s)} \leq 2d 
 \end{equation}
Since any value on which \( \encstr[s](\sigma) \) and \( \encstr[s](\sigma') \) disagree must be a value one of them disagrees with \( \theta(s) \) we have 
\[
\encstr[s](\sigma) \symdiff \encstr[s](\sigma') \subset  \encstr[s](\sigma) \symdiff \theta(s) \union \encstr[s](\sigma') \symdiff \theta(s)
\]
Combining this with \eqref{eq:encinv-bounds:2dgeq} yields that for all \( s \)
\begin{equation}\label{eq:encinv-bounds:2dgeq-final}
\card{\encstr[\tau][q](\sigma) \symdiff \encstr[\tau][q](\sigma')} \leq 2d
\end{equation}
As in lemma \ref{lem:encinv-bounds} let  \( S = \sigma \symdiff \sigma' \),  \( \zeta = \sigma \wwedge \sigma' \) and \( \overbar{z} \) be the median of \( \zeta[S] \).  By that lemma we have some  \( q \in n^n \) and \(  \tau \in \omega^n \) with \( \max_{i < n} \tau(i) \leq \overbar{z} \) satisfying (since \( \ceil{\frac{4d+1}{2}} \geq 2d +1 \))
\begin{equation}\label{eq:encinv-bounds:2dless}
2d < \card{\encstr[\tau][q](\sigma) \symdiff \encstr[\tau][q](\sigma')}
\end{equation}
Letting \( s = \code{\tau, q} \) yields an immediate contradiction with \eqref{eq:encinv-bounds:2dgeq-final} proving the main claim.

To verify the moreover claim, let \( \sigma' = \encinv(\theta) \) and assume that \( \ddist(\encstr(\sigma), \theta) \leq d  \).  By lemma \ref{lem:encinv-gives-true-min} we have \( \ddist(\encstr(\sigma'), \overbar{\theta}) \leq d \) and by definition \ref{def:encinv} we must have \( \max_{i < n} \sigma'(i) \leq \hgtenc{\theta} \).  We now apply the argument used above to \( \sigma, \sigma' \)  noticing that in that proof  the \( \tau \) we used satisfied \(  \max_{i < n} \tau(i) \leq \overbar{z}  \) and \( \overbar{z} \leq  \max_{i < n} \sigma'(i) \leq \hgtenc{\theta}  \).   Thus, by lemma \ref{lem:theta-equal-overbar}, at the value of \( s = \code{\tau, q} \) used in the proof above  we have \( \theta(s) = \overbar{\theta}(s) \).  Thus, we may still conclude that \( \card{\sigma \symdiff \encinv(\theta)} = \card{\sigma \symdiff \sigma'} \leq 4d  \) as desired.      
\end{proof}

\subsection{Building the Reductions}
We've now met our goals with respect to finite strings so it is now time to define \( \Gamma, \hat{\Gamma} \) and show they have the desired properties.   We first invoke lemma \ref{lem:seperate-and-repeat} with a specific choice of arguments.  %respect to a specific choice of \( n_i, b_i \) to define the intervals \( L_i, L^{-}_i \) and associated values.  

\begin{definition}\label{def:b-and-Li}
For all \( i \in \omega \) let \( n_i, s_i \) satisfy \( i = \pair{n_i}{s_i} \) and let  \( b_i = 2n_i + i + 1 \).   We define, as per  lemma \ref{lem:seperate-and-repeat}, the following values.
\begin{alignat*}{2}
   r_i & \eqdef 2b_i + n_i + l_{i-1}  &   \qquad \qquad l^{-}_i &\eqdef 2^{b_i + n_i + l_{i-1}} = 2^{r_i - b_i} \\
   l_i &\eqdef l^{-}_i  + 2^{n_i + r_i} &   l_{-1} &\eqdef 0   \\
    L_i &\eqdef [l^{-}_i, l_i)  &  L^{-}_i &\eqdef [l_{i-1}, l^{-}_i) 
\end{alignat*}
\end{definition}
Now we can define \( \Gamma \).

\begin{definition}\label{def:gamma}
Given \( f \in \baire \) define \( \xi^{f}_i = \encstr[s_i](f/I_{n_i}) \) and define \( \Gamma(f)  \)  to be the unique   \( X \in \cantor \) with \( X \isect L^{-}_i = \eset \) and \( X/L_i = \modr[r_i](\xi^{f}_i) \).
% is the result of applying lemma \ref{lem:seperate-and-repeat} to the sequence \( \xi^{f}_i = \encstr[s_i](f/I_{n_i}) \).   That is, \( X \isect L^{-}_i = \eset \) and \( X/L_i = \modr[r_i](\xi^{f}_i) \).
\end{definition}
The uniqueness is guaranteed by the fact that \( L_i, L^{-}_i \) form a partition of \( \omega \) and \mydash since  \( \lh{\xi^f_i} = 2^{n_i} \) and \( b_i \) is never \( 0 \) \mydash   the definition conforms with all the conditions required by lemma \ref{lem:seperate-and-repeat}.  We now define \( \hat{\Gamma} \). 

\begin{definition}\label{def:hatgamma}
Given \( X \in \cantor \) define \( \theta^X_n\maps{\omega}{2^{2^n}} \)  by \( \theta^{X}_n(s) = \modrinv[r_{\pair{n}{s}}](X/L_{\pair{n}{s}}) \).  Define  \( \hat{\Gamma}(X)  \) to be the function \(  f   \) such that \( f/I_n = \encinv(\theta^{X}_n)  \) and \( f(0) =0 \) . 
\end{definition}
Unlike \( \Gamma \) this functional is not necessarily totally total, i.e., there will be some \( X \in \cantor \) for which \( \hat{\Gamma}(X) \) is not total.  However, we will show that if \( X = \Gamma(f) \) and \( X' \) is a coarse description of \( X \) then \( \hat{\Gamma}(X) \) is total.         

\subsection{Verification}

We are finally in a position to verify the main theorem which we recall said the following.

\newtheorem*{thm:coarse:again}{Theorem \ref{thm:coarse}}
\begin{thm:coarse:again}
There are computable functionals \( \Gamma\maps{\baire}{\cantor}  \) and \(  \hat{\Gamma}\maps{\cantor}{\baire} \) such that for all \( f \in \baire \), \( \Gamma \) is a uniform coarse reduction of \( \Gamma(f) \) to \( f \), \(  \hat{\Gamma} \circ \Gamma = \id  \)  and \( \hat{\Gamma} \) is a uniform coarse reduction of \( f \) to \( \Gamma(f) \). 
\end{thm:coarse:again}

We split the verification up into three lemmas. 

\begin{lemma}\label{lem:gamma-is-uc}
\( \Gamma \) is a computable functional whose range is contained in \( \cantor \) total on all \( f \in \baire \) and if \( f' \) is a coarse description of \( f \) then \( \Gamma(f') \) is a coarse description of \( \Gamma(f) \).      
\end{lemma}
\begin{proof}
The fact that \( \Gamma \) is a computable functional is clear from construction.  As  we remarked above that, the definition of \( \Gamma \) satisfies the conditions of lemma \ref{lem:seperate-and-repeat} and therefore \( L_i, L^{-}_i \) form a partition of \( \omega \) which ensures that \( \Gamma(f) \in \cantor \).  Now suppose that \( \udensity(f \symdiff f') = 0 \) and \( X = \Gamma(f), X' = \Gamma(f') \).   As \( X \isect  L^{-}_i  =  X' \isect  L^{-}_i  = \eset \) and, by lemma \ref{lem:seperate-and-repeat}, \( l_{i-1} < l^{-}_i < l_i \) we have that  
\begin{equation*}
\udensity[l^{-}_{i_0}](X \symdiff X') = \sup_{i \geq i_0 } \udensity[L_i](X \symdiff X')
\end{equation*}
Thus, it is enough to show that for all \( k \) all sufficiently large \( i \) satisfy  \(  \udensity[L_i](X \symdiff X') \leq  2^{-k}  \).  Recall that \( X/L_i = \modr[r_i](\xi^{f}_i)  \),  \( X'/L_i = \modr[r_i](\xi^{f'}_i)  \) and by lemma \ref{lem:modr} part \ref{lem:modr:symdiff} \( X/L_i \symdiff X'/L_i =  \modr[r_i](\xi^{f}_i \symdiff \xi^{f'}_i)   \).  Thus, by lemma \ref{lem:seperate-and-repeat} we have  
\begin{equation}\label{eq:gamma-is-uc:Li-bound}
\udensity[L_i](X \symdiff X') <  \frac{\card{\encstr[s_i](f/I_{n_i}) \symdiff \encstr[s_i](f'/I_{n_i})}}{2^{n_i}} + \frac{1}{2^{b_i}} 
\end{equation}
Thus, if suffices to show the expression on the right is bounded by \( 2^{-k} \) for large enough \( i \). Fix some arbitrary \( k \in \omega \).   As \( \udensity(f \symdiff f') = 0 \)  lemma \ref{lem:frac-to-density} entails that there is some \( n_0 \) such that  if \( n \geq n_0 \) 
 \begin{equation}\label{eq:gamma-is-uc:fsymdiff-bound}
 \frac{\card{f \symdiff f' \isect I_n}}{\card{I_n}} =  \frac{\card{f \symdiff f' \isect I_n}}{2^n} <  \frac{1}{2^{k+1}} 
 \end{equation}
By lemma \ref{lem:encstr-to-zero} for almost all \( s \), \( \encstr[t](f/I_n) = \eset \) and \( \encstr[t](f'/I_n) = \eset \).  Let \( t_n \) be such that  \( t \geq t_n \) implies \( \encstr[t](f/I_n) = \eset \) and \( \encstr[t](f'/I_n) = \eset \) and let  \( i_0  \) be larger than \( k + 1 \) and  every element in the finite set \( \set{\pair{n}{t}}{ n \leq n_0 \land t < t_n} \). That is, we've chosen \( i_0 \) large enough to ensure that we no longer see any influence from \( f/I_n \) with \( n \leq n_0 \).  We assume, from this point on,  that \( i > i_0 \).

If \( n_i \leq n_0 \) then we must have \( s_i \geq t_{n_i} \) in which case (as \( b_i \geq i > k \) )
\[
\frac{\card{\encstr[s_i](f/I_{n_i}) \symdiff \encstr[s_i](f'/I_{n_i})}}{2^{n_i}} + \frac{1}{2^{b_i}}  \leq 0 + \frac{1}{2^i} < \frac{1}{2^{k}}
\]      
This leaves only the case where \( n_i > n_0 \) to consider.  Applying \eqref{eq:gamma-is-uc:fsymdiff-bound} in the case \( n = n_i \) and multiplying  by \( 2^{n_i} \) yields 
\[
\card{f/I_{n_i} \symdiff f'/I_{n_i}} =   \card{f \symdiff f' \isect I_{n_i}} <  2^{n_i-k -1} 
\]
So by lemma \ref{lem:bound-encstr} we have that \( \card{\encstr[s_i](f/I_{n_i}) \symdiff \encstr[s_i](f'/I_{n_i})} <  2^{n_i-k -1}  \).  Dividing by \( 2^{n_i} \) again yields
\[
\frac{\card{\encstr[s_i](f/I_{n_i}) \symdiff \encstr[s_i](f'/I_{n_i})}}{2^{n_i}}  <  \frac{1}{2^{k+1}}
\]
And as \( b_i \geq i \geq k+1 \)  we have 
\[
\frac{\card{\encstr[s_i](f/I_{n_i}) \symdiff \encstr[s_i](f'/I_{n_i})}}{2^{n_i}} + \frac{1}{2^{b_i}} < \frac{1}{2^{k+1}} + \frac{1}{2^{k+1}} = \frac{1}{2^{k}}
\]
This completes the proof.

\end{proof}

We now verify that \( \hat{\Gamma} \) is a computable functional that is defined where we need it to be.  Note that, there is no loss in generality as a result of only considering coarse descriptions of \( \Gamma(f) \) that are sets as we could simply extend  \( \hat{\Gamma} \) to \( \baire \) by  treating all non-zero values as if they were \( 1 \). 

\begin{lemma}\label{lem:gamma-hat-defined}
\( \hat{\Gamma}\maps{\cantor}{\baire} \) is a partial computable functional and if \( X = \Gamma(f) \) and \( X' \) is a coarse description of \( X \) then \( \hat{\Gamma}(X') \) is total.  Moreover, \( \theta^{X'}_n(s) = \eset \) for almost all \( s \).  
\end{lemma} 
\begin{proof}
The computability of \( \hat{\Gamma} \) follows from its definition and lemma \ref{lem:encinv-is-defined}.  Now suppose that \( f, X, X' \) are as in the statement of the lemma.  We now prove that \( \hat{\Gamma}(X') \) is total.   As \( f/I_n = \encinv(\theta^{X}_n) \) and the sets \( I_n \) are a partition of \( \omega \setminus \set{0} \),  lemma \ref{lem:encinv-is-defined} tells us it is enough to show that for each \( n \)  there are only finitely many \( s \) with  \(  \theta^{X'}_n(s) \neq \eset \).    Suppose, for a contradiction, that \( n \) witnesses  \(  \existsinf(s)\left(\theta^{X'}_n(s) \neq \eset \right) \).

We note that as \( \theta^{X'}_n(s) = \modrinv[r_{\pair{n}{s}}](X'/L_{\pair{n}{s}}) \), by lemma \ref{lem:modr} part \ref{lem:modr:modinv-times}  we must have \( \card{X'/L_{\pair{n}{s}}} > 2^{r_{\pair{n}{s}}-1} \) for infinitely many \( s \).   As \( X/L_{\pair{n}{s}} = \modr[r_{\pair{n}{s}}](\encstr[s](f/I_{n})) \) and, by lemma \ref{lem:encstr-to-zero}, we have \( \encstr[s](f/I_n) = \eset \) for almost all \( s \) we have (by lemma \ref{lem:modr} part \ref{lem:modr:times-two-r}) \(  X/L_{\pair{n}{s}} = \eset \) for almost all \( s \). 

Thus, for almost all \( s \) we have \( \card{X'/L_{\pair{n}{s}} \symdiff  X/L_{\pair{n}{s}}  } > 2^{r_{\pair{n}{s}}-1}  \) and, therefore, (also for almost all \( s \) )
\begin{align*}
\udensity[l_{\pair{n}{s}}][l_{\pair{n}{s}}](X \symdiff X') &\geq \frac{ \card{X'/L_{\pair{n}{s}} \symdiff  X/L_{\pair{n}{s}}  }}{l_{\pair{n}{s}}} \geq \frac{2^{r_{\pair{n}{s}} -1}}{l^{-}_{\pair{n}{s}}  + 2^{n + r_{\pair{n}{s}}}} =  \phantom{X} \\
&\frac{2^{r_{\pair{n}{s}} -1}}{2^{r_{\pair{n}{s}} - b_{\pair{n}{s}}}  + 2^{n + r_{\pair{n}{s}}}} = \frac{1}{2^{1- b_{\pair{n}{s}}}  + 2^{n + 1}} \geq \frac{1}{1  + 2^{n + 1}}
\end{align*}
with the last inequality following from the fact that \( b_i > 0 \).  As the final term is a positive constant (for a fixed value of \( n \)) and lemma \ref{lem:seperate-and-repeat} guarantees that (as \( l_i \) is a strictly monotonic sequence)  \( \lim_{s \to \infty} l_{\pair{n}{s}} = \infty \) we have \( \udensity(X \symdiff X') > 0 \) contradicting our assumption that \( X' \) is a coarse description of \( X \).  As our contradiction followed by assuming that the moreover failed we've proved that result as well.
\end{proof}

We now show that \( \hat{\Gamma} \) is also a coarse reduction. 

\begin{lemma}\label{lem:gamma-hat-uc-reduction}
If \( X = \Gamma(f) \) and \( X' \) is a coarse description of \( X \) then \( \hat{\Gamma}(X') \) is a coarse description of \( f \).
\end{lemma} 
\begin{proof}
Assume \( X, X', f \) are as in the statement of the lemma. By lemma \ref{lem:gamma-hat-defined} there is some \( f' \in \baire \) with \( f' = \hat{\Gamma}(X') \).   By lemma \ref{lem:frac-to-density} it is enough to show that for a given value of \( k \) there is some \( n_0 \) such that   
\begin{equation}\label{eq:gamma-hat-uc-reduction:goal}
n > n_0 \implies \frac{\card{f \symdiff f' \isect I_n}}{\card{I_n}} = \frac{\card{f \symdiff f' \isect I_n}}{2^n} \leq \frac{1}{2^k}
\end{equation}
Since \( \udensity(X \symdiff X') = 0 \) and, by lemma \ref{lem:seperate-and-repeat} \( l_i \) is strictly monotonic,   there is some \( i_0 > 0 \) such that for \( i \geq i_0  \) we have
\[
\frac{1}{2^{k+4}} > \udensity[l_{i}][l_{i}](X \symdiff X') \geq  \frac{\card{X/L_i \symdiff X'/L_i}}{l_i} = \frac{\card{\modr[r_i](\xi^f_i) \symdiff X'/L_i}}{l_i}  
\]
Assuming that \( i \geq i_0 \) from this point on, we have that \( \card{\modr[r_i](\xi^f_i) \symdiff X'/L_i} < l_i 2^{-k-4} \).  If \( i = \pair{n_i}{s_i} \)  then \( \theta^{X'}_{n_i}(s_i) = \modrinv[r_i](X'/L_i) \) so by lemma \ref{lem:modr} part \ref{lem:modr:inv-symdiff} we can infer that 
\[
 \card{\xi^f_i \symdiff  \theta^{X'}_{n_i}(s_i)} = \card{\encstr[s_i](f/I_{n_i}) \symdiff  \theta^{X'}_{n_i}(s_i)}   < \frac{l_i}{2^{k+4}} \cdot \frac{1}{2^{r_i -1}} 
\] 
Expanding out the definitions of \( l_i \) and \( r_i \) on the right-hand side we get (as we have \( n_i > 0 \))
\[
\frac{l_i}{2^{k+4}} \cdot \frac{1}{2^{r_i -1}}  = \frac{2^{r_i - b_i} + 2^{n_i + r_i} }{2^{k+3 + r_i}} = \frac{2^{n_i} + 2^{-b_i}}{2^{k+3}} \leq \frac{2^{n_i}}{2^{k+3}} + \frac{1}{2^{k+3}} < \frac{2^{n_i}}{2^{k+2}}
\]  
Now pick \( n_0 \) large enough that for all \( n \geq n_0 \) and \( s \in \omega \),   \( \pair{n}{s} > i_0 \).  Thus, if \( n > n_0 \) \mydash as we will assume for the rest of the proof \mydash for all \( s \) we have 
\[
\card{\encstr[s](f/I_{n}) \symdiff  \theta^{X'}_{n}(s)} < \frac{2^{n}}{2^{k+2}}
\]    
Thus, we have that \( \ddist(\encstr(f/I_n), \theta^{X'}_n) \leq \frac{2^{n}}{2^{k+2}} \).  Now by the moreover in lemma \ref{lem:encinv-bounds} \mydash which we can apply as we saw in the last lemma that \( \theta^{X'}_{n}(s) \) is \( \eset \) for almost all \( s \)   \mydash  we have 
\[
\card{f/I_{n} \symdiff \encinv(\theta^{X'}_{n})} = \card{f/I_{n} \symdiff f'/I_n} \leq \frac{4\cdot 2^{n}}{2^{k+2}} = \frac{2^{n}}{2^{k}}  
\]
Dividing both sides by \( 2^n \) yields the desired inequality from  \eqref{eq:gamma-hat-uc-reduction:goal}.
\end{proof}

To complete the verification of theorem \ref{thm:coarse} it suffices to observe that \(  \hat{\Gamma} \circ \Gamma = \id  \) follows by tracing out the definitions and applying lemma \ref{lem:encinv-is-inverse}.

\section{Complexity Of Coarse Equivalence}\label{sec:coarse-complexity}

\NewDocumentCommand{\Fcond}{so}{\mathbb{F}\IfBooleanT{#1}{^{*}}\IfValueT{#2}{\left(#2\right)}}
\NewDocumentCommand{\Fgeq}{o}{\leq_{\Fcond}\IfValueT{#1}{^{#1}}}
\NewDocumentCommand{\Fleq}{o}{\geq_{\Fcond}\IfValueT{#1}{^{#1}}}

Given that we saw in section \ref{sec:effective-dense-complexity} that the relation \( f \NEDgeq X \) was \( \piin{1} \) complete one might expect a similar degree of complexity for the relation \( f \NCgeq g \).  In this section, we further distinguish coarse and effective dense reducibility by proving the following theorem asserting that  \( f \NCgeq g \) \mydash and hence \( f \NCequiv g \) \mydash is actually arithmetic in \( f, g \).

\begin{theorem}\label{thm:nonu-coarse-complexity}
% There is a \( \sigmazn{4} \) formula \( \psi \) such that \[
% \forall(f, g \in \baire)\left(f \NCgeq g \iff \psi(f, g)  \right)
% \] 
There are \( \sigmazn{4} \) formulas \( \psi_U, \psi_N \) such that \[
\forall(f, g \in \baire)\left(f \NCgeq g \iff \psi_N(f, g) \land f \UCgeq g \iff \psi_U(f, g) \right)
\] 
\end{theorem}

In particular, we will show that \( f \NCgeq g \)  holds iff there is some condition which forces coarse descriptions of \( f \) to compute a coarse description of \( g \).  Obviously, such a condition must exist if we genuinely have  \( f \NCgeq g \) and we will argue that, if \( f' \) is a coarse description of \( f \), then   we can use such a condition as well as knowledge about how quickly  \( \udensity[l](f \symdiff  f') \) goes to \( 0 \) to compute a coarse description of \( g \) from \( f' \).   Thus, in contrast to  section \ref{sec:effective-dense-complexity} where we saw that non-uniform effective dense reducibility can occur as the result of many different functionals which together provide all the reductions all non-uniform coarse reductions are witnessed by a single reduction plus some information about the behavior of \( f \) on an initial segment above which \( f' \) is sufficiently close to \( f \).  More formally, if \( f \NCgeq g \) we will see that there is always some \( k \) such that for any choice of \( l \) every coarse description \( f' \) of \( f \)  with \( \udensity[l](f \symdiff f') \leq 2^{-k} \) uniformly computes a coarse description of \( g \).  This contrasts with lemma \ref{lem:ned-construction-non-uniform} for non-uniform effective dense reductions.

% Thus, while we saw in  section \ref{sec:effective-dense-complexity} that effective dense reducibility might be realized by many reductions working in different ways the only way to realize non-uniform coarse reductions (that aren't uniform) is that exhibited in \cite{Dzhafarov2017Notions}.  That is, for any non-uniform coarse reduction there is some \( k \) such that  non-uniformity stemming from uncertainty about the least \( l \) such that   \( \udensity[l](f \symdiff  f') < 2^{-k} \) for some bound \(  2^{-k} \).

\subsection{Characterizing Coarse Reducibility} 

We begin by defining a notion of genericity for coarse descriptions by pairing the notion of forcing we defined in section \ref{ssec:effective-dense-conditions} for \( 0 \) density sets with values that we can use to modify the function on that \( 0 \) density set.

\begin{definition}\label{def:f-cond}
\( (q, \gamma) \in \Fcond \) iff \( q = (\sigma, k) \in \Icond \) (recall definition \ref{def:zero-density-condition}),  \( \gamma \in (\omega \union \set{\diverge})^{< \omega}  \) and \( \dom \gamma = \set{x}{\sigma(x)\conv = 1} \).  We define \( (q, \gamma) \Fleq (q', \gamma') \) iff \( \gamma' \supfun \gamma \) and \( q' \Igeq q \).  
\end{definition}

We assume, unless explicitly stated otherwise, that every condition in \( \Fcond \) has components named \( \sigma, k, \gamma \) with subscripts, accents and superscripts to match the name of the condition, e.g., \( \hat{p}_i = (\hat{\sigma}_i, \hat{k}_i, \hat{\gamma}_i) \) and \( q' = (\sigma', k', \gamma') \) .  To avoid confusion, we will avoid using conditions distinguished only by the letter used, e.g., we won't use \( p' \) and \( q' \) for conditions at the same time.  We now define how to apply a condition like the above to a function.

\begin{definition}\label{def:fcond-replace}
For \( f \in \baire \) (or \( f \in \wstrs \)) and  \( q = (\sigma, k, \gamma) \in \Fcond \) we define  \( f[q] \) to be the string  such that  
\begin{equation}
f[q](x) = \begin{cases}
                                    \diverge & \text{if } x \nin \dom \sigma\\
                                    \gamma(x) & \text{if } \gamma(x)\conv \\
                                    f(x)    & \text{if } \sigma(x)\conv = 0 \\
                                    \end{cases}
\end{equation}
Also for \( h \in \baire \)  define \( h \Fgeq[f] q \) where \( q = (\sigma, k, \gamma) \in \Fcond \) iff \( h \supfun f[q] \) and \( \udensity[\lh{\sigma}]({ h \symdiff f }) \leq 2^{-k} \).     
\end{definition}

We use conditions from \( \Fcond \) to define a notion of forcing relative to \( f \), denoted \( \forces[f] \), using  \( \breve{f} \) as the constant symbol for the object being built by forcing.   We define \( f \)-forcing (\( \forces[f] \))  in the usual way for Cohen conditions but using \( f[q] \) as the condition for the base case.   That is, \( q = (\sigma, k, \gamma) \forces[f] \psi(\breve{f}) \) for quantifier free \( \psi \) just if \( f[q] \models \psi(\breve{f}) \)  (and the \( \breve{f} \) use of \( \psi(\breve{f}) \) is less then \( \lh{f[q]} \)) and define the inductive steps standardly.  Though, much like we did in section \ref{sec:effective-dense-complexity}, we make the modification that for \( q \) to force an existential sentence requires the witness be below \( \lh{f[q]} \).  We observe, without proof, that this notion of forcing has the usual relativized complexity with respect to \( f \).  That is, \( q \forces[f] \psi(\breve{f}) \) is \( \deltazn(f \Tplus g){n} \) if \( \psi \) is \( \sigmazn(f \Tplus g){n} \) sentence and \( \pizn(f \Tplus g){n} \) if \( \psi \) is \( \pizn(f \Tplus g){n} \).  

Of course, talking about forcing relative to \( f \) with this notion of forcing introduces the complication that there might be many conditions which produce the same result because they tell us to replace \( f(x) \) with \( f(x) \).  Therefore, we make the following definition.

\begin{definition}\label{def:f-proper-condition}
A condition  \( p \in \Fcond \) is \( f \)-proper just if \( \hat{\sigma}(x) = 1 \implies \gamma(x) \neq f(x) \). \( h \in \baire \) is \( n \)-\( f \)-generic (for \( \Fcond \)) just if for every \( \psi \in \sigmazn{n} \), there is an \( f \)-proper condition \( q \in \Fcond \) such that  \( h \Fgeq[f] q \) and either \( q \forces[f] \psi(\breve{f}) \) or  \( q \forces[f] \lnot \psi(\breve{f}) \).          
\end{definition}
With this definition in hand, we can note (though we won't need it) that \( h \) is \( n \)-\( f \)-generic for \( \Fcond \) relative to \( g \)  iff for every \( \sigmazn(f \Tplus g){n} \) set of \( f \)-proper conditions \( h \) meets or strongly avoids them.   However, for certain technical reasons, we will want to always assume that the conditions we use are \( f \)-proper.  As every condition we consider will be associated with a single function we can assume this condition is silently satisfied in the background.

% As every condition in \( \Fcond \) we use in this section will be associated with a single function  \mydash and since \( p^{f \rightarrow f'} \) is always \( f' \)-proper \mydash  we can simply assume this condition is always satisfied in the background. 

We claim that \( f \NCgeq g \) just if this fact is \( f \)-forced by some condition \( q \in \Fcond \).    

\begin{proposition}\label{prop:coarse-reduction-equivalent}
\( f \NCgeq g \) iff there is some \( f \)-proper condition  \( q_0 \in \Fcond \) and \( e \in \omega \) such that 
\begin{align*}
q_0 \forces[f] &\forall(x)\exists(s)\left(\recfnl{e}{\breve{f}}{x}\conv[s] \right) \land \udensity({\recfnl{e}{\breve{f}}{} \symdiff g})) = 0 
\end{align*}  
Moreover, \( f \UCgeq g \) iff we can take \( q_0 = (\estr, 0, \estr) \).  
\end{proposition}
As in the previous section we understand \( \udensity({\recfnl{i}{\breve{f}}{} \symdiff g})) = 0 \) to be expressed by a \( \pizn{3} \) sentence (in this case in terms of \( g, \breve{f} \)).  Thus, this proposition clearly implies theorem \ref{thm:nonu-coarse-complexity} since the assertion  that a condition \( q \) forces a   \( \pizn(f, g){3} \) sentence \( \psi \)  is \( \pizn(f, g){3} \) in \( f, g,  \) (an index for) \(  \psi \) and \( q \).  The only if direction of the main part of the proposition is clear and we note that if \( e \) is an index witnessing \( f \UCgeq g \) then the sentence above is satisfied by all \( f' \in \baire \) which are \( 3 \)-\( f \)-generic for \( \Fcond \) and therefore there can be no extension extending \( (\estr, 0, \estr)  \) forcing the existence of a counterexample and therefore we must have \( (\estr, 0, \estr)  \) force the universal statement above as desired.  Thus, to complete the proof of the proposition, we must merely show that if \( q_0 \) forces the sentence in the proposition then \( f \NCgeq g \) and that if we further have \( q_0 = (\estr, 0, \estr) \) then \( f \UCgeq g \). 

For the rest of the section, we will assume that \( q_0, e \) satisfy the statement of the proposition.

\subsection{Excellent Sequences}

The primary tool we will use to prove the above proposition will be translating between a condition \( p \) which produces a given initial sequence \( f[p] \) relative to \( f \) and a condition  \( p^{f \rightarrow f'} \) which produces the same initial sequence  \(  f'[p^{f \rightarrow f'}] = f[p] \) relative to \( f' \).   We now define this operation.  

\begin{definition}\label{def:condition-translate}
Given \( p \in \Fcond \) and \( f, f' \in \baire \) define 
\begin{align*}
\sigma^{f \rightarrow f'} &= \set{x}{f[p](x)\conv \neq f'(x)}  \\
\gamma^{f \rightarrow f'}(x) &= \begin{cases}
                            \gamma(x) & \text{if } \gamma(x)\conv \land \sigma^{f \rightarrow f'}(x)\conv = 1 = \sigma(x) \\
                            f(x) & \text{if } \gamma(x)\diverge \land x < \lh{\sigma} \land f(x) \neq f'(x) \\
                            \diverge & \text{otherwise}
                          \end{cases}\\
k^{f \rightarrow f'} &= \sup \set{k'}{0 \leq k' \leq k \land  (\sigma^{f \rightarrow f'}, k', \gamma^{f \rightarrow f'}) \in \Fcond } \\
p^{f \rightarrow f'} &= (\sigma^{f \rightarrow f'}, k, \gamma^{f \rightarrow f'})
\end{align*} 
\end{definition}  

It is evident that this translation has the desired property.  We observe two facts about this definition that we will make use of below.  First, \(\sigma^{f \rightarrow f'}  \subset \sigma \union (f\restr{\lh{\sigma}} \symdiff f'\restr{\lh{\sigma}})  \).  Second, that \( p^{f \rightarrow f'} \) is always \( f' \)-proper and if \( p \) is \( f \)-proper then we can invert the operations on \( \sigma \) and \( \gamma \).  In particular, if \( \hat{p} = p^{f \rightarrow f'} \) then   \( \hat{\sigma}^{f' \rightarrow f} = \sigma  \) and \( \hat{\gamma}^{f' \rightarrow f} = \gamma  \).  As we are tacitly assuming that all conditions we work with are proper with respect to the relevant function we will use this fact freely below.  Unfortunately, we can't guarantee the same for \( k \) in general which prevents us from claiming that if \( q \Fgeq p^{f \rightarrow f'} \) then \( q^{f' \rightarrow f} \Fgeq p \).   However, the following lemmas will situations in which something like this obtains.

\begin{lemma}\label{lem:translate-extend}
If \( f' \) is a coarse description of \( f \) and \( q_1 = (\sigma_1, k_1, \gamma_1) \in \Fcond \) then there is \( q_2= (\sigma_2, k_2, \gamma_2) \Fgeq q_1 \) with \( k_2 = k_1 \) and \( \card{\sigma_2} = \card{\sigma_1} \) such that \(p_3 =  q_2^{f \rightarrow f'} \in \Fcond \) satisfies \( k_3 = k_1 \).  
Moreover, there is some \( q_4 \Fgeq q_2 \) such that if \( q \Fgeq q_4 \) then \( q^{f \rightarrow f'} \Fgeq p_3 \) and if \(p_5 \Fgeq   q_4^{f \rightarrow f'} \)   (and hence if \( p_5 \Fgeq q^{f \rightarrow f'} \)) then \( p_5^{f' \rightarrow f} \Fgeq q_1 \).  
\end{lemma}
While the statement of this lemma is somewhat intimidating we will cite it (almost) exclusively for the proposition that every condition \( q_1 \)  we are considering for \( f \) can be extended to \( q_4 \) such that every extension of the translation  \( q^{f \rightarrow f}_4 \) of \( q_4 \) to \( f' \) translates back to an extension of \( q_1 \).  
\begin{proof}
Let \( \lh{\sigma_1} = l_1 \) and choose \( l_2 > 8l_1  \) large enough that \( \udensity[l_2]({f \symdiff f'}) \leq 2^{-k_1 - 3} \).  Let \( q_2 \Fgeq q_1 \) be the unique condition extending \( q_1 \)  with \( \lh{\sigma_2} = l_2 \) and \( \card{\sigma_2} = \card{\sigma_1} \), i.e., extend \( \sigma_1 \)  with \( 0 \)s.  Let \( p_3 = (\sigma_3, k_3, \gamma_3) = q^{f \rightarrow f'}_2  \).  We now observe that 
\begin{equation}\label{eq:translate-extend:first-forward}
\frac{\card{\sigma_3\restr{l_2}}}{l_2} \leq \frac{\card{\sigma_2}}{l_2} + \frac{\card{f \restr{l_2} \symdiff f'\restr{l_2}}}{l_2} \leq  \frac{1}{8}\frac{\card{\sigma_1}}{l_1} + 2^{-k_1 - 3} \leq   2^{-k_1 - 2}   
\end{equation}    
which is enough to show that \( (\sigma_2^{f \rightarrow f'}, k_1, \gamma_2^{f \rightarrow f'}) \in \Fcond \) and therefore that  \(p_3 =   (\sigma_2^{f \rightarrow f'}, k_1, \gamma_2^{f \rightarrow f'})\). For the moreover claim, \eqref{eq:translate-extend:first-forward}  also shows that if we take \( q_4 = (\sigma_2, k_1 + 2, \gamma_2) \)  then \( q_4 \Fgeq q_2 \).  If  \( q \Fgeq q_4 \) it is easy to see that this only adds a term of size at most \( 2^{-k_1 - 2} \) to \eqref{eq:translate-extend:first-forward} meaning that if \( q = (\sigma, k, \gamma) \) and \( \hat{p} = (\hat{\sigma}, \hat{k}, \hat{\gamma}) = q^{f \rightarrow f'} \) then \( \hat{k} \geq k_1 + 1 \) which verifies that  \( q^{f \rightarrow f'} \Fgeq p_3 \).  

Finally, assume that \( p_5 = (\sigma_5, k_5, \gamma_5) \Fgeq  q_4^{f \rightarrow f'} \) and that \( \dot{q}_5 = (\dot{\sigma}_5, \dot{k}_5, \dot{\gamma}_5) = p^{f' \rightarrow f}_5 \), \( l_4 = \lh{\sigma_4} \) and \( l_5 = \lh{\sigma_5} \).  Thus, as \( p_5 \Fgeq  q^{f \rightarrow f'}_4 \)  and \( k_4 = k_1 + 2 \), if \( l' \in [l_4, l_5) \) then   
 \[
\frac{\card{\dot{\sigma}_5\restr{l'}}}{l'} \leq \frac{\card{\sigma_5\restr{l'}}}{l'} +  \frac{\card{f \restr{l'} \symdiff f'\restr{l'}}}{l'} \leq 2^{-k_1 -1} + 2^{-k_1 -3} < 2^{-k_1}
\]   
While if \( l' \in [l_2, l_4) \) we have \( \dot{\sigma}_5\restr{l'} = \sigma_4\restr{l'} \) and as \( q_4 \Fgeq q_1 \) we automatically have \( \udensity[l_1][l_4](\dot{\sigma}_5) \leq 2^{-k_1} \).  Taken together, this is enough to verify that \( p_5^{f' \rightarrow f} \Fgeq q_1 \). 
\end{proof}

% \begin{lemma}\label{lem:back-and-forth}
% If \( f' \) is a coarse description of \( f \) and \( \hat{p}_1 \Fgeq \hat{p}_0^{f' \rightarrow f} \) then there is some \( \hat{p}_2 \Fgeq \hat{p}_1 \) such that if \( \hat{p} \Fgeq \hat{p}_2 \) then \( \hat{p}^{f \rightarrow f'} \Fgeq \hat{p}_0 \).     
% \end{lemma} 
% \begin{proof}
% Examining the proof in lemma \ref{lem:translate-extend} we see that we could have built \( q_2 \) in that lemma so that \( k_2 = k_2^{f \rightarrow f'} \),  was greater than or equal to any value.  If we take \( q_1 = \hat{p}_1 \) and apply this insight we ensure \( q_2^{f \rightarrow f'} \Fgeq \hat{p}_0 \) (on \( \gamma, \sigma \) the two operations are inverses) and then taking \( \hat{p}_2 \)  to be \( q_4 \) from lemma \ref{lem:translate-extend} ensures that if \( \hat{p} \Fgeq \hat{p}_2 \) then \( \hat{p}^{f \rightarrow f'} \Fgeq  q_2^{f \rightarrow f'} \Fgeq p_0  \).  
% \end{proof}

%We also remark that the definitions for  \( \sigma^{f \rightarrow f'}, \gamma^{f \rightarrow f'} \)  will be used independently and aren't merely   

We now define a property of a sequence \( p_i \in \Fcond \) that ensures that the images \( \recfnl{e}{f'[p_i]}{} \) approach a limit `in density,' i.e., that when \( f_0, f_1 \Fgeq[f'] p_i \) we have that \( \udensity[l_i]({\recfnl{e}{f_0}{} \symdiff \recfnl{e}{f_1}{} }) \leq 2^{-i} \).  The important thing is that we will be able to show that, given \( p_0 \), it is possible to build such a sequence computably in \( f' \) and that when \( f' \) is a coarse description of \( f \) we will demonstrate that \( \lim_{i \to \infty} \recfnl{e}{f'[p_i]}{}  \) will be a coarse description of \( g \).

% Since the question at hand is whether every coarse description of \( f \) computes some coarse description of \( g \) we will want to consider sequences of conditions that represent such computations.  The important idea here is the property of \( w \)-good, which (roughly) means that changes to \( f \) that extend the condition can't produce changes that modify more than \( 2^{-w} \) of \( \recfnl{j}{f[p]}{}  \).

\begin{definition}\label{def:paired-F-cond}
A \( f' \)-proper  condition \( p_0 \in \Fcond \) is \( (w, f') \)-good  just if for any \( p_2, p_1 \Fgeq p \) if \( \theta_i = \recfnl{e}{f[p_i]}{} \)  then
\[
 \forall(l)\left( \lh{\theta_0} \leq l <\min{\lh{\theta_1}, \lh{\theta_2}} \implies   \frac{\card{ \theta_1\restr{l} \symdiff \theta_2\restr{l}}}{l} \leq 2^{-w}  \right)
\] 
The sequence \( p_i = (\sigma_i, k_i, \gamma_i), i < n \) is an \( f' \)-excellent sequence of length \( n \) if, taking \( \theta_i \eqdef \recfnl{e}{f'[p_i]}{}  \), for all \( i < n \)
\begin{enumerate}
    % \item \( p_i  \) has the form  \(  (\sigma_i, k_i, \gamma_i) \) 
    % \item \( \theta_i \subfun \recfnl{j}{f[p_i]}{} \) 
    \item \( p_i  \Fleq  p_{i+1}  \) 
    \item \( \theta_i \subfunneq \theta_{i+1} \) and \( k_{i+1} > k_i  \) 
    \item  \( p_i \) is \( i, f' \)-good
    \item \( \forall(p \Fgeq p_0)\left(p^{f' \rightarrow f} \Fgeq q_0\right) \) 
    \item\label{def:paired-F-cond:non-trivial} \( p_i \) is \( f' \)-proper.   
\end{enumerate} 
An \( f' \)-excellent sequence is a sequence that satisfies the above for all \( i \in \omega \).  
\end{definition}
We will simply assume part \ref{def:paired-F-cond:non-trivial} of the above definition always holds since, as remarked above, we can always modify conditions to ensure that it is true.   The next lemma establishes the existence of such sequences.  Note that, the empty sequence is always a \( 0 \) length \( f' \)-excellent sequence.

\begin{lemma}\label{lem:f-prime-good-exists}
If \( f' \) is a coarse description of \( f \) and \( p_i, i < n \) is a \( f' \)-excellent sequence of length \( n \geq 0 \)  then there is some \( p_n \Fgeq p_{n-1} \) such that \( p_i, i \leq n \) is a \( f' \)-excellent sequence of length \( n+1 \).   Moreover, if \( q_0 = (\estr, 0, \estr) \) then we can take \( p_0 = q_0. \)    
\end{lemma}
If we wanted to build an \( f \)-excellent sequence its existence would follow straightforwardly by the quasi-completeness of forcing.  This lemma just verifies that since \( f' \) is a coarse description of \( f \) we can (by using lemma \ref{lem:translate-extend}) always pick our conditions \( p_i \) so that they constrain extensions to be close enough to \( f \) to act like that \( f \)-excellent sequence.
\begin{proof}
First we consider the case where \( n = 0 \).  As all conditions are \( (0, f') \)-good the only non-trivial property \( p_0 \) must posses is ensuring that if \( p \Fgeq p_0 \) then \( p^{f' \rightarrow f} \Fgeq q_0 \).  If  \( q_0 = (\estr, 0, \estr) \) this is trivially satisfied by letting \( p_0 = q_0 \) otherwise  we apply lemma \ref{lem:translate-extend} letting \( q_1 \) in that lemma be \( q_0 \) and setting \( p_0  \) equal to \(  q_4^{f \rightarrow f'} \) from that lemma.  This handles the case where \( n = 0 \) and verifies the moreover claim.  Now suppose that \( p_i, i < n \) is an \( f' \)-excellent sequence of length \( n > 0 \).
  
Let  \( \dot{p}_n \Fgeq p_{n-1} \) satisfy \( \dot{k}_n = k_{n-1} + 1 \). Using  lemma \ref{lem:translate-extend} (but with the roles of \( f, f' \) flipped) with \( q_1 = \dot{p}_n \) build \( \ddot{p}_n \Fgeq \dot{p}_n \) as \( q_4 \) from the lemma.  As \( \ddot{p}_n^{f' \rightarrow f} \Fgeq q_0 \) we can find some \( \hat{q}_0 \Fgeq \ddot{p}_n^{f' \rightarrow f} \) and some \( l_0 \)  such that \[
\hat{q}_0 \forces[f]  \udensity[l_0]({\recfnl{e}{\breve{f}}{} \symdiff g})) \leq 2^{-n -2} \land \exists(s)\left(\lh{\recfnl{e}{\breve{f}\restr{s}}{}} \geq l_0 \right)
\]
Note that, by our definition of forcing, this requires that \( \lh{\recfnl{e}{f[\hat{q}_0]}{}} \geq l_0 \).  Applying lemma \ref{lem:translate-extend} again (this time with \( q_1 = \hat{q}_0 \)) we can find \( \hat{q}_1 \Fgeq \hat{q}_0 \) such that if \( p \Fgeq \hat{q}_1^{f \rightarrow f'}  \) then \( p^{f' \rightarrow f} \Fgeq \hat{q}_0 \).  We set \( p_n = \hat{q}_1^{f \rightarrow f'}  \) and note that by our application of lemma \ref{lem:translate-extend}  to \( \dot{p}_n \) we have that \( p_n \Fgeq \dot{p}_n \Fgeq p_{n-1} \) and that \( k_n > k_{n-1} \).  As \( f'[p_n] = f[\hat{q}_1] \) we have that \( \theta_n = \recfnl{e}{f'[p_n]}{}  \) satisfies  \( \lh{\theta_n} \geq l_0 \).  

Now suppose, for a contradiction, that \( p_n \) fails to be \( (n, f') \)-good.  Then we have  \( p_a, p_b \Fgeq p_n \) with \( \theta_x = \recfnl{e}{f'[p_x]}{}, x \in \set{a,b} \), \( l_1 = \min(\lh{\theta_a}, \lh{\theta_b})  \) and some  \( l \in [\lh{\theta_n}, l_1) \subset [l_0, l_1)   \)  such that \[
\frac{\card{\theta_a\restr{l} \symdiff \theta_b\restr{l}}}{l} > 2^{-n}
\]                
However, if we now consider the conditions \( \hat{q}_x = p_x^{f' \rightarrow f}, x \in \set{a,b} \) we have, as \( p_a, p_b \Fgeq p_n \), that \( \hat{q}_a, \hat{q}_b \Fgeq \hat{q}_0 \).  As \( f'[p_x] = f[\hat{q}_x] \) we have \( \theta_x = \hat{\theta}_x \).  Now, as \( l_0 \leq l \leq l_1 \)  we must have that \[
\frac{\card{\theta_a\restr{l} \symdiff g\restr{l}}}{l} \leq 2^{-n -2} \land \frac{\card{\theta_b\restr{l} \symdiff g\restr{l}}}{l} \leq 2^{-n -2}
\]   
as if either was greater then we would have an extension of \( \hat{q}_0 \) forcing an existential witness to the failure of the universally quantified sentence which \( \hat{q}_0 \) forced.  But this implies    
\[
\frac{\card{\theta_a\restr{l} \symdiff \theta_b\restr{l}}}{l} \leq \frac{\card{\theta_a\restr{l} \symdiff g\restr{l}}}{l} + \frac{\card{\theta_b\restr{l} \symdiff g\restr{l}}}{l} \leq 2^{-n -2} + 2^{n -2} = 2^{n-1}
\] 
contradicting our assumption above.  Thus, \( p_n \) is \( (n, f') \)-good completing the proof of the lemma.
\end{proof}

As a \( f' \)-excellent sequence is just an infinite sequence \( p_i \)  whose restriction to \( i < n \) is an \( f' \)-excellent sequence of length \( n \) this proves that every finite \( f' \)-excellent sequence (including the empty sequence) can be extended to an (infinite) \( f' \)-excellent sequence.  We now relate \( f' \)-excellent sequences to \( g \). 

\begin{lemma}\label{lem:respecting-excellent-correct}
If  \( f' \) is a coarse description of \( f \) and  \( p_i \) is a \( f' \)-excellent sequence then
\[
\forall(w)\exists(l)\forall(p\Fgeq p_{w+1})\left(\udensity[l]({\recfnl{e}{f'[p]}{}  \symdiff g }) \leq 2^{-w} \right)
\] 
\end{lemma} 
\begin{proof}
Given \( w \) and \( p \Pgeq p_{w+1} \) build \( \hat{q}_0 \Fgeq p_{w+1}^{f' \rightarrow f} \)  so that any extension \( \hat{q} \Fgeq \hat{q}_0 \) satisfies \( \hat{q}^{f \rightarrow f'} \Fgeq p_{w+1} \).  Specifically, we use lemma \ref{lem:translate-extend} with \( f, f' \) switched and setting \( q_1  \)  in that lemma to be \(  p_{w+1} \) and taking \( \hat{q}_0  \) to be \( q^{f' \rightarrow f}_4 \) where \( q_4 \) is given by the lemma.  As \( p_i \) is \( f' \)-excellent we have \( \hat{q}_0 \Fgeq p_{w+1}^{f' \rightarrow f} \Fgeq q_0 \).     

Choose \( \hat{q} \Fgeq \hat{q}_0 \) and \( l \geq \lh{\theta} \) (where \( \theta = \recfnl{e}{f'[p]}{} \)) such that  \[
\hat{q} \forces[f] \udensity[l]({\recfnl{e}{\breve{f}}{} \symdiff g})) \leq  2^{-w - 1}  \land \exists(s)\left(\lh{\recfnl{e}{\breve{f}\restr{s}}{}} \geq \lh{\theta} \right)
\]        
Let \( \hat{\theta} = \recfnl{e}{f[\hat{q}]}{} \) and notice  \[
\lh{\theta} \geq l \land l' \in [l, \lh{\hat{\theta}}) \implies \frac{\card{\hat{\theta}\restr{l'} \symdiff g\restr{l'}}}{l'} \leq 2^{-w -1} 
\] 
Let \( p' = \hat{q}^{f \rightarrow f'} \Fgeq p_{w+1} \).  As \( f[\hat{q}]=f'[p'] \) we have  \( \theta' = \recfnl{e}{f'[p']}{} = \hat{\theta}  \) and therefore \[
\udensity[l]({\recfnl{e}{f'[p]}{}  \symdiff g }) \leq \sup_{l' \in [l, \lh{\hat{\theta}})}  \frac{\card{\hat{\theta}\restr{l'} \symdiff g\restr{l'}}}{l'}  + \frac{\card{\theta'\restr{l'} \symdiff \theta\restr{l'}}}{l'} \leq 2^{-w -1}  + 2^{-w -1}  = 2^{-w}
\]  
where the bound on the second term derives from the fact that \( p_{w+1} \) is \( (w+ 1, f') \)-good.     
\end{proof}

\subsection{Computing Coarse Descriptions}

In this subsection, we finally present the computation we will use to compute a coarse description of \( g \) from a coarse description \( f' \) of \( f \).  The basic idea will be that given \( p_0 \), the first element of a \( f' \)-excellent sequence, we can guess at \( p_i \).  If our guess is wrong, we will argue that we eventually discover this and can take another guess.  We need to show that as long as we are eventually correct that we can piece together a coarse description of \( g \) by using each guess to approximate \( g \) on a sufficiently long interval.  Our notion of a sufficiently long interval will be \( I_n \) as defined previously whose definition we now recall.

\begin{equation}\label{eq:In-again}
I_n = [2^{n}, 2^{n+1})
\end{equation}

We now show we can put together guesses taken on the intervals \( I_n \) to give a coarse description \( h \) of \( g \).   

\begin{lemma}\label{lem:piece-together}
Suppose that \( f' \) is a coarse description of \( f \), \( p_i \) is a \( f' \)-excellent sequence and 
\[ 
h \in \baire \land \forall(i)\forall*(k)\exists(p \Fgeq p_i)\left(h\restr{I_k} = \recfnl{e}{f'[p]}{}\restr{I_k} \right)     %\forall({x \in [l_{k-1}, l_k)})\left(h(x) = \recfnl{e}{f'[p]}{x}\conv \right) \implies \udensity({h \symdiff g}) = 0
\]
then  \( h \) is  a coarse description of \( g \). 
\end{lemma}
\begin{proof}
For this proof we let \( \delta  \) abbreviate \(  h \symdiff g \) and show that given \( w \) there is some \( l \) with \( \udensity[l]({\delta}) \leq 2^{-w} \).  By lemma \ref{lem:respecting-excellent-correct} there is some \( l_0 \) such that   
\begin{equation*}%\label{eq:piece-together:bound}
\forall(p \Fgeq p_{w+3})\udensity[l_0]({\recfnl{e}{f'[p]}{}  \symdiff g }) \leq 2^{-w -2}
\end{equation*}
Given \(l' \geq l_0 \) and \( p \Fgeq p_{w+3}\)  we have that 
\begin{equation}\label{eq:piece-together:lprime-section}
\frac{\card{\recfnl{e}{f'[p]}{}\restr{l'}  \symdiff g\restr{l'}}}{l'} \leq \udensity[l_0]({\recfnl{e}{f'[p]}{}  \symdiff g }) \leq 2^{-w -2}
\end{equation}
Let \( k_0 \) be large enough that 
\begin{equation*}%\label{eq:piece-together:k-zero}
2^{k_0} > l_0 \land \forall(k \geq k_0)\exists(p \Fgeq p_{w+3})\left(h\restr{I_k} = \recfnl{e}{f'[p]}{}\restr{I_k}  \right) 
\end{equation*}
Given \( k \geq k_0 \) let \( p \Fgeq p_{w+3} \) satisfy \( h\restr{I_k} = \recfnl{e}{f'[p]}{}\restr{I_k} \).  Substitute \( l' \) in \eqref{eq:piece-together:lprime-section} with \( 2^{k+1} \geq 2^{k_0} > l_0 \) and multiply through by \( l' = 2^{k+1} \)  to derive that 
\begin{equation}\label{eq:piece-together:delta-I-k}
2^{k-w-1}  \geq \card{\recfnl{e}{f'[p]}{}\restr{2^{k+1}}  \symdiff g\restr{2^{k+1}}} \geq \card{\recfnl{e}{f'[p]}{}\restr{I_k}  \symdiff g\restr{I_k}} = \card{\delta\restr{I_k}}
\end{equation}
We set \( k_1 = k_0 + w +2 \) and  \( l = 2^{k_1} \) and verify that  \( \udensity[l]({\delta}) \leq 2^{-w} \) as required.  Given \( l' \geq l > l_0 \) let \( n \)  be such that \( l' \in I_{k_0 + n} \) (clearly \( n \geq w+2  \)) and observe
\begin{align*}
\frac{\card{\delta\restr{l'}}}{l'} =& \frac{\card{\delta\restr{2^{k_0}}}}{l'} + \left(\sum_{i=0}^{n-1 } \frac{\card{\delta\restr{I_{k_0 + i}}}}{l'} \right) + \frac{\card{\delta\restr{[2^{k_1}, l')}}}{l'} \leq \\
\intertext{appealing to \eqref{eq:piece-together:lprime-section} for the rightmost term gives us the bound}
& \frac{2^{k_0}}{2^{k_0 + w +2}} + \left(\sum_{i=0}^{n-1 } \frac{\card{\delta\restr{I_{k_0 + i}}}}{2^{k_0 + n}} \right) + 2^{-w -2} \leq \\
\intertext{Applying \eqref{eq:piece-together:delta-I-k} gives us}
& 2^{-w -2} + 2^{-w -2} + \sum_{i=0}^{n-1 } \frac{2^{k_0 +i -w -1}}{2^{k_0 + n}} = 2^{-w -1} + 2^{-w-2}\sum_{i = 0}^{n-1} 2^{ i - n+1}  = \\
& 2^{-w -1} + 2^{-w-2}\sum_{j = 0}^{n-1} 2^{-j} \leq 2^{-w -1} + 2^{-w-1} = 2^{-w} 
\end{align*}
As \( l' \geq l \) was arbitrary we have \( \udensity[l]({h \symdiff g}) \leq 2^{-w} \) and as \( w \) was arbitrary this gives us that  \(  \udensity({h \symdiff g}) = 0\) as required. 

\end{proof}         

We now describe how to compute a coarse description of \( g \) from a coarse description \( f' \)  of \( f \) given the first term \( p_0 \)  of a \( f' \)-excellent sequence.    

\begin{lemma}\label{lem:non-uniform-coarse-reduction-exists}
There is a computable functional \( \Gamma \) such that if  \( f' \) is a coarse description of \( f \) and \( p_0 \) is an \( f' \)-excellent sequence of length \( 1 \) then \( \Gamma(p_0, f)\conv = h \) where \( h \) is a coarse description of \( g \).   
\end{lemma}
\begin{lemma}
We sketch\footnote{Initially we described the algorithm in detail but found it was less informative than this brief sketch.} the algorithm for computing \( h \) from \( f' \) and \( p_0 \).  During step \( k \) we will specify \( h\restr{I_k} \) and we always maintain a finite sequence \( p_i, i \leq n \)  which is our guess at an \( f' \)-excellent sequence.   During each step we simultaneously search for a condition \( p_{n+1} \Fgeq p_n \) which appears to extend our \( f' \)-excellent sequence (i.e. satisfies all properties from definition \ref{def:paired-F-cond} except possibly being \( (n+1, f') \)-good) a condition \( p \Fgeq p_n \) such that \( \recfnl{e}{f'[p]}{} \) is defined on all of \( I_k \) and for witnesses showing that one of our existing guesses \( p_i, 0 < i \leq n \) isn't actually \( (i, f') \)-good.   If we discover that \( p_i \) isn't actually \( (i, f') \)-good we throw out \( p_i \)  and all our guesses after it, relax our requirement on \( p \) to only extend \( p_{i-1} \)  and fall back to searching for a replacement for \( p_i \) in our \( f' \)-excellent sequence.   When we find such a \( p \) we commit to \( h\restr{I_k} = \recfnl{e}{f'[p]}{}\restr{I_k} \) and once we've also found a guess for the next element in our sequence we continue to the next step.   

As being \( (i, f') \)-good is a \( \pizn(f'){1} \) property we eventually reject any incorrect guesses.    
Since lemma \ref{lem:f-prime-good-exists} guarantees that every finite \( f' \)-excellent sequence can be extended to an infinite \( f' \)-excellent sequence our guesses at \( p_i \) clearly for an \( f' \)-excellent sequence in the limit.  Hence, \( h \) must be total since, if noting else, \( p_{2^{k+1}} \) would satisfy the conditions for \( p \), and lemma  \ref{lem:piece-together} guarantees that \( h \) is a coarse description of \( g \).   As the algorithm above clearly describes a computable process with inputs \( f' \) and \( p_0 \) this completes the proof.
\end{lemma}

This is enough to complete the proof of proposition \ref{prop:coarse-reduction-equivalent}.  If \( f' \) is a coarse description of \( f \) then by lemma \ref{lem:f-prime-good-exists} there is a \( p_0 \) satisfying the statement of lemma \ref{lem:non-uniform-coarse-reduction-exists} and therefore every coarse description of \( f \) computes a coarse description of \( g \).  Moreover, if \( q_0 = (\estr, 0, \estr) \) \mydash as we observed to be true when we have a uniform coarse reduction \mydash then, by lemma \ref{lem:f-prime-good-exists}, for all \( f' \) we can take \( p_0 = (\estr, 0, \estr) \) giving us a single functional computing a coarse description of \( g \) from any coarse description of \( f \).  As per the discussion after proposition \ref{prop:coarse-reduction-equivalent} this is enough to complete the proof of the proposition and consequently the proof of theorem \ref{thm:nonu-coarse-complexity}.

We briefly remark that the argument in \cite{Dzhafarov2017Notions} proving the existence of sets \( X \NCgeq Y \) but \( X \nUCgeq Y \) worked by ensuring that every coarse description \( X' \) of \( X \) could uniformly compute \( X \) above some finite initial segment once \( \udensity[l](X \symdiff X') \) was small enough \mydash thus the non-uniformity in some sense all derives from uncertainty as to when \( \udensity[l](X \symdiff X') \) becomes sufficiently small and the resultant uncertainty about \( X \) on that initial segment.  In some sense, our argument here proves this is the only way a non-uniform reduction can fail to be uniform as we've shown that the reduction is uniform given \( p_0 \) where \( p_0 \) encodes information both about the point at which \( \udensity[l](f \symdiff f') \) becomes small enough as well as finitely much information about \( f \).  

This points to another difference between non-uniform coarse and non-uniform effective dense reducibility.  In the case of non-uniform coarse reducibility we can identify some particular value \( k \) such that for all \( l \) there is a single reduction that works for all \( f' \) satisfying \( \udensity[l]({f' \symdiff f}) \leq 2^{-k} \).  To see this note that we can replace the initial segments of those functions \( f' \) with some fixed (in terms of \( l \)) length initial segment of \( f \) to allow them to share the same value of \( p_0 \).  In contrast, we saw in theorem \ref{thm:ned-piii-complete} that we could have a non-uniform effective dense reduction that not only failed to have this property but did so in a pretty strong sense (the witnesses \( l_k \)  of the failure are a computable sequence).

\section{Future Directions}\label{sec:future}

We  hypothesize that the method in \S\ref{sec:effective-dense-complexity} could be extended to  show that \( f \NEDequiv g \) is \( \piin{1} \) complete.  However, such an extension wouldn't be trivial since our construction relies on duplicatively encoding some of the information needed to compute \( Z(x) \) into the values of \( f(x) \) an almost all of \( U_i \) which makes it difficult to encode just \( Z(x) \) and the part of \( f \) that encodes \( Z(x) \) into some \( g(x) \) while ensuring that replacing \( g(x) \) with \( \Box  \) wouldn't make us ignorant of too large a fraction of \( f \). On the other hand, if we try to also duplicatively encode the information about \( f \) into many different values of \( Z \) the we risk giving\( f^{\Box S} \) the ability to compute \( Z(x) \) as if \( S = \eset \) by using these interdependence to recover the hidden behavior of \( f \).  Ultimately, we hypothesize that these challenges are not insurmountable but we still ask the following question.

\begin{question}\label{q:ned-complexity-equiv}
Is the relation \( f \NEDequiv g \) for \( f, g \in \baire \) also \( \piin{1} \) complete?  If not, what is its complexity?  
\end{question}

However, modifying the construction in \S\ref{sec:effective-dense-complexity} to replace \( f \) with a set faces even fundamental barriers leaving us no clear hypothesis as to the outcome of the following question.  To motivate the aspect of the following question with respect to r.e. sets we note that a small tweak to our construction (only allowing \( Z(x) \) to change value once) would render \( Z \) to be r.e. (in \( X \) for the relativized version)  and, as \( f \in \deltazn(X){2}  \) it is trivial to encode \( f \) into a computable set of density \( 1 \) and \( 0' \) into the complementary set of density \( 0 \) to render \( f \) of \( r.e. \) degree.

\begin{question}\label{q:ned-set-complexity}
What is the complexity of \( Y \NEDgeq Z \) for \( Z, Y \in \cantor \)? What about the complexity of the relation \( \REset{i} \NEDgeq \REset{j} \)?  
\end{question}

A result showing that even, when restricted to sets, non-uniform effective dense reducibility was also \( \piin{1} \) complete \mydash or even just non-arithmetic \mydash would show the notions of non-uniform effective dense and coarse reducibility remain extremely different even on the sets.  On the other hand, a result showing that non-uniform effective dense reducibility for sets is much simpler would demonstrate that, considering the degrees of functions in addition to those of sets, isn't merely adding extra degrees (as the non-equivalence result from \S\ref{sec:edd-no-sets} shows)  but actually moving to a much more complicated structure.  Almost as interesting would be to determine whether uniform effective dense reducibility remains as complicated as the non-uniform version.

% The discussion at the start of section \ref{sec:edd-no-sets} will offer a (weak) indication that the complexity of non-uniform effective dense reducibility restricted to sets might not be the same as that for functions.   
Somewhat surprisingly, even the question of the complexity of uniform effective dense reducibility remains open.  This might seem strange given that it might appear that uniformity might allow us to leverage K\"{o}nig's basis theorem via the compactness of \( \cantor \) when considering the possible sets \( S \) we use to mask \( f \).    Unfortunately, the fact that our reduction only has to behave on sets of density \( 0 \) has the effect of injecting an extra parameter (the location \( l \) at which we commit to a certain bound on \( \udensity[l](S) \)) making the question look, at least superficially, more like trying to show that a tree through \( \wstrs \) is well-founded.  Indeed, we hypothesize that the solution to the following question is that even uniform effective dense reducibility is \( \piin{1} \) complete.  

\begin{question}\label{q:ued-complexity-equiv}
What is the complexity of \( f \UEDequiv g \) for \( f, g \in \baire \)?  What about  \( Y \UEDgeq Z \) for \( Z, Y \in \cantor \)?
\end{question}

As we noted at the end of the previous section,  non-uniform effective dense reducibility seems to offer a more profound sense in which a reduction can be non-uniform than that available for coarse reducibility raises the following, somewhat vague, question.

\begin{question}
Is there an interesting notion sitting between uniform and non-uniform effective dense reducibility that weakens the notion of uniformity required in a way that somehow only doesn't count non-uniformity that arises from differences in how quickly effective dense descriptions of \( f \) approach it in density?
\end{question}

It would also be interesting to answer the question that motivated our consideration of the complexity of effective dense reducibility in the first place.

\begin{question}\label{q:ned-definable-set-element}
Suppose that \( f \NEDequiv X \) for some set  \( X  \).  Is there a set \( Y  \NEDequiv f \) with \( Y \) hyperarithmetic in \( f \)?  What about arithmetic or even \( \deltazn{2} \) in \( f \)?     
\end{question}

While we showed in theorem \ref{thm:effective-dense} that sufficiently generic functions aren't effective dense equivalent to any set Cohen generics obviously aren't the only functions with this property, e.g., we could have used a notion of genericity that imposed further conditions on \( f \) (for instance Hechler style conditions requiring \( f \) to dominate certain functions).  This raises the question of whether this class can be interestingly characterized.

\begin{question}\label{q:charachterize-equivalent}
Can the set \( \set{f \in \baire}{\exists(X \subset \omega)\left(f \NEDequiv X \right)} \) be usefully characterized?  For instance, are all such \( f \) bounded by a computable function on a set of density \( 1 \)? 
\end{question}    

Finally, there are a number of questions one might ask about the structure of the Turing degrees of representatives of effective dense and coarse equivalence classes.  We call out the question below specifically as it is plausibly related to question \ref{q:ned-definable-set-element}.

\begin{question}\label{q:least-element}
Do the degrees considered in this paper (coarse, effective dense) always contain an element of least Turing degree, e.g., is there always \( g \NCequiv f \) such that any \( g' \NCequiv f \) satisfies \( g' \Tgeq g \)?  More generally, what does the spectrum of Turing degrees of representatives of these degrees look like?
\end{question}

\printbibliography
\end{document}